              \def\version{November 21, 2025}	        	%

\documentclass[reqno,11pt]{amsart} 
\usepackage[T1]{fontenc}
\usepackage[utf8]{inputenc}
\usepackage{amsmath,amsthm} 
\usepackage[normalem]{ulem}

\usepackage[mathscr]{eucal}
\usepackage{amssymb}
\usepackage{srcltx} 
\usepackage{dsfont}
\usepackage{hyperref}
\usepackage{xcolor}
\usepackage{enumerate}
\usepackage{tikz, pgfplots}
\usepackage{caption}
\usepackage{subcaption}
\usepackage{comment}
\usepackage{lscape}
\usepackage{graphicx}
\usepackage{tabularx}
\numberwithin{equation}{section}

\def\d{{\rm d}} 
\def\e{\varepsilon} 
 
\newfam\Bbbfam 
\font\tenBbb=msbm10 
\font\sevenBbb=msbm7 
\font\fiveBbb=msbm5 
\textfont\Bbbfam=\tenBbb 
\scriptfont\Bbbfam=\sevenBbb 
\scriptscriptfont\Bbbfam=\fiveBbb

\def\2{\mathbf 2}

\def\barn1a{\bar n_{1a}}
\def\tilden3{\widetilde n_3}

\newcommand{\R}     {\mathbb{R}} 
 
\newcommand{\N}     {\mathbb{N}} 
\renewcommand{\P}   {\mathbb{P}} 
 
\newcommand{\E}     {\mathbb{E}}

\newcommand{\smfrac}[2]{\textstyle{\frac {#1}{#2}}}

\newcommand{\wt}{\widetilde}

\def\1{{\mathchoice {1\mskip-4mu\mathrm l}      
{1\mskip-4mu\mathrm l} 
{1\mskip-4.5mu\mathrm l} {1\mskip-5mu\mathrm l}}} 
 
\def\comment#1{} 
\newtheoremstyle{thm}{2ex}{2ex}{\itshape\rmfamily}{} 
{\bfseries\rmfamily}{}{1.7ex}{} 
 
\newtheoremstyle{rem}{1.3ex}{1.3ex}{\rmfamily}{} 
{\itshape\rmfamily}{}{1.5ex}{}

\renewcommand{\theequation}{\thesection.\arabic{equation}} 
 
\newtheorem{theorem}{Theorem}[section] 
\newtheorem{lemma}[theorem]{Lemma} 
\newtheorem{prop}[theorem] {Proposition}

\newtheorem{conj}[theorem] {Conjecture}

\theoremstyle{definition}

\newtheorem{remark}[theorem]{Remark}

\renewcommand{\d}{{\rm d}} 
 
\newcommand{\eps}{\varepsilon} 
\newcommand{\ups}{\upsilon}

\newcommand{\Bcal}  {{\mathcal B}}

\newcommand{\Lcal}   {{\mathcal L }}

\newcommand\numberthis{\addtocounter{equation}{1}\tag{\theequation}}
\renewcommand{\e}   {{\operatorname e }}

\definecolor{Red}{rgb}{1,0,0}

\pgfplotsset{compat=1.18}

\definecolor{gb}{rgb}{0.0, 0.36, 0.56}
\newcommand{\bl}{\color{black}}
\definecolor{lightgreen}{rgb}{0.8, 0.9, 0.1}
\definecolor{darkgreen}{rgb}{0.6, 0.8, 0.3}
\definecolor{red1}{rgb}{0.9, 0, 0.1}
\definecolor{purple1}{rgb}{0.5, 0, 0.6}
\definecolor{blue1}{rgb}{0.4, 0.6, 0.99}
\definecolor{orange1}{rgb}{0.9, 0.7, 0}

\newcommand{\oneC}{\mbox{$(\rm{Coex}_{1,3})$\bl}}
\newcommand{\twoC}{\mbox{$(\rm{Coex}_{2,3})$\bl}}
\newcommand{\oneItwo}{\mbox{$(\rm{Inv}_{1 \rightarrow 2,3})$\bl}}
\newcommand{\twoIone}{\mbox{$(\rm{Inv}_{2 \rightarrow 1,3})$\bl}}
\newcommand{\oneNItwo}{\mbox{$(\rm{Inv}_{1 \nrightarrow 2,3})$\bl}}
\newcommand{\twoNIone}{\mbox{$(\rm{Inv}_{2 \nrightarrow 1,3})$\bl}}

\begin{document} 
 
\title[Emergence of \bl microbial \bl host dormancy]{Emergence of \bl microbial \bl host dormancy \\ \bl during \bl a persistent virus epidemic}
\author[Jochen Blath and András Tóbiás]{}
\maketitle
\thispagestyle{empty}
\vspace{-0.5cm}

\centerline{\sc Jochen Blath{\footnote{Goethe-Universität Frankfurt, Institute of Mathematics, and $\mathrm{C^3S}$ -- Center for Critical Computational Studies, Robert-Mayer-Straße 10, 60325 Frankfurt am Main, Germany, {\tt blath@math.uni-frankfurt.de}}} and András Tóbiás{\footnote{\emph{Corresponding author.} Department of Computer Science and Information Theory, Faculty of Electrical Engineering and Informatics, Budapest University of Technology and Economics, Műegyetem rkp. 3., H-1111 Budapest, Hungary, and HUN-REN Alfréd Rényi Institute of Mathematics, Reáltanoda utca 13--15, H-1053 Budapest, Hungary,  {\tt tobias@cs.bme.hu}}}}
\renewcommand{\thefootnote}{}

\bigskip

\centerline{\small(\version)} 
\vspace{.5cm} 
 
\begin{quote} 
{\small {\bf Abstract:}} We study a minimal stochastic individual-based model for a microbial population challenged by a persistent (lytic) virus epidemic. We focus on the situation in which the resident microbial host population and the virus population are in stable coexistence upon arrival of a single new ``mutant'' host individual. We assume that this mutant is capable of switching to a reversible state of dormancy upon contact with  virions as a means of avoiding infection by the virus. At the same time, we assume that this new dormancy trait comes with a cost, namely a reduced individual reproduction rate. We prove that there is a non-trivial range of parameters where the mutants can nevertheless invade the resident population with with strictly positive probability (bounded away from 0) in the large population limit. Given the reduced reproductive rate, such an invasion would be impossible in the absence of either the dormancy trait or the virus epidemic. We explicitly characterize the parameter regime where this {\em emergence of a 
host dormancy trait} is possible, determine the success probability of a single invader and the typical amount of time it takes the successful mutants to reach a macroscopic population size. 
We conclude this study by an investigation of the fate of the population after the successful emergence of a dormancy trait. Heuristic arguments and simulations suggest that after successful invasion, either both host types and the virus will reach coexistence, or the mutants will drive the resident hosts to extinction while the virus will stay in the system.  
\end{quote}



\setcounter{tocdepth}{3}

\setcounter{section}{0}
\begin{comment}{
This is not visible.}
\end{comment}

\section{Introduction}

\subsection{Motivation and background}
 The notion of dormancy describes a class of strategies  -- employed in one form or another by many species -- to withstand unfavorable or stressful conditions by transitioning into a protected and reversible state of reduced metabolic activity. Having evolved numerous times throughout the tree of life \cite{LdHWB21}, potentially already \bl very early in life's history \cite{WL25}, dormancy is now ubiquitous in particular \bl in microbial communities, where it contributes to the resilience, coexistence, and diversity of populations. However, a dormancy trait typically comes with additional costs \cite{LJ11}, and the question  under which conditions it is advantageous \bl has attracted some mathematical interest in recent years, see e.g.\ \cite{MS08,DMB11,BT19,BHS21}. 

 For microbial host--virus \bl systems, {\em host dormancy} has been suggested and described as an effective defense mechanism against virus infections,  see e.g.~\bl  \cite{B15,GW15,GW18,JF19,MNM19}. In particular,  host dormancy \bl has been shown to be able to stabilize populations challenged by a persistent virus epidemic \cite{BT21}. However, the question how such a costly dormancy trait can emerge in a host population lacking this trait has not been treated in these models.
 
 
 The starting point for the present study thus is to provide a minimal individual based model for the mechanistic explanation of the emergence of a host-dormancy trait in the presence of a stable host-virus equilibrium. Based on this micro-model, we aim to carry out a stochastic and subsequently  deterministic \bl invasion analysis, consisting of the following two phases: 

\begin{itemize}
\item {\em Phase I: Arrival and potential invasion of the  new dormancy \bl trait -- stochastic phase.} Suppose we start with a single mutant invader (coming with the new dormancy trait) in a stable host--virus population. During an initial phase, which either leads to the extinction of the invader  population \bl or its growth to a  ``macroscopic'' \bl scale, the invader will be described by a stochastic birth-death process  with (competitive) interactions \bl in order to account for  random reproductive \bl fluctuations. The costs of dormancy will be  incorporated \bl by a reduced birth rate in comparison to \bl the resident host population.\smallskip
\item {\em Phase II: Macroscopic dynamics after emergence of the new trait -- deterministic phase}. Once on the macroscopic scale (that is, on the order of the  initial \bl host population size,  given by its \bl carrying capacity in equilibrium), the population dynamics of the whole system  (resident, invader and virus population) \bl can be  properly \bl approximated by a  multi-\bl dimensional, nonlinear system of ODEs. The population can  then \bl be expected to enter either fixation, extinction or coexistence regimes.
\end{itemize}

Our main goal will be to identify  non-trivial \bl parameter regimes, characterized by the trade-off between reduced reproductive rate and dormancy initiation capability of the mutant, \bl in which the invasion and subsequent emergence on the macroscopic scale of  the \bl new dormancy trait is possible, and to determine the probability of such a successful \bl invasion event. We are also interested in the time it takes for the dormancy trait to become macroscopic, and in its long-term fate (fixation, extinction or coexistence).

\begin{remark}
    As indicated by the above two-step approach, this paper follows the set-up of {\em stochastic} adaptive dynamics, as employed e.g.\ in \cite{C06}, and is in some aspects also similar to stochastic individual-based models in epidemiology. Indeed, our invasion analysis begins with a stochastic micro-model for the early invasion phase, followed by an ODE-based model for the ecological dynamics once the invader has established itself. ``Classical'' invasion analysis, based on the system of ODEs alone, would already yield important (but not all) parts of our results, in particular the form of the dormancy emergence condition. In order to make the paper accessible and interesting to readers from different backgrounds, we will subsequently comment on the classical invasion analysis, see in particular Section~\ref{sec:dyn_system_results}. 
\end{remark}

\subsection{The stochastic individual-based model for phase I} 
We begin with a minimal model at the individual (``microscopic'') \bl  level describing the reproduction, competition, and infection dynamics of host cells and virions, before passing on to the large population limit that will lead to a dynamical system. 

 Our \bl  mechanistic micro-model considers three types of individuals.
Type 1 refers to the resident microbial hosts (without dormancy trait), type 2 to the new mutant host type, now with dormancy trait,  and, finally, \bl type 3 refers  to \bl  the extra-cellular free virus particles (virions). For types 1 and 2 we distinguish several sub-types corresponding to different states of the individuals: Type 1 individuals can be either active or infected,
 denoted by \bl type 1a and 1i, respectively. Type 2 comes with three sub-types.  We denote the active, dormant, and infected type 2 individuals by type 2a, 2d, and 2i, respectively. Our type space is thus $\mathcal T:=\{ 1a,1i,2a,2d,2i,3\}$. 

\smallskip

The {\bf population model} is then given by a continuous-time Markov chain $\mathbf N=(\mathbf N(t))_{t\geq 0}$ with values in $\N_0^6$ recording the number of particles of the respective types at each time $t\ge 0$:
$$
\mathbf N=(\mathbf N(t))_{t\geq 0} = (N_{1a}(t),N_{1i}(t),N_{2a}(t),N_{2i}(t),N_3(t))_{t\geq 0}.
$$ 
We use the abbreviations $N_1(t)=N_{1a}(t)+N_{1i}(t)$ for the total type 1 population size and $N_2(t)=N_{2a}(t)+N_{2d}(t)+N_{2i}(t)$ for the total type 2 population size at time $t \geq 0$. 

\smallskip

The {\bf transition rates} are given as follows  (note that in Section~\ref{ssn-modeling-choices}, \bl we discuss  some of \bl our modelling choices and related prior work). 


\tikzstyle{1a}=[circle,draw=blue!90,fill=blue!20,thick,minimum size=5mm]
\tikzstyle{1i}=[circle,draw=blue!90,fill=purple!10,thick,minimum size=5mm]
\tikzstyle{2a}=[circle,draw=green!90,fill=green!20,thick,minimum size=5mm]
\tikzstyle{2d}=[circle,draw=green!50,fill=green!2,thick,minimum size=5mm]
\tikzstyle{2i}=[circle,draw=green!90,fill=purple!10,thick,minimum size=5mm]
\tikzstyle{D}=[rectangle,draw=black!50,fill=white!20,thick,minimum size=5mm]
\tikzstyle{3}=[circle,draw=purple!50,fill=purple!20,thick,minimum size=2mm]
\tikzstyle{X}=[circle,draw=black!20,fill=black!10,thick,minimum size=2mm]


\medskip

\begin{enumerate}[(I)]

\item {\em Reproduction and natural death}. Type 1a individuals {\em give birth} to a new 1a individual at rate $\lambda_1>0$. Type 2a individuals do this at a reduced rate $\lambda_2 \in (0,\lambda_1)$. Both types of individuals {\em die naturally} at rate $\mu_1 \in (0,\lambda_2)$, so that both populations are fit when they are on their own. \bl

\medskip

\hspace{2pt}
\begin{center}
\noindent 
\begin{tikzpicture}[node distance=2cm, semithick, shorten >=5pt, shorten <=5pt]
  \node  [1a] 	(1a_rep)	 {$1a$};
  \node  [1a] 	(1a_child_1)	[right of=1a_rep, yshift=1cm]	 {$1a$};    
  \node  [1a] 	(1a_child_2) 	[right of=1a_child_1, xshift=-1cm] 	{$1a$};
  \node  [D] 		(death) 	[below of=1a_child_1] 	{$\dagger$}; 
  	 \draw[->] (1a_rep) to node[auto, swap] {$\lambda_1$} (1a_child_1);
 	 \draw[->] (1a_rep) to node[auto, swap] {$\mu_1$} (death);  
\end{tikzpicture} \phantom{AAAAA}
\begin{tikzpicture}[node distance=2cm, semithick, shorten >=5pt, shorten <=5pt]
  \node  [2a] 	(1a_rep)	 {$2a$};
  \node  [2a] 	(1a_child_1)	[right of=1a_rep, yshift=1cm]	 {$2a$};    
  \node  [2a] 	(1a_child_2) 	[right of=1a_child_1, xshift=-1cm] 	{$2a$};
  \node  [D] 		(death) 	[below of=1a_child_1] 	{$\dagger$}; 
  	 \draw[->] (1a_rep) to node[auto, swap] {$\lambda_2$} (1a_child_1);
 	 \draw[->] (1a_rep) to node[auto, swap] {$\mu_1$} (death);  
\end{tikzpicture}
\end{center}

\medskip

\item {\em Competition.} 
Let $C>0$ denote the overall competition strength and $K>0$ the carrying capacity of the population. For any pair of individuals for which the first one is of type 1a or 2a and the second one  is \bl from $\{ 1a,1i,2a,2d,2i\}$, death by competition occurs at rate $C/K>0$, leading to the  removal \bl  of the first (active) individual. 

\medskip

\begin{center}
\noindent
\begin{tikzpicture}[node distance=2cm, semithick, shorten >=5pt, shorten <=5pt]
    \node  [1a] 	(2_comp_1)	 {$1a$};  
    \node  [X]  	(2_comp_2)	 [right of=2_comp_1, xshift=-1cm]{$X$};   
    \node  [D] 		(death) 	 [right of=2_comp_2] 	{$\dagger$}; 
    \node  [X] 	    (2_child_1)	 [right of=death, xshift=-1cm]	 {$X$}; 
    \draw[->] (2_comp_2) to node[auto, swap] {$\frac CK$} (death);
\end{tikzpicture} \phantom{AAAAA}
\begin{tikzpicture}[node distance=2cm, semithick, shorten >=5pt, shorten <=5pt]
    \node  [2a] 	(2_comp_1)	 {$2a$};  
    \node  [X]  	(2_comp_2)	 [right of=2_comp_1, xshift=-1cm]{$X$};   
    \node  [D] 		(death) 	 [right of=2_comp_2] 	{$\dagger$}; 
    \node  [X] 	    (2_child_1)	 [right of=death, xshift=-1cm]	 {$X$}; 
    \draw[->] (2_comp_2) to node[auto, swap] {$\frac CK$} (death);
\end{tikzpicture}
$$
\phantom{AAAA}X\in \big\{1a, 1i, 2a, 2i, 2d \big\}.
$$
\end{center}

\medskip
 
\item  {\em Virus contact followed by infection or dormancy.} Let $D>0$.
Then, for any pair of individuals consisting of a host-type (1a or 2a) and a virion (type 3), a {\em virus contact} occurs at rate $D/K$.  Upon virus contact, we distinguish two cases:
\medskip
\begin{itemize}
    \item If the affected host individual is of type 1a, due to its lack of a dormancy trait, it always {\em becomes infected},  that is, the type 3 individual enters the now infected cell,   producing a new type 1i particle. \bl 
    \medskip
    \item If the affected host individual is of type 2a, then the host has a chance to {\em escape the infection  by entering \bl dormancy}. Indeed, let $q \in (0,1)$. With probability $q$, the affected type 2a host individual becomes infected (i.e.\ switches to type 2i), and the type 3 virion enters the cell. However, with probability $1-q$, the affected type 2a host individual becomes dormant and the virion is repelled. 
\end{itemize}
\medskip

 

\smallskip
\begin{center}
\begin{tikzpicture}[node distance=2cm, semithick, shorten >=5pt, shorten <=5pt]
  \node  [1a] 	(1a_rep)	 {$1a$};
  \node  [3] 	(virus)	[right of=1a_rep, xshift=-1.2cm]	 {{\tiny $3$}};    
  \node  [1i] 	(infected)	[right of=virus,xshift=1cm]	 {$1i$};      
  	 \draw[->] (virus) to node[auto, swap] {\small $\frac{D}{K}$} (infected);
 \end{tikzpicture}
 \hspace{30pt} 
 \begin{tikzpicture}[node distance=2cm, semithick, shorten >=5pt, shorten <=5pt]
  \node  [2a] 	(1a_rep)	 {$2a$};
  \node  [3] 	(virus)	[right of=1a_rep, xshift=-1.2cm]	 {{\tiny $3$}};    
  \node  [2i] 	(infected)	[right of=virus,yshift=1cm, xshift=1cm]	 {$2i$};      
  	 \draw[->] (virus) to node[auto, swap] {\small $\frac{(1-q)D}{K}$} (infected);
  \node  [2d] 	(virus_2)	[below of=infected]	 {$2d$};    
  \node  [3] 		(dormant) 	[right of=virus_2, xshift=-1.2cm] 	{{\tiny $3$}}; 
  	 \draw[->] (virus) to node[auto, swap] {\small $\frac{qD}{K}$} (virus_2);	 
 \end{tikzpicture}
\end{center}
 \smallskip
 
\item {\em Consequences of infection.} 
\medskip
\begin{itemize}
    \item Type 1i (resp.\ 2i) individuals {\em recover from infection}, switching back to type 1a (resp.\ 2a) at rate $r>0$. 
    \medskip
    \item Infected (type 1i or 2i) individuals {\em die from infection via lysis} at rate $v>0$, leading to the {\em instantaneous release} of $m\in\N$ new type 3 individuals (free virions). The number $m$ is called the \emph{burst size}.
\end{itemize}

\smallskip
\begin{center}
\noindent
      \begin{tikzpicture}[node distance=2cm, ,semithick, shorten >=5pt, shorten <=5pt]
  \node  [1i]  	(infected)	 {$1i$};  
  \node  [1a] 	(recovered)	[right of=infected,yshift=1cm, xshift=0cm]	 {$1a$};    
  	 \draw[->] (infected) to node[auto, swap] {$r$} (recovered);
  \node  [D] 	 (death)	[below of=recovered]	 {$\dagger$};    
  \node  [3] 		(virus_1) 	[right of=death, xshift=-1cm,  label=below:$1$] {{\tiny $3$}}; 
  	 \draw[->] (infected) to node[auto, swap] {$v$} (death);	 
  \node  [3] 		(virus_2) 	[right of=virus_1, xshift=-1.3cm,  label=below:$2$, label=right:$\quad\cdots$] 	{{\tiny $3$}}; 
  \node  [3] 		(virus_m) 	[right of=virus_2, xshift=-0.3cm,  label=below:$m$] 	{{\tiny $3$}}; 
 \end{tikzpicture}
  \hspace{30pt}
  \begin{tikzpicture}[node distance=2cm, ,semithick, shorten >=5pt, shorten <=5pt]
  \node  [2i]  	(infected)	 {$2i$};  
  \node  [2a] 	(recovered)	[right of=infected,yshift=1cm, xshift=0cm]	 {$2a$};    
  	 \draw[->] (infected) to node[auto, swap] {$r$} (recovered);
  \node  [D] 	 (death)	[below of=recovered]	 {$\dagger$};    
  \node  [3] 		(virus_1) 	[right of=death, xshift=-1cm,  label=below:$1$] {{\tiny $3$}}; 
  	 \draw[->] (infected) to node[auto, swap] {$v$} (death);	 
  \node  [3] 		(virus_2) 	[right of=virus_1, xshift=-1.3cm,  label=below:$2$, label=right:$\quad\cdots$] 	{{\tiny $3$}}; 
  \node  [3] 		(virus_m) 	[right of=virus_2, xshift=-0.3cm,  label=below:$m$] 	{{\tiny $3$}}; 
 \end{tikzpicture}
\end{center}
\smallskip

\item {\em Exit from dormancy.} 
\medskip
\begin{itemize}
    \item Dormant (type 2d) individuals {\em resuscitate}, switching back to the active state 2a, at rate $\sigma>0$.\medskip
    \item At rate $\kappa\mu_1$, where $\kappa \geq 0$, dormant (type 2d) individuals  {\em die naturally}. \bl Typically, $\kappa<1$ to reflect the fact that death rates during dormancy should be lower than during activity.
\end{itemize}

\medskip
\begin{center}
\noindent
  \begin{tikzpicture}[node distance=2cm, semithick, shorten >=5pt, shorten <=5pt]
  \node  [2d] (dormant)	 {$2d$};  
  \node  [2a] (active)	 [right of=dormant, yshift=1cm]{$2a$};
  \node  [D] 	(death)	[below of=active]	 {$\dagger$};    
  	 \draw[->] (dormant) to node[auto, swap] {$\sigma$} (active);
  	 \draw[->] (dormant) to node[auto, swap] {$\kappa \mu_1$} (death);	 
 \end{tikzpicture} 
\end{center}
\smallskip
\item {\em Degradation of virions.} Type 3 individuals cannot reproduce outside host-cells and they  {\em degrade} (are removed) \bl at rate $\mu_3>0$.
\medskip
\begin{center}
\noindent
  \begin{tikzpicture}[node distance=2cm, ,semithick, shorten >=5pt, shorten <=5pt]
\node  [3] 	(virus) [right of=recovered,xshift=5cm,yshift=-1cm]	{{\tiny $3$}}; 
  \node  [D] 		(death) 	[right of=virus] 	{$\dagger$}; 
 	 \draw[->] (virus) to node[auto, swap] {$\mu_3$} (death);  
 \end{tikzpicture}
\end{center}
\smallskip

\end{enumerate}

The above mechanisms characterize the dynamics of the continuous-time Markov chain $\mathbf N$. Indeed, the precise transitions are given as follows. If $\mathbf N$ is currently in a state $$\widehat {\mathbf n}=(\widehat n_{1a},\widehat n_{1i},\widehat n_{2a},\widehat n_{2d},\widehat n_{2i},\widehat n_3) \in\N_0^6,$$ then  its possible new states and jump rates are given as follows: \bl
\[ (\widehat n_{1a},\widehat n_{1i},\widehat n_{2a},\widehat n_{2d},\widehat n_{2i},\widehat n_3) \to \begin{cases}
    (\widehat n_{1a}+1,\widehat n_{1i},\widehat n_{2a},\widehat n_{2d},\widehat n_{2i},\widehat n_3) &\text{ at rate } \lambda_1 \widehat n_{1a}, \\
    (\widehat n_{1a}-1,\widehat n_{1i},\widehat n_{2a},\widehat n_{2d},\widehat n_{2i},\widehat n_3) &\text{ at rate } (\mu_1+C(\widehat n_{1}+\widehat n_{2}))\widehat n_{1a}, \\
    (\widehat n_{1a}-1,\widehat n_{1i}+1,\widehat n_{2a},\widehat n_{2d},\widehat n_{2i},\widehat n_3-1) &\text{ at rate} D\widehat n_{1a}\widehat n_3, \\
    (\widehat n_{1a}+1,\widehat n_{1i}-1,\widehat n_{2a},\widehat n_{2d},\widehat n_{2i},\widehat n_3) &\text{ at rate } r \widehat n_{1i}, \\
    (\widehat n_{1a},\widehat n_{1i}-1,\widehat n_{2a},\widehat n_{2d},\widehat n_{2i},\widehat n_3+m) &\text{ at rate } v\widehat n_{1i}, \\
     (\widehat n_{1a},\widehat n_{1i},\widehat n_{2a}+1,\widehat n_{2d},\widehat n_{2i},\widehat n_3) &\text{ at rate } \lambda_2 \widehat n_{2a}, \\
     (\widehat n_{1a},\widehat n_{1i},\widehat n_{2a}-1,\widehat n_{2d},\widehat n_{2i},\widehat n_3) &\text{ at rate } (\mu_1+C(\widehat n_{1}+\widehat n_{2}))\widehat n_{2a}, \\
     (\widehat n_{1a},\widehat n_{1i},\widehat n_{2a}-1,\widehat n_{2d},\widehat n_{2i}+1,\widehat n_3-1) &\text{ at rate } (1-q)D\widehat n_{1a}\widehat n_3, \\
     (\widehat n_{1a},\widehat n_{1i},\widehat n_{2a}-1,\widehat n_{2d}+1,\widehat n_{2i}+1,\widehat n_3)   &\text{ at rate } qD\widehat n_{1a}\widehat n_3, \\
       (\widehat n_{1a},\widehat n_{1i},\widehat n_{2a}+1,\widehat n_{2d},\widehat n_{2i}-1,\widehat n_3) &\text{ at rate } r \widehat n_{2i}, \\
    (\widehat n_{1a},\widehat n_{1i},\widehat n_{2a},\widehat n_{2d},\widehat n_{2i}-1,\widehat n_3+m) &\text{ at rate } v\widehat n_{2i}, \\
      (\widehat n_{1a}+1,\widehat n_{1i}-1,\widehat n_{2a},\widehat n_{2d},\widehat n_{2i},\widehat n_3) &\text{ at rate } r \widehat n_{1i}, \\
    (\widehat n_{1a},\widehat n_{1i},\widehat n_{2a}+1,\widehat n_{2d}-1,\widehat n_{2i},\widehat n_3) &\text{ at rate } \sigma \widehat n_{2d}, \\
    (\widehat n_{1a},\widehat n_{1i},\widehat n_{2a},\widehat n_{2d}-1,\widehat n_{2i},\widehat n_3) &\text{ at rate } \kappa\mu_1 \widehat n_{2d}, \\
    (\widehat n_{1a},\widehat n_{1i},\widehat n_{2a},\widehat n_{2d},\widehat n_{2i},\widehat n_3-1) &\text{ at rate } \mu_3 \widehat n_3,
\end{cases}
\]
where we abbreviated $\widehat n_1=\widehat n_{1a}+\widehat n_{1i}$ and $\widehat n_2=\widehat n_{2a}+\widehat n_{2d}+\widehat n_{2i}$.

\subsection{Large populations: The limiting dynamical system, its equilibria and special cases}

In order to obtain a large population limit, we employ the classical rescaling by the  {\em carrying capacity~$K$}. \bl Indeed, we consider,  for each $t \ge 0$, \bl
$$
\mathbf N^K(t)=\frac{1}{K} \mathbf N(t)
$$
and 
$$
N^K_{\upsilon}(t)=\frac{1}{K} N_\upsilon(t) \quad \mbox{ for each } \quad \upsilon \in \{1,1a,1i,2,2a,2d,2i,3\}.
$$ 
According to~a standard functional law of large numbers, see e.g.\ \cite[Theorem 11.2.1, p456]{EK}, on any fixed time interval of the form $[0,T]$, $(\mathbf N^K(t))_{t\in [0,T]}$  then \bl converges as $K\to\infty$ in probability w.r.t.\ the supremum norm to the unique solution 
$$
\mathbf n=(\mathbf n(t))_{t\in [0,T]}=(n_{1a}(t),n_{1i}(t),n_{2a}(t),n_{2d}(t),n_{2i}(t),n_{3}(t))_{t\in [0,T]}
$$
 of \bl the six-dimensional ODE system 
\begin{equation}\label{6dimvirus}
\begin{cases}   
\begin{aligned}
\dot n_{1a}(t)& = n_{1a}(t) (\lambda_1-\mu_1 - C(n_{1}(t)+n_{2}(t))-D n_3(t))+ r n_{1i}(t), \\
\dot n_{1i}(t) &= D n_{3}(t) n_{1a}(t) - (r+v) n_{1i}(t), \\
\dot n_{2a}(t) &= n_{2a}(t)(\lambda_2-\mu_1 - C(n_{1}(t)+n_{2}(t))-D n_3(t))+ r n_{2i}(t)+ \sigma n_{2d}(t), \\
\dot n_{2d}(t) & = q D n_{3}(t) n_{2a}(t) - (\kappa\mu_1+\sigma)n_{2d}(t), \\
\dot n_{2i}(t) & = (1-q) D n_{3}(t) n_{2a}(t) - (r+v) n_{2i}(t), \\
\dot n_{3}(t) & = -  D n_{3}(t) n_{1a}(t) - (1-q) D n_{3}(t) n_{2a}(t) + m v (n_{1i}(t)+n_{2i}(t)) - \mu_3 n_{3}(t),
\end{aligned}
\end{cases}
\end{equation}
given convergence of the initial conditions in probability, where we abbreviate $n_{1}(t)=n_{1a}(t)+n_{1i}(t)$ and $n_2(t)=n_{2a}(t)+n_{2d}(t)+n_{2i}(t)$. 
This system describes thelarge-population limit of our {\em host--virus system with invading host dormancy trait}. 

Although our main invasion result will be formulated for the rescaled Markov chain $\mathbf N^K$, its statement requires some information about the equilibria of certain sub-systems of \eqref{6dimvirus} which we shall now briefly discuss. 

\subsubsection{The basic Lotka-Volterra sub-system}
We begin with a very simple special case. Suppose that, initially, virions, dormant and infected forms are absent from the system, i.e.\ $n_{1i}(0)=n_{2d}(0)=n_{2i}(0)=n_3(0)$. Then \eqref{6dimvirus} reduces to a classical two-dimensional {\em Lotka--Volterra system}
\begin{equation}\label{2dimLV}
\begin{cases}   
\begin{aligned}
\dot n_{1a}(t)& = n_{1a}(t) (\lambda_1-\mu_1 - C(n_{1a}(t)+n_{2a}(t)), \\
\dot n_{2a}(t) &= n_{2a}(t) (\lambda_2-\mu_1 - C(n_{1a}(t)+n_{2}a(t)).
\end{aligned}
\end{cases}
\end{equation}
Recall that we assume that $0 < \mu_1 <\lambda_2<\lambda_1$, i.e.\ both types are fit.
Then, the following \emph{equilibrium population sizes}  $(\bar n_{1a},0)$ and $(0,\bar n_{2a})$ \bl are positive: \[ \bar n_{1a}:=\frac{\lambda_1-\mu_1}{C} \qquad \text{ and } \qquad \bar n_{2a}:=\frac{\lambda_2-\mu_1}{C} < \bar n_{1a}. \numberthis\label{barn1a}  \] 
Clearly, 
$(\bar n_{1a},0,0,0,0,0)$, $(0,0,\bar n_{2a},0,0,0,0)$ are the corresponding equilibria of the full system.   Note that this system (with symmetric competition) does not allow for coexistence. \bl 


\subsubsection{The single-host-virus sub-system without dormancy}\label{sec-3dimvirus}

In~\cite{BK98}, a three-dimensional sub-system 
corresponding to $n_{2a}(0)=n_{2i}(0)=n_{2d}(0)=0$),   that is, \bl a {\em host-virus system without dormancy},  has been investigated, namely the following system  with \bl $r=0$:
\[
\begin{cases}   
\begin{aligned}
\dot n_{1a}(t) & = n_{1a}(t)\big( \lambda_1-\mu_1-C (n_{1a}(t)+n_{1i}(t))-D n_{3}(t) \big) + r n_{1i}(t),\\
\dot n_{1i}(t) & = D n_{1a}(t) n_{3}(t) -(r+v) n_{1i}(t), \\
\dot n_{3}(t) & = mv n_{1i}(t) -  D n_{1a}(t) n_{3}(t) -  \mu_3 n_{3}(t).
\end{aligned}
\numberthis\label{3dimvirus}
\end{cases}
\]
There (and for $r>0$ in \cite{BT21}) it was shown  that a strictly positive coexistence equilibrium -- representing a persistent virus epidemic -- of the form $(n_{1a}^*,n_{1i}^*,n_{3}^*)$
exists if and only if 
\[ 
mv>r+v \text{ and } \bar n_{1a} > \frac{\mu_3(r+v)}{D(mv-(r+v))} =n_{1a}^*  \qquad \oneC \numberthis\label{viruscoexcondq=0} 
\]
holds. Again, this corresponds to an equilibrium $(n_{1a}^*,n_{1i}^*,0,0,0,n_{3}^*)$ in the full system \eqref{6dimvirus}, where the first component $n_{1a}^*$ is given as above.

In the following, \bl by saying that an equilibrium of a system of ODEs is unstable (resp.\ asymptotically stable), we mean that it is \emph{hyperbolically} unstable (resp.\ \emph{hyperbolically} asymptotically stable), i.e.\ all eigenvalues of the corresponding Jacobi matrix have nonzero real parts, unless mentioned otherwise. The case of non-hyperbolic equilibria will often be excluded. Moreover, by \bl asymptotic stability we mean \bl local asymptotic stability, i.e.\ the real parts of the eigenvalues are all negative, unless mentioned otherwise.
In this sense, in~\cite{BK98} it is shown that $(\bar n_{1a},0,0)$ is unstable whenever the coexistence condition\bl~\eqref{viruscoexcondq=0} holds and asymptotically stable whenever $mv \leq r+v$ or $\bar n_{1a} < n_{1a}^*$. Moreover, if $m$ is sufficiently close to the value where $\bar n_{1a}=n_{1a}^*$ holds, but strictly above this value, then $(n_{1a}^*,n_{1i}^*,n_3^*)$ is asymptotically stable. 

\subsubsection{The single-host-virus sub-system with dormancy}

Finally, in~\cite{BT21} the sub-system describing a {\em single-host-virus population with dormancy}, corresponding to $n_{1a}(0)=n_{1i}(0)=0$, given by 
\begin{equation}\label{4dimvirus}
\begin{cases}
\begin{aligned}
\dot n_{2a}(t) & = n_{2a}(t)\big( \lambda_2-\mu_1-C {(n_{2a}(t)+n_{2i}(t)+n_{2d}(t))}-D n_{3}(t) \big)  + \sigma n_{2d}(t) + r n_{2i}(t),\\
\dot n_{2d}(t) & = q D n_{2a}(t)  n_{3}(t) - (\kappa\mu_1+\sigma) n_{2d}(t), \\
\dot n_{2i}(t) & = (1-q) D n_{2a}(t) n_{3}(t) -(r+v) n_{2i}(t), \\
\dot n_{3}(t) & = mv n_{2i}(t) - (1-q)  D n_{2a}(t) n_{3}(t) -  \mu_3 n_{3}(t),
\end{aligned}
\end{cases}
\end{equation}
is investigated. Here, the above system -- and similarly  for\bl~\eqref{6dimvirus} -- has an equilibrium of the form  $(\widetilde n_{2a},\widetilde n_{2d},\widetilde n_{2i},\widetilde n_{3})$ resp.\ \bl $(0,0,\widetilde n_{2a},\widetilde n_{2d},\widetilde n_{2i},\widetilde n_{3})$ with four positive coordinates if and only if  the coexistence condition \bl 
\[ 
mv>r+v \text{ and } \bar n_{2a} >  \frac{\mu_3 (r+v)}{(1-q)D(mv-(r+v))} =:\widetilde n_{2a} 
\qquad \twoC 
\numberthis\label{viruscoexcond2}
\]
holds, in which case the equilibrium is unique and $\widetilde n_{2a}$ is given as above. 
\bl
Again, $(\bar n_{2a},0,0,0)$ is unstable if $\bar n_{2a} > \wt n_{2a}$ and asymptotically stable if $mv \leq r+v$ or $\bar n_{2a} < \widetilde n_{2a}$, and if $m$ is larger than the value where $\bar n_{2a}=\widetilde n_{2a}$ holds but sufficiently close to it, then 
$(\widetilde n_{2a}, \widetilde n_{2d}, \widetilde n_{2i}, \widetilde n_3)$ is asymptotically stable. 
See~\cite{BT21} for details.

\medskip

Note that for large $m$, the equilibria of the systems \eqref{3dimvirus} and \eqref{4dimvirus},  
$$
(n_{1a}^*,n_{1i}^*,n_3^*) \quad \mbox{ and } \quad (\wt n_{2a}, \wt n_{2d}, \wt n_{2i}, \wt n_3),
$$ 
may or may not lose their stability. We summarize the corresponding results of~\cite{BK98,BT21} in Appendix~\ref{sec-Hopfwiederholung} below. The question of existence and uniqueness of a true coexistence equilibrium  for~\eqref{6dimvirus} will be settled in Propositions~\ref{prop-coexeq} and~\ref{prop-coexeqonly1virus} (see also Section~\ref{sec-lambda2lambda1}) and will turn out to be strongly related to the question of stability of $(n_{1a}^*,n_{1i}^*,0,0,0,n_3^*)$ and $(0,0,\wt n_{2a}, \wt n_{2d}, \wt n_{2i}, \wt n_3)$ in the full system, see Proposition~\ref{prop-fixation} below.

\subsection{Informal statement of the main result: Conditions for the emergence of dormancy. \bl}
 Assume \bl that a resident type 1 population is in stable coexistence with type 3 ( the case of a \bl { \em persistent virus epidemic}), and that at time 0, a single new active type 2 individual appears, e.g.\ via  ``mutation'' \bl  from the resident type 1 population or via  ``immigration''. Our aim is to \bl understand the fate of this new mutant:  Under which conditions (if at all) is it able to successfully invader \bl the resident population? 

Note that by successful invasion we  mean \bl  that the single newly arriving invader reaches a macroscopic population size on the order of the carrying capacity $K$ of the system. That is, for some $\beta >0$, the size of the type 2 population reaches as level  $\beta K$ with positive probability, as $K \to \infty$.  In this case, we speak of the {\em emergence of a dormancy trait}. \bl

 A persistent virus epidemic requires that \bl $(n_{1a}^*,n_{1i}^*,n_3^*)$ be a coordinatewise positive and asymptotically stable equilibrium of~\eqref{3dimvirus}. 
Recall that given all other parameters, we can always choose a burst size $m$ such that this holds. Our initial condition  $\mathbf N^K(0)$ will then \bl be such that 
$$
N^K_{\ups}(0) \approx n_{\ups}^* \quad \mbox{ for all } \quad \ups \in \{ 1a, 1i, 3 \},
$$
while  $N^K_{2a}(0)=1/K$ and $N^K_{2d}(0)=N^K_{2i}(0)=0$. \bl That is, types 1a, 1i, and 3 are close to their joint equilibrium upon arrival of the new active type 2a individual.  Under these conditions, we have the following result (for a more precise statement, including also invasion probabilities and the corresponding timescales, see Theorem \ref{theorem-2invasion}).

\begin{theorem}[Invasion of a dormancy trait -- informal version]
\label{thm:informal}
    Consider the population model  $\mathbf N^K$ with initial condition as above and the corresponding limiting \bl system \eqref{6dimvirus}. Assume that $r\kappa\mu_1<v\sigma$ and $q>0$. Then, invasion of a dormancy trait is possible whenever
    \begin{equation}
        \lambda_1 - \lambda_2 < \frac{q D n_3^*(v\sigma -r\kappa\mu_1)}{(r+v)(\kappa\mu_1+\sigma)}, 
        \qquad \twoIone
        \numberthis\label{2invades1showSimple}
    \end{equation}
    if one chooses the burst size $m$ appropriately. \bl 
\end{theorem}

Condition \eqref{2invades1showSimple1} has an intuitive interpretation: Let us consider the special case that the rate of recovery of infected individuals, $r$, and the factor scaling the natural death rate for dormant individuals, $\kappa$, are almost zero. Then, the above condition reduces to
    \begin{equation}
         \lambda_1 - \lambda_2 \lesssim q D n_3^* \numberthis\label{2invades1showSimple1}. 
    \end{equation}
In other words: In a stable virus epidemic, the fitness disadvantage of the dormancy trait  needs to be smaller than the rate at which the invading type  may escape from mortal infection into dormancy when the number of virions is at its equilibrium. 
We call this the {\em basic dormancy emergence condition}. \bl
The somewhat more involved expression in \eqref{2invades1showSimple} is  also still \bl intuitive and will be interpreted in the next Section, when the fully detailed version of the result is stated.

The proof of the theorem combines  a branching process approximation during {\em Phase I} and a dynamical system approximation during {\em Phase II}. It fits into the general paradigm of stochastic population-genetic invasion models~\cite{C06} and more  specifically \bl uses arguments from the paper~\cite{C+19} and our previous works~\cite{BT19,BT20,BT21}. 


\begin{remark}[On the  necessity of a persistent (stable) virus epidemic\bl]\label{remark-2hosts}
    It is classical that in absence of the virus epidemic, always the host (competitor) with a higher reproduction rate wins. 
    Thus, \bl 
    not realizing the full reproductive potential of the host and investing in dormancy-defense instead is a self-constraining strategy, which might be vulnerable to the invasion of selfish cheaters 
    during time periods when the virus concentration is small. 
    This is the reason why we focus on a {\em stable} virus epidemic. Moving away from this equilibrium and allowing for fluctuations may pose interesting but likely also challenging problems for future research. 
    %
\end{remark}

\subsection{After successful invasion: Outline of further results and conjectures}

We shall also be interested in the fate of the entire population after a successful invasion,  in particular \bl whether the system reaches a state of stable coexistence of all the six types or 
 whether certain types are driven to extinction. However, \bl 
due to the rather high complexity of the 
six-dimensional dynamical system~\eqref{6dimvirus},  we will unfortunately only be able to provide partial results and conjectures  supported by simulations and heuristics. \bl 

 A crucial role will be played by \bl the reverse invasion direction of type 1 against type 2  while in coexistence \bl with type 3,
where the analogue of the detailed invasion Theorem~\ref{theorem-2invasion} is Theorem~\ref{theorem-1invasion}.
Although the invasion analysis  for \bl type 1 seems  biologically less relevant, 
 it will be useful in order to \bl distinguish between the case of full six-dimensional type coexistence and fixation of type 2.


In Conjectures~\ref{conj-theorem2invades1} and~\ref{conj-theorem1invades2} we 
 claim that in the parameter regime where both types may invade, after a successful emergence of either host type also the stochastic individual-based model $\mathbf{N}^K$ will approach the six-type coexistence state with high probability, while if only one host type can invade, after a successful invasion,  this type \bl will reach fixation, leading to the extinction of the other host type (but not the virus). We also  discuss \bl simulation results supporting this conjecture.

There are choices of parameters where type 1 coexists with type 3 in absence of the other host type but type 2 does not. 
Interestingly,~\eqref{6dimvirus} can have a coordinatewise positive equilibrium in this case too. Comparing~\eqref{viruscoexcondq=0} to~\eqref{viruscoexcond2}, we see that since $\lambda_2<\lambda_1$ and $q>0$, the coexistence of type 2 with type 3 implies that type 1 also coexists with type 3. We will show that if type 1 does not coexist with type 3, then~\eqref{6dimvirus} has no coordinatewise positive equilibrium, see Proposition~\ref{prop-coexeqonly1virus} below. In Section~\ref{sec-lambda2lambda1} we will briefly comment on the biologically less relevant but mathematically equally interesting case $\lambda_2>\lambda_1$
. In this case, six-type coexistence does not occur, while it can happen that none of the host types can invade the other one  while coexisting with the virus. \bl  This scenario  will be called \bl called \emph{founder control}. 
 Here,\bl~\eqref{6dimvirus} has a coordinatewise positive equilibrium, but we expect that it is unstable.


\subsection{Organization of the paper}  In Section~\ref{sec-dynsystresultsvirus} we present our results on the dynamical system~\eqref{6dimvirus} along with their proofs. In Section~\ref{sec-modeldefresultsvirus} we state our main theorems on the invasion of either host type against the other host type plus the virus type in the individual-based model, and we sketch their proof. The complete proof will be worked out in Appendix~\ref{sec-proofdetails}. In Section~\ref{sec-afterinvasion} we provide further conjectures on our stochastic model and deterministic dynamical system, accompanied by simulations in the case of the latter, and we visualize the different parameter regimes. Finally, in Section~\ref{sec-discussion} we discuss the related literature and our modelling choices, we provide some further interpretation of our results and conjectures, and we mention some open questions and potential directions for future work.

While the proofs in Section~\ref{sec-dynsystresultsvirus} are a prerequisite for the proofs of our theorems below on the stochastic individual-based model, the following shortcut is available for readers mainly interested in the statement of the main results and conjectures. Section~\ref{sec-12invasion} (about the results on the stochastic model) can be understood without reading Section~\ref{sec-fixationproof}, and Section~\ref{sec-afterinvasion} (which in particular includes the conjectures and simulations related to the main results) is also independent of Section~\ref{sec-proofsketch}.


\section{Results on the dynamical system}\label{sec-dynsystresultsvirus}
\label{sec:dyn_system_results}
In this section we present our formal results on the behaviour of the six-dimensional dynamical system~\eqref{6dimvirus}. First, in Section~\ref{sec-dynsyst} we state Proposition~\ref{prop-fixation} on the stability of the ``one-host-type-plus-virus'' equilibria $(n_{1a}^*,n_{1i}^*,0,0,0,n_3^*)$ and $(0,0,\wt n_{2a},\wt n_{2d},\wt n_{2i},\wt n_3)$ under a local stability assumption on the corresponding equilibria $(n_{1a}^*,n_{1i}^*,n_3^*)$ and $(\wt n_{2a},\wt n_{2d},\wt n_{2i}, \wt n_3)$ in the sub-systems~\eqref{3dimvirus} and~\eqref{4dimvirus}, and Proposition~\ref{prop-coexeq} on the existence of an equilibrium of~\eqref{6dimvirus} with six positive coordinates. In Section~\ref{sec-only1virus}, we extend these results to the case when only type 1 but not type 2 can stably coexist with the virus. Finally, in Section~\ref{sec-lambda2lambda1} we briefly comment on the case $\lambda_2>\lambda_1$. We will observe that in this case, it is possible that both ``one-host-type-plus-virus'' equilibria are stable, i.e.\ each host type as a resident type is able to avoid invasion of the other host type as a mutant type. Section~\ref{sec-fixationproof} contains the proof of Proposition~\ref{prop-fixation}.  The proofs of the other assertions of Section~\ref{sec-dynsystresultsvirus} are rather standard (even if somewhat lengthy due to the high dimension of the system), therefore we postpone them until Appendix~\ref{sec-coexeq}. 

\subsection{Equilibria of the six-dimensional dynamical system}\label{sec-dynsyst}
In this section, we analyse the local and global \bl stability properties of equilibria of the  
system~\eqref{6dimvirus}.
The first result  is concerned with \bl the  {\em (local)} \bl stability of the equilibria $(n_{1a}^*,n_{1i}^*,0,0,0,n_3^*)$ and $(0,0,\wt n_{2a},\wt n_{2d}, \wt n_{2i}, \wt n_3)$, which can easily be determined  from the properties of the corresponding sub-systems. 

\begin{prop}[Stability of equilibria inherited from sub-systems\bl]\label{prop-fixation}
Assume $r\kappa\mu_1 \neq v \sigma$.
\begin{enumerate}
\item  Assume that~\eqref{viruscoexcondq=0} holds. Then the equilibrium $(n_{1a}^*,n_{1i}^*,0,0,0,n_3^*)$ of~\eqref{6dimvirus} is 
\begin{itemize}
    \item {\em  unstable} under the condition
    \[ 
\lambda_1-\lambda_2 < \frac{qDn_3^* (v\sigma-r\kappa\mu_1)}{(r+v)(\kappa\mu_1+\sigma)}, 
\qquad \twoIone
\numberthis\label{2invades1} 
\] 
\item {\em asymptotically stable} under the condition
\[ 
\lambda_1-\lambda_2 > \frac{qDn_3^* (v\sigma-r\kappa\mu_1)}{(r+v)(\kappa\mu_1+\sigma)}, \qquad \twoNIone
\numberthis\label{2doesn'tinvade1} 
\] 
if additionally $(n_{1a}^*,n_{1i}^*,n_3^*)$ is an asymptotically stable equilibrium of~\eqref{3dimvirus}.\\
\end{itemize}
\item  Assume that~\eqref{viruscoexcond2} holds. Then, the equilibrium $(0,0,\wt n_{2a}, \wt n_{2d}, \wt n_{2i}, \wt n_3)$  of~\eqref{6dimvirus} is 
\begin{itemize}
    \item {\em unstable} under the condition
     \[ 
    \lambda_1-\lambda_2 > \frac{qD \wt n_3 (v\sigma-r\kappa\mu_1)}{(r+v)(\kappa\mu_1+\sigma)}, \numberthis\label{1invades2}  \qquad \oneItwo
    \] 
    \item {\em asymptotically stable} under the condition
     \[ 
    \lambda_1-\lambda_2 < \frac{qD \wt n_3 (v\sigma-r\kappa\mu_1)}{(r+v)(\kappa\mu_1+\sigma)}, \qquad \oneNItwo\numberthis\label{1doesn'tinvade2} 
    \] 
    if additionally $(\wt n_{2a},\wt n_{2d}, \wt n_{2i}, \wt n_3)$ is an asymptotically stable equilibrium of~\eqref{4dimvirus}.
\end{itemize}
\end{enumerate}
\end{prop}


Note that the mere existence of $(n_{1a}^*,n_{1i}^*,n_3^*)$  resp.\ $(\wt n_{2a},\wt n_{2d}, \wt n_{2i}, \wt n_3)$ as a coordinatewise positive equilibrium is not sufficient  for asymptotic stability here, and \bl the additional condition of their stability is required because the sub-systems~\eqref{3dimvirus} and~\eqref{4dimvirus} exhibit Hopf bifurcations in certain parameter regimes, leading to the loss of stability of these equilibria for large $m$, see  also \bl Appendix~\ref{sec-Hopfwiederholung}. 

\begin{remark}[Link to classical invasion analysis in adaptive dynamics] 
Note that condition~\eqref{2invades1} is already our dormancy emergence condition from Theorem~ \ref{thm:informal}. An intuitive biological interpretation of \eqref{2invades1} will be provided below in terms of the stochastic individual-based model in Remark~\ref{remark-rkappamu1-vsigma}; condition~\eqref{1invades2} can be interpreted in a similar way. Essentially,~\eqref{2invades1} ensures that a single type 2 individual can invade a resident population of type 1 stably coexisting with type 3 with non-vanishing probability as $K\to\infty$, and~\eqref{1invades2} yields the same with the roles of the two host types swapped; hence the notations \oneItwo, \twoIone, and their negations \oneNItwo, \twoNIone. See Theorems~\ref{theorem-2invasion} and~\ref{theorem-1invasion} for a precise formulation of these assertions.

The derivation from the stochastic model precisely aligns itself with the classical notion of {\em invasion fitness} from adaptive dynamics for the deterministic equation (see eg.\ \cite{M+92,M+96,DD99}): Indeed, 
type 2 as a mutant type has a positive {invasion fitness}, that is, initial growth rate while the residents are equilibrium, under \eqref{2invades1}. A way to formally derive the invasion fitness from the deterministic system is to compute the largest eigenvalue of the matrix $A$ introduced in~\eqref{Adef} in the proof of Proposition~\ref{sec-fixationproof} below, which is simply the sub-matrix corresponding to type 2 in the Jacobi matrix of the dynamical system~\eqref{6dimvirus} at the equilibrium $(n_{1a}^*,n_{1i}^*,0,0,0,n_3^*)$. This real part is positive if and only if~\eqref{2invades1} holds and strictly negative if and only if the strict reverse inequality~\eqref{1invades2} holds. 

Note, however, that information such as the invasion {\em probability} of a single mutant, or the time it takes for invasion, need a refined probabilistic analysis that we provide in Section~\ref{sec-modeldefresultsvirus}.

\end{remark}
\color{black}

A proof of Proposition~\ref{prop-fixation} can be found in Section~\ref{sec-fixationproof}.
According to the proposition,
if  the invasion conditions\bl~\eqref{2invades1} and~\eqref{1invades2}  {\em both} \bl  hold, then the equilibria $(n_{1a}^*,n_{1i}^*,0,0,0,n_3^*)$ and $(0,0,\wt n_{2a},\wt n_{2d}, \wt n_{2i}, \wt n_3)$ cannot be stable even if $(n_{1a}^*,n_{1i}^*,n_3^*)$ and $(\wt n_{2a},\wt n_{2d}, \wt n_{2i}, \wt n_3)$ are stable. It turns out that in this case, a coordinatewise positive equilibrium  of the full system~\eqref{6dimvirus} \bl emerges.

\begin{prop}\label{prop-coexeq}
Assume that $r\kappa\mu_1 \neq v \sigma$ and that both~\eqref{viruscoexcondq=0} and~\eqref{viruscoexcond2} hold (so that $(n_{1a}^*,n_{1i}^*,n_{3}^*)$ and $(\wt n_{2a},\wt n_{2d},\wt n_{2i},\wt n_3)$ are coordinatewise positive equilibria of the corresponding sub-systems, which in particular implies $mv>r+v$). Then the system~\eqref{6dimvirus} has a unique coordinatewise nonzero equilibrium which we denote by 
$$
\mathbf x=(x_{1a},x_{1i},x_{2a},x_{2d},x_{2i},x_3).
$$ 
It satisfies
\[ x_3 = \frac{\lambda_2-\lambda_1}{qD}
\frac{(\kappa\mu_1+\sigma)(r+v)}{r\kappa\mu_1-v\sigma}, \numberthis\label{x3def} \]
\[ x_{1a} + (1-q)  x_{2a} = \frac{\mu_3(r+v)}{D(mv-(r+v))} = n_{1a}^* = (1-q) \wt n_{2a}, \numberthis\label{xasum} \]
and
\[ x_{1i}+x_{2i} = \frac{\mu_3 x_3}{mv-(r+v)}. \numberthis\label{xicond} \]
Moreover,
\begin{enumerate}
\item $\mathbf x$ is coordinatewise positive if and only if  the invasion conditions\bl~\eqref{1invades2} and~\eqref{2invades1} both hold,  or, \bl equivalently
\[
\widetilde n_3 < x_3 < n_3^*.
\numberthis\label{n3x3n3} \]
\item It is never true that $n_3^* < x_3 < \wt n_3$. 
\end{enumerate}

\end{prop}
The proof of Proposition~\ref{prop-coexeq} will be carried out in Appendix~\ref{sec-coexeq}. Proposition~\ref{prop-coexeqonly1virus} below determines whether a coordinatewise positive equilibrium exists in certain cases not covered by Proposition~\ref{prop-coexeq}. \color{black}

\begin{remark}[Heuristics for the emergence of a full coexistence equilibrium]\label{remark-virusproportions}
The content of part (1) of Proposition~\ref{prop-coexeq} is that for $\lambda_2<\lambda_1$, \eqref{6dimvirus} has a coordinatewise positive equilibrium if and only if  both ``one-host-type-plus virus'' equilibria are unstable (which is a strong indication of long-term coexistence of all six types, both in the stochastic and in the dynamical system).  From the perspective of the virus, condition \eqref{n3x3n3} says that \bl the invasion of type 2 reduces and the invasion of type 1 increases the  equilibrium virus population size. \bl 


Given that type 2a has a lower birth rate than type 1a,  the \twoC-condition~\eqref{viruscoexcond2} implies the \oneC-condition~\eqref{viruscoexcondq=0} even for $q=0$. Hence it is always ``easier'' for the virus to persist with type 1 in absence of type 2 than with type 2 in absence of type 1, and with \bl dormancy $(q>0)$, this effect becomes  only \bl stronger. Thus, part (2) of Proposition~\ref{prop-coexeq} does not come as a surprise. 
\end{remark}



{Note that in the context of Proposition~\ref{prop-coexeq}, the mere existence of $\mathbf x$ as a coordinatewise positive equilibrium does not require the (local) stability of $(n_{1a}^*,n_{1i}^*,n_{3}^*)$ and $(\wt n_{2a},\wt n_{2d},\wt n_{2i},\wt n_3)$.}
 Unfortunately, we \bl have no rigorous results on the stability of $\mathbf x$, but there is some evidence for its asymptotic stability from simulations (see Figure~\ref{figure-green}) and from the numerical computation of the eigenvalues of the Jacobi matrix for concrete choices of parameters\. 
. In fact, we expect that $\mathbf x$ is always asymptotically stable when $(n_{1a}^*,n_{1i}^*,n_{3}^*)$ and $(\wt n_{2a}, \wt n_{2d}, \wt n_{2i}, \wt n_3)$ are both asymptotically stable and $(n_{1a}^*,n_{1i}^*,0,0,0,n_{3}^*)$ and $(0,0,\wt n_{2a}, \wt n_{2d}, \wt n_{2i}, \wt n_3)$ are both unstable, see Conjecture~\ref{conj-fixation} below.  


\subsection{The case when only  type 1 but not type 2 can stably coexist with the virus\bl }\label{sec-only1virus} 

The following proposition  complements the results \bl of Proposition~\ref{prop-coexeq}  by considering the \bl existence vs.\ non-existence of a coordinatewise positive equilibrium of the  6-dimensional \bl system~\eqref{6dimvirus} in case only type 1  can coexist with type 3, while type 2  cannot \bl coexist with type 3.  The reverse situation (where \bl only type 2 coexists with type 3) is not possible under our standing assumption $\lambda_2<\lambda_1$.

\begin{prop}\label{prop-coexeqonly1virus}
Assume that $r \kappa\mu_1 \neq v \sigma$. 
\begin{enumerate}
\item If 
\[ 
mv \leq r+v \qquad \text{ or } \qquad  \Big( mv>r+v \quad \text{ and } \quad \bar n_{1a} < \frac{\mu_3(r+v)}{D(mv-(r+v))} =\wt n_{2a} \Big) \numberthis\label{virusnoncoexcondq=0} 
\]
holds (so that the sub-system~\eqref{3dimvirus} has no coordinatewise positive equilibrium), then~\eqref{6dimvirus} has no such equilibrium either.
\item If the \twoC-condition\bl~\eqref{viruscoexcondq=0} and
\[ 
mv \leq r+v \quad \text{ or } \quad  \Big( mv>r+v  \quad \text{ and } \quad \bar n_{2a} < \frac{\mu_3(r+v)}{(1-q)D(mv-(r+v))} =\wt n_{2a} \Big)  \numberthis\label{virusnoncoexcond2} 
\]
hold (so that $(n_{1a}^*,n_{1i}^*,n_3^*)$ is a coordinatewise positive equilibrium of~\eqref{3dimvirus} but~\eqref{4dimvirus} has no coordinatewise positive equilibrium), then \eqref{6dimvirus} again has a unique coordinatewise nonzero equilibrium $(x_{1a},x_{1i},x_{2a},x_{2d},x_{2i},x_3)$ satisfying the equations~\eqref{x3def}--\eqref{xicond}, which is coordinatewise positive if and only if
\[
0 < x_3 < n_3^*. \numberthis\label{0x3n3} \]
\end{enumerate}
\end{prop}
The proof of Proposition~\ref{prop-coexeqonly1virus} will be carried out in Appendix~\ref{sec-coexeq}, based on the proof of Proposition~\ref{prop-coexeq}. An interpretation of condition~\eqref{0x3n3} will be provided in Remark~\ref{remark-2novirus} after the presentation of our main results on the stochastic individual-based model.

\begin{remark}[Conjectured stability of the coexistence equilibrium for non-positive $\wt n_3$\bl]
If $(n_{1a}^*,n_{1i}^*,n_3^*)$ is not just a coordinatewise positive but an asymptotically stable equilibrium of the three-dimensional system~\eqref{3dimvirus}, then (for $\lambda_2<\lambda_1$) the condition $0<x_3<n_3^*$ is equivalent to the invasion condition~\eqref{2invades1} of type 2. In this case, assertion (1) of Proposition~\ref{prop-fixation} is valid and implies that $(n_{1a}^*,n_{1i}^*,0,0,0,n_3^*)$ is unstable,
which suggests in turn \bl that $\mathbf x$ is asymptotically stable, but we have no results in that direction, apart from numerical evidence on the negativity of the real part of the eigenvalues of the Jacobi matrix of~\eqref{6dimvirus} at $\mathbf x$ for certain concrete choices of parameters (cf.\ Section~\ref{sec-afterinvasion}). 
\end{remark}


\subsection{The case $\lambda_2>\lambda_1$ and variants of ``founder control''}\label{sec-lambda2lambda1}
So far we have avoided the study of the case $\lambda_2 > \lambda_1$ since this case does not seem biologically plausible \bl (why should an additional dormancy trait increase the reproduction rate?). 
However, we note that the proofs of all propositions of Section~\ref{sec-dynsystresultsvirus} (and also of Theorems~\ref{theorem-2invasion} and~\ref{theorem-1invasion} on the stochastic system) apply verbatim in the case $\lambda_2 > \lambda_1$ as well, and in fact a new kind of behaviour arises. \bl

We have seen that for $\lambda_2<\lambda_1$,  it is impossible \color{black} that neither host type can invade the other host type when coexisting with type 3. In a simpler model studied in \bl paper~\cite{BT20}, we called  the latter \bl situation \emph{founder control}. The term ``founder control'' is borrowed from spatial ecology, where it refers to a situation such that whichever population first establishes  itself \bl at a certain location excludes the other from future invasion, see e.g.\ \cite{V15}. It is intuitive that in case each host type stably coexists with the viruses in absence of the other host type, such a scenario should occur when  both non-invasion conditions\bl~\eqref{2doesn'tinvade1} and~\eqref{1doesn'tinvade2} hold simultaneously.
 \bl In this case, we have
\[  \frac{qD n_3^* (r\kappa\mu_1-v\sigma)}{(r+v)(\kappa\mu_1+\sigma)}> \lambda_2-\lambda_1 > \frac{qD \widetilde n_3(r\kappa\mu_1-v\sigma)}{(r+v)(\kappa\mu_1+\sigma)}. \numberthis\label{foundercontrol} \]
Now, if $r\kappa\mu_1-v\sigma>0$, then it follows that $\lambda_2>\lambda_1$. Nevertheless, it can be observed in the proof of Proposition~\ref{prop-coexeq} (whose arguments are also valid for $\lambda_2>\lambda_1$, see Appendix~\ref{sec-coexeq}) that if~\eqref{foundercontrol} holds, then not only $x_3$ is positive but $\mathbf x$ is coordinatewise positive. Due to the  mutual \bl impossibility of invasions, we conjecture that $\mathbf x$ is unstable in this case (and we also have some numerical evidence on this, see Section~\ref{sssn-simulations-heuristics}).   The opposite case $r\kappa\mu_1-v\sigma <0$, by  application of  the formula for $x_3$ given in~\eqref{x3def}, yields \bl $n_3^* < x_3 < \wt n_3$  which however is \bl excluded by Proposition~\ref{prop-coexeq}. 

Moreover, in Section~\ref{sec-only1virus} we argued that if type 1 stably coexists with type 3 in absence of the other host type but type 2 does not, then both invasion directions are possible (and we expect stable six-type coexistence) if and only if $0<x_3 < n_3^*$. Now, if the same condition is satisfied for $\lambda_2>\lambda_1$ (and thus necessarily $r\kappa\mu_1-v\sigma>0$), then the proof of Proposition~\ref{prop-coexeqonly1virus} still applies. This implies that
type 2 cannot invade type 1 coexisting with type 3, and $\mathbf x$ is coordinatewise positive, and we again conjecture that it is unstable. Of course, the invasion of type 1a against type 2a will be unsuccessful with high probability since $\lambda_1<\lambda_2$. Thus, we again observe founder control since both invasions are unsuccessful with high probability, but just like in the scenario described in Section~\ref{sec-only1virus}, the invasion of type 1 is to be understood in the virus-free setting.


\subsection{Proof of Proposition~\ref{prop-fixation}}\label{sec-fixationproof}  Throughout the proof of Proposition~\ref{prop-fixation}, we can and will assume that  neither~\eqref{1invades2} nor~\eqref{2invades1} hold with equality. 
\begin{proof}[Proof of Proposition~\ref{prop-fixation}]
To verify assertion (1), assume that $(n_{1a}^*,n_{1i}^*,n_3^*)$ is an asymptotically stable equilibrium of~\eqref{3dimvirus}. Let us now study the stability of $(n_{1a}^*,n_{1i}^*,0,0,0,n_{3}^*)$ for the full system~\eqref{6dimvirus}. Writing $n_1^*=n_{1a}^*+n_{1i}^*$, the Jacobi matrix $A(n_{1a}^*, n_{1i}^*,0,0,0,n_{3}^{*})$ of~\eqref{6dimvirus} at $(n_{1a}^*, n_{1i}^*,0,0,0,n_{3}^{*})$ is given by  
\begin{small}
\[ 
\begin{pmatrix} \lambda_1-\mu_1 - C n_{1a}^* - Cn_1^*-D n_3^* & r-C n_{1a}^* & -Cn_{1a}^* & - C n_{1a}^* & - C n_{1a}^* & - D n_{1a}^* \\ D n_3^* & -(r+v) & 0 & 0 & 0 & D n_{1a}^* \\ 0 & 0 & \lambda_2-\mu_1- C n_{1}^*-Dn_3^* & \sigma & r & 0 \\ 0 & 0 & q D n_3^* & -(\kappa\mu_1+\sigma)  & 0 & 0 \\ 0 & 0 & (1-q) D n_3^* & 0 & -(r+v) & 0 \\ - D n_3^* & mv & -(1-q) D n_3^* & 0 & mv & - D n_{1a}^* - \mu_3 \end{pmatrix}. 
\]
\end{small}
It is easy to see that the $3\times 3$ submatrix given by the 1st, 2nd, and 6th rows and columns of $A(n_{1a}^*, n_{1i}^*,0,0,0,n_{3}^{*})$ equals the Jacobi matrix of~\eqref{3dimvirus} at $(n_{1a}^*, n_{1i}^*,n_{3}^{*})$, which is assumed to have three eigenvalues with negative real parts.

Let us now consider the $3\times 3$ submatrix given by the 3rd, 4th, and 5th rows and columns of $A(n_{1a}^*, n_{1i}^*,0,0,0,n_{3}^{*})$, i.e.\ the matrix
\[ A:= \begin{pmatrix} \lambda_2-\mu_1-C(n_{1a}^*+n_{1i}^*)-Dn_3^* &  \sigma  & r  \\  q D n_3^*  & -\kappa\mu_1-\sigma & 0 \\ (1-q)Dn_3^*  & 0 & -r-v \end{pmatrix}. \numberthis\label{Adef} \]
Using the assumption that $(n_{1a}^*,n_{1i}^*,n_3^*)$ is a coordinatewise positive equilibrium of~\eqref{3dimvirus}, setting \bl the first equality of the system to zero implies that
\[ 
\lambda_1-\mu_1-C(n_{1a}^*+n_{1i}^*)-D n_3^* + r \frac{n_{1i}^*}{n_{1a}^*} = 0. 
\]
Consequently, we have
\[ 
\lambda_2-\mu_1-C(n_{1a}^*+n_{1i}^*)-Dn_3^* = \lambda_2-\lambda_1-r \frac{n_{1i}^*}{n_{1a}^*}.   
\]
Hence, using also the second line of~\eqref{6dimvirus}, 
\[ 
\begin{aligned} \det A &= \big( (\lambda_2-\lambda_1-r\frac{n_{1i}^*}{n_{1a}^*}) (-r-v)-r(1-q) D n_3^* \big) (-\kappa\mu_1-\sigma) + q D n_3^* (r+v) \sigma \\& = (\lambda_2-\lambda_1)(r+v)(\kappa\mu_1+\sigma) + q D n_3^*(v\sigma-r\kappa\mu_1).
 \end{aligned} 
 \]
Depending on the parameters, the determinant may be positive or negative. It is positive if and only if
\[ 
\lambda_2-\lambda_1 > \frac{q D n_3^*(r \kappa \mu_1-v\sigma)}{(r+v)(\kappa\mu_1+\sigma)} 
\]
holds, which is precisely the invasion condition \twoIone \ given by inequality~\eqref{2invades1}. 
It is easy to check that for the value $\lambda_2^*$ of $\lambda_2$ for which \eqref{2invades1} holds with an equality, the trace of $A$ is still negative. As the trace agrees with the sum of all eigenvalues, there must be at least one eigenvalue with a strictly negative real part. Further, 
\bl since  (given all the other parameters) $\det A$  as a function of $\lambda_2$ \color{black} depends linearly on $\lambda_2$, this eigenvalue must retain a negative real part throughout an interval \bl  $(\lambda_2^*-\eps,\lambda_2^*+\eps)$ for some $\eps>0$ small enough. \bl  We claim that there must in fact be two such eigenvalues. 

Indeed, the determinant changes sign from negative to positive precisely at $\lambda_2^*$. Hence (given that the determinant equals the product of the eigenvalues), the only other option that we should exclude is that one eigenvalue is negative in an open neighbourhood of $\lambda_2^*$, one positive, and the third one switches from positive to negative at $\lambda_2^*$. Now, since the off-diagonal entries of $A$ are non-negative, for sufficiently large $\varrho>0$ the matrix $A+\varrho I$ has non-negative entries (where $I$ denotes the $3\times 3$ identity matrix). Then, thanks to the Perron--Frobenius theorem, $A+\varrho I$ has a positive real eigenvalue $\widehat \lambda$  whose absolute value equals the spectral radius of the matrix, i.e.\ this absolute value is maximal among all eigenvalues. It follows that the eigenvalue of $A$ with the largest real part is also real, namely it is $\widehat \lambda-\varrho$. This excludes the case that there are values of $\lambda_2$ such that $A$ has two complex conjugate eigenvalues with positive real parts and the third eigenvalue is a negative real number. 
Now, assume for a contradiction that for some value $\lambda_2$ below $\lambda_2^*$ the matrix $A$ has two positive eigenvalues and one negative one. Then, for such $\lambda_2$, the two eigenvalues with positive real parts must always be real, and they cannot change sign because the determinant cannot vanish anywhere else but at the critical point $\lambda_2^*$. This is however a problem because for $\lambda_2$ very small, e.g.\ for $\lambda_2<\mu_1$, the matrix $A$ must have three eigenvalues with negative real parts.


Consequently, when the strict reverse inequality of \eqref{2invades1} holds, all eigenvalues of $A$ have a negative real part. Therefore, the determinant of $A$ is positive (i.e.\ \eqref{2invades1} holds) if and only if $A$ has a positive eigenvalue. 
Moreover, the determinant of $A$ is negative, i.e.\ 
\[ 
\lambda_2-\lambda_1 < \frac{q D n_3^*(r \kappa \mu_1-v\sigma)}{(r+v)(\kappa\mu_1+\sigma)}, 
\]
and the reverse invasion condition \oneItwo\ of~\eqref{1invades2} holds, if and only if all eigenvalues of $A$ have negative real parts.


Given this, it suffices to show that the characteristic polynomial of $A(n_{1a}^*, n_{1i}^*,0,0,0,n_{3}^{*})$ equals the characteristic polynomial of the block diagonal matrix given by these two $3 \times 3$ block given by the 1st, 2nd, and 6th rows and columns of $A(n_{1a}^*, n_{1i}^*,0,0,0,n_{3}^{*})$ and the matrix $A$. But the latter assertion is clear from Laplace's expansion theorem. Hence, we have proven assertion (1).

To prove assertion (2), assume that $(\wt n_{2a},\wt n_{2d},\wt n_{2i},\wt n_3)$ is an asymptotically stable equilibrium of~\eqref{4dimvirus}. Let us now analyse the stability of $(0,0,\wt n_{2a},\wt n_{2d},\wt n_{2i},\wt n_3)$ for the full system~\eqref{6dimvirus}. Writing $\wt n_{2} = \wt n_{2a} + \wt n_{2d} + \wt n_{2i}$, the Jacobi matrix $A(0,0,\wt n_{2a}, \wt n_{2d}, \wt n_{2i}, \wt n_3)$ of~\eqref{6dimvirus} at $(0,0,\wt n_{2a}, \wt n_{2d}, \wt n_{2i}, \wt n_3)$ is given by 
\begin{footnotesize}
\begin{align*}
\begin{pmatrix} \lambda_1-\mu_1 - C \wt n_2 - D \wt n_3 & r & 0 & 0 & 0 & 0 \\ D \wt n_3 & -(r+v) & 0 & 0 & 0 & 0 \\ - C \wt n_{2a} & -C\wt n_{2a} & \lambda_2-\mu_1-C \wt n_2 - C \wt n_{2a} - D \wt n_3 & \sigma - C \wt n_{2a} & r-C \wt n_{2a} & - D \wt n_{2a} \\ 0 & 0 & q D \wt n_{3} & -(\kappa\mu_1+\sigma) & 0 & q D \wt n_{2a} \\ 0 & 0 & (1-q)D\wt n_3 & 0 & -(r+v) & (1-q) D \wt n_{2a} \\ -D\wt n_3 & mv & -(1-q) D \wt n_3 & 0 & mv & -(1-q)D\wt n_{2a} - \mu_3 \end{pmatrix}. 
\end{align*}
\end{footnotesize}
The last $4 \times 4$ block of this matrix (given by the 3rd to 6th rows and columns) is the Jacobi matrix of~\eqref{4dimvirus} at $(\wt n_{2a}, \wt n_{2d}, \wt n_{2i}, \wt n_3)$, which is assumed to have four eigenvalues with negative real parts. Further, the first $2 \times 2$ block (given by the first two rows and columns) reads
\[ F: = \begin{pmatrix}  \lambda_1-\mu_1 - C \wt n_2 - D \wt n_3 & r \\ D \wt n_3 & -(r+v) \end{pmatrix} . \numberthis\label{Fdef} \]
Since $(\widetilde n_{2a},\widetilde n_{2d},\widetilde n_{2i},\tilden3)$ is assumed to be an equilibrium of \eqref{4dimvirus} with four positive coordinates, the first entry of the first row of $F$ equals
\[  \lambda_1-\mu_1-C(\widetilde n_{2a}+\widetilde n_{2d}+\widetilde n_{2i})-D\tilden3  = \lambda_1-\lambda_2 -\sigma \frac{\widetilde n_{2d}}{\widetilde n_{2a}} - r \frac{\widetilde n_{2i}}{\widetilde n_{2a}}, \]
while we have
\[ \widetilde n_{2d} = \frac{qD\widetilde n_{2a}\tilden3}{\kappa\mu_1+\sigma}, \qquad \widetilde n_{2i} = \frac{(1-q) D\widetilde n_{2a}\tilden3}{r+v} , \]
so that we obtain
\[ 
\det F = ( \lambda_2-\lambda_1) (r+v) + D\tilden3 \Big( \frac{q\sigma(r+v)}{\kappa\mu_1+\sigma} + (1-q)r - r \Big) = (\lambda_2-\lambda_1)(r+v) + q D \tilden3 \frac{v\sigma - r \kappa\mu_1}{\kappa\mu_1+\sigma}. 
\]
Hence, $F$ has a positive (real) eigenvalue if and only if
\[ 
\lambda_2-\lambda_1 < \frac{q D\tilden3(r\kappa\mu_1-v\sigma)}{(r+v)(\kappa\mu_1+\sigma)},  
\]
which is precisely the invasion condition \oneItwo \ given by inequality~\eqref{1invades2}. 
Indeed, since (given all other parameters) $\det F$ as a function of $\lambda_2$ \color{black} depends linearly on $\lambda_2$, it has a unique zero locus. Similarly to the case of the matrix $A$ in Section~\ref{sec-2invasion}, the eigenvalue of $F$ with the largest real part is always real thanks to the Perron--Frobenius theorem, and thus both eigenvalues have to be real. Now, if the strict reverse inequality of~\eqref{1invades2} holds, then the determinant is positive, which  by definition \color{black} implies that the first entry of the first row of $F$ is negative. Hence, the trace of $F$ is negative, and since the trace equals the sum of the eigenvalues, it follows that both eigenvalues must be negative. Consequently, when~\eqref{1invades2} holds, then one eigenvalue (namely $\wt \lambda$) is positive and the other one is negative.

Therefore, assertion (2) holds due to the fact that the characteristic polynomial of $
A(0,0,\wt n_{2a}, \wt n_{2d}, \wt n_{2i}, \wt n_3)$ equals the characteristic polynomial of the block diagonal matrix given by the aforementioned $4\times 4$ block and the aforementioned $2\times 2$ block and $F$ zeroes everywhere else, which again follows from Laplace's expansion theorem.
\end{proof}

\section{Main results on the stochastic individual-based model}\label{sec-modeldefresultsvirus} 
In Section~\ref{sec-12invasion}, we state and interpret Theorem~\ref{theorem-2invasion} on the invasion of type 2 against type 1 
(in Section~\ref{sec-2invades1}) and Theorem~\ref{theorem-1invasion} on the reverse invasion direction (in Section~\ref{sec-otherinvasion}). One main ingredient of the proof of these theorems is the analysis of the branching processes approximating the respective (mutant) invader host types, which is carried out in Section~\ref{sec-branchingprocess}. Using this, in Section~\ref{sec-proofsketch} we sketch the proof of the two theorems, which will be completed in Appendix~\ref{sec-proofdetails}.


\subsection{Statement of results}\label{sec-12invasion}
\subsubsection{Invasion of type 2 against type 1  while \bl coexisting with type 3}\label{sec-2invades1}
Recall that $K>0$ is our carrying capacity,  $\mathbf{N}$ resp.\ $\mathbf{N}^K$ are our stochastic population models \bl and  $\mathcal T=\{ 1a,1i,2a,2d,2i,3\}$ is the type space. 
 Our starting point is a situation where type 1 is in stable coexistence with type upon arrival of a single new invader of type 2a. Thus, we assume that $(n_{1a}^*,n_{1i}^*,n_3^*)$ is asymptotically stable, and that
\[ 
\mathbf M^*_K=\big((M_{1a,K}^*,M_{1i,K}^*,1/K,0,0\color{black},M_{3,K}^*)\big)_{K>0} \numberthis\label{M*def} 
\] 
is a generic family of initial conditions for $\mathbf{N}^K$ such that $ (M_{1a,K}^*,M_{1i,K}^*,M_{3,K}^*) \in (\frac{1}{K}\N)^3$ for all $K>0$ and  
$$
\lim_{K\to\infty} (M_{1a,K}^*,M_{1i,K}^*,M_{3,K}^*)= (n_{1a}^*,n_{1i}^*,n_3^*)
$$ 
in probability. 
Now, consider the $\mathbf{N}^K$-stopping time 
\[ 
T_\beta := \inf \big\{ t \geq 0 \colon N_{j}^K(t)>\beta,~\forall j \in \mathcal T \big\},\numberthis\label{Tbetadef} 
\]
which is the first time that all sub-populations in $\mathbf{N}^K$ have size at least $\beta$ for some $\beta>0$. In this case, we say that all types are {\em macroscopic} or {\em visible} on the scale of the carrying capacity $K$. \bl 
 Further, \bl for $\eps \geq 0$ and  a subset of types \bl $A \subseteq \mathcal T$ we define the stopping time
\[ 
T_\eps^{A} = \inf \bigg\{ t \geq 0 \colon \sum_{j\in A}  N^K_{j}(t) = \frac{\lfloor \eps K \rfloor}{K} \bl \bigg\}. \numberthis\label{Tepsdef} 
\]
In particular, $T_0^A$ is the extinction time of all types  of $\mathbf{N}^K$ \bl in $A \subseteq \mathcal T$.  For convenience, we \bl abbreviate 
$$
T_\eps^{1} := T_\eps^{\{ 1a,1i\}} \quad \mbox{ and } \quad T_\eps^2:=T_\eps^{\{ 2a, 2d, 2i\}}.
$$
The next theorem states conditions  that ensure \bl that the event $\{ T_\beta < T_0^2 \}$ has an asymptotically positive probability, describes the limit of this probability, the growth rate of $T_\beta$ on the $\log K$ time scale on the event of a successful invasion, and states that unsuccessful invasions typically take an amount of time that is sub-logarithmic in $K$.

\begin{theorem}[Invasion of type 2 against type 1  while \bl coexisting with type 3]\label{theorem-2invasion} 
Assume that  $(n_{1a}^*,n_{1i}^*,n_3^*)$ is a coordinatewise positive (so that \eqref{viruscoexcondq=0} holds, which implies $mv>r+v$) and asymptotically stable equilibrium of~\eqref{3dimvirus}, and that $\lambda_2-\lambda_1 \neq \frac{qDn_3^* (r\kappa\mu_1-v\sigma)}{(r+v)(\kappa\mu_1+\sigma)}$. 
Then, we have for all sufficiently small $\beta>0$ that 
\[ 
\lim_{K\to\infty} \P\Big( T_\beta < T_0^{2}\,  \Big| \, \mathbf N^K(0) = \mathbf M^*_K \Big) = 1-s_{2a}, \numberthis\label{2invasion} 
\]
where the number $s_{2a} \in (0,1]$ is uniquely characterized as the first coordinate of the coordinatewise smallest positive solution of the system of equations~\eqref{s2as2ds2i} below. 
 In particular, 
\begin{enumerate}[(I)]
\item  $s_{2a} = 1$ holds if the~\twoNIone\ condition~\eqref{2doesn'tinvade1} holds,
\item whereas $0< s_{2a}<1$ holds if the~\twoIone\ condition~\eqref{2invades1} holds.
In this case, conditional on the event $\{ T_\beta < T_0^2 \}$ we have
\[ \lim_{K\to\infty} \frac{T_\beta}{\log K } = \frac{1}{\lambda^*} \qquad \text{ in probability}, \numberthis\label{lambda*conv} \] where $\lambda^*$ is the largest eigenvalue of the matrix $J^*$ defined in~\eqref{J*def} below (which is positive if~\eqref{2invades1} holds).
\end{enumerate}
Finally, in both cases, conditional on the event $\{ T_0^2 < T_\beta \}$, we have
\[ \lim_{K\to\infty} \frac{T_0^2}{\log K } = 0 \qquad \text{ in probability}. \numberthis\label{sublog} \]
\end{theorem}

 Under the above conditions, we say \bl 
hat \emph{type 2 can invade type 1 while \bl coexisting with type 3} (or for short, \emph{type 2 can invade type 1}) if $s_{2a}<1$,  which is true when the invasion condition \twoIone\ holds. \bl  
The proof of Theorem~\ref{theorem-2invasion} will be sketched in Section~\ref{sec-proofsketch} and carried out in full detail in Appendix~\ref{sec-proofdetails}. 


\begin{remark}[Interpretation of the invasion condition  \twoIone \ \bl from~\eqref{2invades1}\bl]
    \label{remark-rkappamu1-vsigma}
 For \bl $\lambda_2<\lambda_1$, we see that  condition~\eqref{2invades1} can only be satisfied  when $v\sigma-r\kappa\mu_1$ is positive.
Since we have 
$$
\frac{v\sigma-r\kappa\mu_1}{(r+v)(\kappa\mu_1+\sigma)} =  \frac{\sigma}{\kappa\mu_1+\sigma}-\frac{r}{r+v},
$$ 
(which also \bl holds when $\kappa$ is zero),  this means that the probability $\frac{\sigma}{\kappa\mu_1+\sigma}$ that a dormant individual resuscitates before dying needs to be higher than the probability $\frac{r}{r+v}$ that an infected individual recovers before dying by lysis.
Altogether, if dormancy comes with a reproductive trade-off, the fitness difference $\lambda_1-\lambda_2$ needs to be smaller than the net dormancy-based escape rate of individuals from lethal infection when the virus population size is at its equilibrium. \bl 

\end{remark}

\begin{remark}[Interpretation of $s_{2a}$ and $\lambda^*$]
    \label{remark-interpretation_s2_lambda_star}
 The \bl quantities $s_{2a}$ and $\lambda^*$  can be computed explicitly. \bl 
Indeed, they  depend on the mean matrix of a coupled three-type branching process describing the type 2a, 2d, and 2i population during the initial stochastic phase of the invasion. \bl This branching process will be \bl subcritical if \eqref{2doesn'tinvade1} and supercritical if~\eqref{2invades1} holds.  Now, \bl $s_{2a}$  is precisely the {\em survival probability} of the branching process when started with  a single \bl type 2a individual.
Moreover, $\lambda^*>0$ is the {\em largest eigenvalue of the mean matrix} of the branching process in the supercritical case. In Section~\ref{sec-2invasion} we will introduce this process formally.
The reason for the exclusion of equality in~\eqref{2doesn'tinvade1} in Theorem~\ref{theorem-2invasion} is due to the fact that in this case the branching process is critical, which poses technical challenges.
\end{remark}

\begin{remark}[No invasion of a costly dormancy trait in the absence of a persistent virus epidemic] \label{remark-virus-needed} \hfill
    Theorem~\ref{theorem-2invasion} states that there is a non-trivial parameter regime where type 2 can invade type 1  while \bl coexisting with type 3.  This is clearly \bl not possible in the absence of the  virus. \bl 
    Indeed, the sub-system of the six-dimensional dynamical system~\eqref{6dimvirus} corresponding to the virus-free situation $n_{1i}(0)=n_{2d}(0)=n_{2i}(0)=n_3(0)$ is the two-dimensional competitive Lotka--Volterra system~\eqref{2dimLV} where competition is symmetric.
    It is well-known (see e.g.\ \cite{Istas}) that for $\lambda_2<\lambda_1$, whenever $n_{1a}(0)$ and $n_{1i}(0)$ are positive, $(n_{1a}(t),n_{2a}(t))$ tends to $(\bar n_{1a},0)$ as $t\to\infty$. 
\end{remark}

\subsubsection{Invasion of type 1 against type 2  while \bl coexisting with type 3}\label{sec-otherinvasion}


We now present Theorem~\ref{theorem-1invasion}, the analogue of Theorem~\ref{theorem-2invasion} for the reverse invasion direction. In case $(\wt n_{2a}, \wt n_{2d}, \wt n_{2i}, \wt n_{3})$ is an asymptotically stable equilibrium of~\eqref{4dimvirus}, $\wt{\mathbf M}_K=((1/K,0,\color{black} \wt M_{2a,K},\wt M_{2d,K},\wt M_{2i,K},\wt M_{3,K}))_{K>0}$ will denote a generic family of initial conditions such that $(\wt M_{2a,K},\wt M_{2d,K},\wt M_{2i,K},\wt M_{3,K}) \in (\frac{1}{K}\N)^3$ for all $K>0$ and  
$$
\lim_{K\to\infty} (\wt M_{2a,K}, \wt M_{2d,K}, \wt M_{2i,K},\wt M_{3,K})= (\widetilde n_{2a},\widetilde n_{2d}, \widetilde n_{2i},\widetilde n_3)
$$ 
in probability.

\begin{theorem}[Invasion of type 1 against type 2 coexisting with type 3]
    \label{theorem-1invasion}
    Assume $(\wt n_{2a}, \wt n_{2d}, \wt n_{2i}, \wt n_{3})$ to be a coordinatewise positive and asymptotically stable equilibrium of~\eqref{4dimvirus} (so that \eqref{viruscoexcond2} holds, which implies $mv>r+v$),
    and $\lambda_2 - \lambda_1 \neq \frac{qD \wt n_3 (r\kappa\mu_1-v\sigma)}{(r+v)(\kappa\mu_1+\sigma)}$. Then we have for all sufficiently small $\beta>0$ that 
    \[ 
    \lim_{K\to\infty} \P\Big( T_\beta < T_0^{1} \, \Big| \, \mathbf N^K(0) = \widetilde{\mathbf M}_K \Big) = 1-s_{1a}, \numberthis\label{1invasion}
    \]
    where the number $s_{1a} \in (0,1]$ is uniquely characterized as the first coordinate of the 
    smallest positive solution of the system of equations~\eqref{s1as1i} below. Here,
    \begin{enumerate}[(I)]
    \item $s_{1a} = 1$ holds if the \oneNItwo\ condition~\eqref{1doesn'tinvade2} holds,
    \item whereas $s_{1a}<1$ holds if the \oneItwo\ condition~\eqref{1invades2} holds.
    In this case, conditional on the event $\{ T_\beta < T_0^1 \}$ we have
    \[ \lim_{K\to\infty} \frac{T_\beta}{\log K } = \frac{1}{\wt \lambda} \qquad \text{ in probability}, \] where $\widetilde \lambda$ is the largest eigenvalue of the matrix $\widetilde J$ defined in~\eqref{Jtildedef} below (which is positive if~\eqref{1invades2} holds).
    \end{enumerate}
    Finally, in both cases, conditional on the event $\{ T_0^1 < T_\beta \}$, we have
    \[ \lim_{K\to\infty} \frac{T_0^1}{\log K } = 0 \qquad \text{ in probability}.        \numberthis\label{lambdatildeconv} \]
\end{theorem}

Similarly to the case of the invasion of type 2, the quantity $s_{1a}$ is the extinction probability of the branching process approximating types 1a and 1i in the initial phase of the invasion, while $\wt \lambda>0$ is the positive largest eigenvalue of the mean matrix of the branching process whenever it is supercritical, see Section~\ref{sec-1invasion} for details. The assertion of Theorem~\ref{theorem-1invasion} is very similar to the one of Theorem~\ref{theorem-2invasion}, but the role of $n_3^*$ is now played by $\wt n_3$, and the ``$<$'' and ``$>$'' in the conditions of sub-/supercriticality are swapped. In the same vein as before, given that 
$(\wt n_{2a},\wt n_{2d}, \wt n_{2i}, \wt n_3)$ is an asymptotically stable equilibrium of~\eqref{4dimvirus}, we say that \emph{type 1 can invade type 2 while coexisting with type 3} (or \emph{type 1 can invade type 2} for short) if $s_{1a}<1$. Since the proof of Theorem~\ref{theorem-1invasion} is very similar to the one of Theorem~\ref{theorem-2invasion}, we will only provide the earlier one (in Section~\ref{sec-proofsketch}).

\begin{remark}[Mutual invasion]
Under the assumption $\lambda_2<\lambda_1$, we see that if $r\kappa\mu_1>v\sigma$ (i.e.\ dormant individuals die more frequently before becoming active again than infected individuals, cf.\ Remark~\ref{remark-rkappamu1-vsigma}), then type 1 can always invade type 2 coexisting with type 3. Nevertheless, in some cases it can also invade  when \bl $r\kappa\mu_1-v\sigma$ is negative (but not too large in absolute value), and there are cases when both host types can invade  each other while \bl coexisting with the viruses. Comparing Propositions~\ref{prop-fixation} and~\ref{prop-coexeq} with Theorems~\ref{theorem-2invasion} and~\ref{theorem-1invasion}, we see that under the assumption that $(n_{1a}^*,n_{1i}^*,n_{3}^*)$ and $(\wt n_{2a},\wt n_{2d},\wt n_{2i},\wt n_3)$ are well-defined coordinatewise positive and locally asymptotically stable equilibria of~\eqref{3dimvirus} resp.\ \eqref{4dimvirus}, this mutual ability of invasion is equivalent to the existence of the coordinatewise positive equilibrium $\mathbf x$ of \eqref{6dimvirus}, i.e.\ to the condition~\eqref{n3x3n3}. 
\end{remark}

\begin{remark}[Mutual invasion in case types 2 and 3 only coexist in the presence of type 1]
\label{remark-2novirus} 
{Consider the case when only type 1 coexists with the virus type in absence of the other host type; recall that we analysed this case in the setting of dynamical systems in Section~\ref{sec-only1virus}. 
The corresponding condition~\eqref{0x3n3} can also be interpreted as follows (just as condition~\eqref{n3x3n3}): $\mathbf x$ is coordinatewise positive if and only if both host types can invade. To see this,
recall that Theorem~\ref{theorem-2invasion} applies whenever  $(n_{1a}^*,n_{1i}^*,n_3^*)$ is a well-defined coordinatewise positive and asymptotically stable equilibrium of  the three-dimensional system\bl~\eqref{3dimvirus}. Hence, under the \twoIone~condition $x_3<n_3^*$, type 2 can invade type 1 while coexisting with type 3.
Informally speaking, in this case, the only possible analogue of the invasion of type 1 against type 2 coexisting with type 3 is the invasion of type 1a against type 2a with $\approx K \bar n_{2a}$ type 2a individuals and one type 1a individual initially. Since $\lambda_1>\lambda_2$, this invasion is always successful with asymptotically positive probability. \color{black} Thus, the \oneItwo~condition~$\wt n_3 < x_3$ degenerates to a void condition in this case.
}
\end{remark}

\subsection{The approximating branching processes}\label{sec-branchingprocess}
In this section, we provide the definition and main characteristics (mean matrix, critical behaviour, growth rate, survival probability etc.) of the multi-type branching processes approximating the mutant/invader sub-populations for large $K$ during the {\em initial stochastic phase} of their invasions, both in the context of Theorem~\ref{theorem-2invasion} (invasion of type 2 against type 1 coexisting with type 3, \twoIone, see Section~\ref{sec-2invasion}) and of Theorem~\ref{theorem-1invasion} (same with the roles of type 1 and 2 interchanged: \oneItwo, see Section~\ref{sec-1invasion}). The idea is twofold: On the one hand, when the population size of the invading type is still small, stochastic fluctuations in the reproduction of invading type matter and need to be modelled explicitly; on the other hand, the resident type populations do not yet ``feel'' the competitive pressure of the invader and remain close to their  (effectively deterministic) \bl equilibrium size. Their impact on the invader can thus also be assumed to be of that equilibrium size, hence entering ``effective birth and death rates''. The details about how  these branching processes actually approximate the invader sub-populations  during the stochastic phase \bl will be explained in Section~\ref{sec-proofsketch} below. In view of our main conjectures presented below in Section~\ref{sec-afterinvasion} below, in Section~\ref{sec-bpdiscussion} we will summarize and discuss our observations on the stability of the equilibria $(n_{1a}^*,n_{1i}^*,0,0,0,n_3^*)$ and $(0,0,\widetilde n_{2a},\wt n_{2d},\wt n_{2i},\wt n_3)$ of the system~\eqref{6dimvirus} and the  sub- and super-\bl critical behaviour of the branching processes, which are strongly related to each other. 
\subsubsection{A branching process for the invasion \bl of type 2 against type 1 while coexisting with type 3}
\label{sec-2invasion}
First we define the \twoIone\emph{-branching process}, 
which will consist of three types: \bl Its first coordinate will correspond to type 2a, the second one to type 2d and the third one to type 2i. The idea is that the sub-population sizes types 1a, 1i, and 3 will be assumed to be \bl constant equal to $Kn_{1a}^*,Kn_{1i}^*$, and $Kn_3^*$, respectively, and will not be \bl not affected by the actions of individuals of types 2a, 2d, 2i. 
We employ classical multi-type branching process theory \bl that can e.g.\ be found in~\cite[Section 7 in Chapter V]{AN72}.
 
With $(n_{1a}^*,n_{1i}^*,n_3^*)$ denoting the asymptotically stable equilibrium of~\eqref{3dimvirus},  the transition rates of our continuous-time $\N_0^3$-valued \twoIone-branching process will be given as follows: 
\begin{align*}
(k,l,n) \to
    \begin{cases}
         (k+1,l,n) &\mbox{ at rate } k\lambda_2 \mbox{ (birth of type 2a individuals)},\\
         (k-1,l,n) &\mbox{ at rate } k(\mu_1 + C (n_{1a}^*+n_{1i}^*)) \mbox{ (death of type 2a individuals)},\\
         (k-1,l,n+1) &\mbox{ at rate } (1-q) D k n_3^* \mbox{ (virus contact of type 2 leading to infection)},\\
         (k-1,l+1,n) &\mbox{ at rate } q D k n_3^* \mbox{ (virus contact of type 2 leading to dormancy)},\\
         (k+1,l-1,n) &\mbox{ at rate } \sigma l \mbox{ (resuscitation of type 2d individuals)},\\
         (k,l-1,n) &\mbox{ at rate } \kappa\mu_1 l \mbox{ (death of type 2d individuals)},\\
         (k+1,l,n-1) &\mbox{ at rate } r n \mbox{ (recovery of type 2i individuals)},\\
         (k,l,n-1) &\mbox{ at rate } vn \mbox{ (death of type 2i individuals via lysis)}.
    \end{cases}
\end{align*} 
The mean matrix of this branching process is
\[ J^*:= \begin{pmatrix} \lambda_2-\mu_1-C(n_{1a}^*+n_{1i}^*)-Dn_3^* & q D n_3^*  &(1-q)Dn_3^*   \\ \sigma  & -\kappa\mu_1-\sigma & 0 \\ r & 0 & -r-v \end{pmatrix}. \numberthis\label{J*def}  \]
We identify $J^*$ as the transpose of the matrix $A$ defined in~\eqref{Adef}, and thus the eigenvalues of $J^*$ are the same as the ones of $A$. We showed in Section~\ref{sec-fixationproof} that when the \twoIone-condition~\eqref{2invades1} holds, then $J^*$ has an eigenvalue with a positive real part and the branching process is supercritical, and that on the other hand, if the \twoNIone-condition~\eqref{2doesn'tinvade1} holds, then all eigenvalues of $J^*$ have negative real parts, so that the branching process is subcritical.

The extinction probability $s_{2a}$ can be obtained via standard first-step analysis. Indeed, for $\upsilon \in \{ 2a, 2d, 2i \}$ let $s_{\upsilon}$ denote the probability that the \twoNIone
-branching process started from a single type $\upsilon$ individual 
goes extinct within finite time. Then $(s_{2a},s_{2d},s_{2i})$ is the coordinatewise smallest positive solution to the system of equations
\[
\begin{aligned}
0 & = \lambda_2 (s_{2a}^2-s_{2a})+ (\mu+C(n_{1a}^*+n_{1i}^*))(1-s_{2a}) + (1-q) D n_{3}^* (s_{2i}-s_{2a}) + q D n_3^* (s_{2d}-s_{2a}), \\
0 & = \kappa \mu_1 (1-s_{2d}) + \sigma (s_{2a}-s_{2d}), \\
0 & = r (s_{2a}-s_{2i}) + v (1-s_{2i}).
\end{aligned}
\numberthis\label{s2as2ds2i}
\]
Since the matrix $J^*$ is irreducible, it follows that in the supercritical case $s_{2a},s_{2d}$, and $s_{2i}$ are all less than one (while in the subcritical and critical case of course $s_{2a}=s_{2d}=s_{2i}=1$). In this case, the unique positive eigenvalue of the matrix will be denoted by $\lambda^*$, and this is the quantity \bl which appears in~\eqref{lambda*conv}. 

\subsubsection{A branching process for the invasion \bl of type 1 against type 2  while \bl coexisting with type 3}\label{sec-1invasion}

Now we define the 
\oneItwo-branching process, which will  have only two types: \bl Its first coordinate will correspond to type 1a and its second one to type 1i. The principle of the approximation is similar to the case of the \twoIone-branching process,  now assuming \bl that the sub-population sizes types 2a, 2d, 2i, and 3 are constant equal to $K \wt n_{2a},K \wt n_{2d}, K\wt n_{2i}$, and $K\wt n_3$, respectively, and not affected by the actions of types 1a and 1i. 

In case $(\wt n_{2a},\wt n_{2d}, \wt n_{2i}, \wt n_3)$ is an  asymptotically \bl stable equilibrium of~\eqref{4dimvirus}, we define the \oneItwo-branching process as the linear (binary) branching process in continuous time with state space $\N_0^2$ and transition rates, for $(i,j) \in \N_0^2$, 
\[ (i,j) \to \begin{cases}
    (i+1,j) & \text{ at rate $i\lambda_1$ (birth of type 1a individuals)}, \\
    (i-1,j) & \text{ at rate $i(\mu_1 + C (\widetilde n_{2a}+\widetilde n_{2d}+\widetilde n_{2i}))$ (death of type 1a individuals)}, \\
    (i-1,j+1) & \text{ at rate $Di \widetilde n3$ (virus contacts of type 1, leading to infection of the host)},  \\
    (i+1,j-1)  & \text{ at rate $r j$ (recovery of type 1i individuals)}, \\
    (i,j-1) & \text{ at rate $vj$ (death of type 1i individuals via lysis)}.
\end{cases}
\]
This branching process has mean matrix
\[ \widetilde J:= \begin{pmatrix} \lambda_1-\mu_1-C(\widetilde n_{2a}+\widetilde n_{2d}+\widetilde n_{2i})-D\tilden3 &  D\tilden3  \\ r& -r-v \end{pmatrix}. \numberthis\label{Jtildedef}\]
Since $\widetilde J$ is the transpose of the matrix $F$ defined in~\eqref{Fdef}, it follows from Section~\ref{sec-fixationproof} that  under the invasion condition \oneItwo \ from~\eqref{1invades2} the branching process is supercritical, while under the reverse \bl condition \oneNItwo 
it is strictly subcritical. 

The extinction probability $s_{1a}$ can again be \bl derived via standard first-step analysis. For $\chi \in \{ 1a, 1i \}$ let $s_{\chi}$ denote the probability that the \oneItwo-branching process started from a single \bl type $\chi$ individual 
goes extinct within finite time. Then $(s_{1a},s_{1i})$ is the coordinatewise smallest positive solution to
\[
\begin{aligned}
0 & = \lambda_1 (s_{1a}^2-s_{1a})+ (\mu_1+C(\wt n_{2a}+\wt n_{2i}))(1-s_{1a}) + D \wt n_{3} (s_{1i}-s_{1a}), \\
0 & = r (s_{1a}-s_{1i}) + v (1-s_{1i}).
\end{aligned} \numberthis\label{s1as1i}
\]
Again, the mean matrix $\wt J$ is irreducible, which implies that $s_{1a}$ and $s_{1i}$ are both less than one in the supercritical case (and of course, they are both one in the subcritical and critical case).

\subsection{Sketch of proof of Theorems~\ref{theorem-2invasion} and~\ref{theorem-1invasion}}\label{sec-proofsketch}
In this section, we will focus on the proof of Theorem~\ref{theorem-2invasion}; the  proof \bl of Theorem~\ref{theorem-1invasion} can be carried out in a similar fashion, just with the roles of ``resident'' and ``mutant/invader population'' exchanged. 
The overall strategy will be similar to the proofs of~\cite[Theorems 2.8, 2.9, and 2.10]{BT21}. However, in our case, due to the additional dimensions of the dynamical system that prevent us from carrying out a full stability analysis, some  further \bl arguments and workarounds  are required. Remarkably, in some aspects the proof will actually be simpler than  those \bl in \cite{BT21}, since in our case the approximating branching processes have {\em irreducible} mean matrices (see Sections~\ref{sec-2invasion} and~\ref{sec-1invasion}). 




\medskip

{\bf Proof strategy.} For the analysis of the initial phase of the invasion (including the branching process approximation of the mutant), we fix a small positive approximation \color{black} parameter $\eps$ that we will let tend to $0$ in the end \emph{after} taking the large-population limit $K\to\infty$. The key steps of our approach for handling the first phase are the following.

\begin{itemize}
    \item[i)] We first check \bl that $(N^K_{1a}(t),N^K_{1i}(t),N^K_{3}(t))_{t\geq 0}$ stays close to $(n_{1a}^*,n_{1i}^*,n_{3}^*)$ until $N_{2a}(t)+N_{2d}(t)+N_{2i}(t)$ reaches size $\lfloor \eps K \rfloor$ or $0$, with probability tending to 1 in the limit $K\to\infty$  for all $\varepsilon$ small enough. \bl This can be achieved by using standard Freidlin--Wentzell type large-deviation arguments. See Lemma~\ref{lemma-2residentsstay} for a precise formulation.
    
    \medskip
    
    \item[ii)] Given this control of $(N^K_{1a}(t),N^K_{1i}(t),N^K_{3}(t))_{t\geq 0}$  during the initial phase, we then show that the invading $N_2(t)$-subpopulation can be well-approximated by the \twoIone-branching process from Section\ \ref{sec-branchingprocess}. The point is that the interaction with the rescaled subpopulations $N_{1a}^K(t)$, $N_{1i}^K(t)$ and $N_3^K(t)$ can be replaced by an interaction with their equilibrium values $(n_{1a}^*,n_{1i}^*,n_{3}^*)$ without changing the limiting invasion probability and asymptotic invasion time. This is achieved by a \bl coordinate-wise ``sandwich-coupling'' of  both $(N_{2a}(t),N_{2d}(t),N_{2i}(t))$ and the \twoIone-branching process \bl between two  further \bl three-type branching processes with fixed lower and upper bounds on the interaction with residents and virions, whose distribution does not depend on $K$. This can be done in a way that the extinction probability of both  enveloping \bl branching processes tends to $s_{2a}$ and the largest eigenvalue of their mean matrix to $\lambda^*$ in the limit $\eps \downarrow 0$,
     which \bl allows us to prove Proposition~\ref{prop-2firstphase} below.

\medskip
    
    \item[iii)]  Finally, \bl starting from a state where the total mutant population size  $N_2$ \bl is $\lfloor \eps K \rfloor$  while all other sub-populations are still close to their positive equilibrium (also of order $K$) \bl and approximating the rescaled stochastic process $\mathbf N^K(t)$ by the solution $\mathbf n(t)$ to the dynamical system~\eqref{6dimvirus}, a comparison to a  suitable \bl system of linear ODEs can be used to ensure that the rescaled stochastic process  $\mathbf N^K$ \bl reaches a state in which each component takes at least  a \bl value  $\beta>0$ (not depending on $\varepsilon$) in very short time, and with probability tending to 1 as $\varepsilon \downarrow 0$. This is \bl stated in Proposition~\ref{prop-2dynsyst}. 
\end{itemize}
Based on these three steps, in Section~\ref{sec-proofdetails} we will complete the proof of Theorem~\ref{theorem-2invasion}, using moderate modifications of a chain of arguments from~\cite{BT21}, in turn  originating \bl from~\cite{C+19}.

Indeed, \bl the proof of Lemma~\ref{lemma-2residentsstay} is based on the proof of~\cite[Lemma 3.2]{C+19}. Further, \bl compared to~\cite{C+19,BT19} 
the proof of Proposition~\ref{prop-2firstphase} only requires a suitable ``taylor-made'' \bl construction of the  approximating sandwiching branching processes. Nevertheless, we provide the full proof of the lemma and the proposition in Appendix~\ref{sec-proofdetails} below for completeness. However, in order to finish the proof of Theorem~\ref{theorem-2invasion}, also some \bl new arguments are required, in particular for \bl Step iii). \bl  
They are necessary since for the present model we are unable to establish convergence of the dynamical system to a suitable equilibrium from distant initial conditions. Our alternative arguments lead to \bl Proposition~\ref{prop-2dynsyst}.

\section{Dynamics after a successful invasion: Conjectures, simulations, and heuristics}\label{sec-afterinvasion}

We now present conjectures for the fate of the population after a successful invasion in either direction in Section~\ref{sssn-conjectures}, followed by some related simulations and heuristics in Section~\ref{sssn-simulations-heuristics}.

\subsection{Dynamics after a successful invasion: Conjectures, simulations, and heuristics}
\subsection{Conjectures}\label{sssn-conjectures}
Theorems~\ref{theorem-2invasion} and~\ref{theorem-1invasion} together with Propositions~\eqref{prop-fixation} and~\ref{prop-coexeq} 
 indicate a rather complete picture of \bl
the qualitative behaviour of the system~\eqref{6dimvirus} as well as our individual-based model $\mathbf{N}^K$ when $K$ is large. \bl Indeed, in the regime where an invasion is possible 
in both directions (or, equivalently, where $\mathbf x$ is coordinatewise positive), we expect stable long-term coexistence between the two host types and the virus, while in the regime where only one invasion direction is possible, we expect fixation of one of the hosts and extinction of the other,
while the virions (and thus also the infected/dormant states of the fixing host type) will stay in the system.

Since the detailed formulation of the corresponding conjectures 
are already somewhat lengthy in the biologically more natural case $\lambda_2<\lambda_1$, we refrain from stating the analogous conjectures for \bl $\lambda_2>\lambda_1$. However, \bl the visual summary of the parameter regimes in Figures~\ref{figure-regimes} and~\ref{figure-regimesnocoex} also includes the latter case.


\begin{conj}[Stability of the equilibria of~\eqref{6dimvirus}]\label{conj-fixation}
Let us assume that $r \kappa\mu_1 \neq v\sigma$ and that the \twoC-condition\bl~\eqref{viruscoexcondq=0} holds (which in particular implies $mv>r+v$). 
\begin{enumerate}[(A)]
\item If $(n_{1a}^*,n_{1i}^*,n_{3}^*)$ is an asymptotically stable equilibrium of~\eqref{3dimvirus}, then for any solution $((n_{1a}(t),n_{1i}(t),n_{3}(t))_{t\geq 0}$ to~\eqref{3dimvirus} such that $n_{1a}(0),n_{1i}(0),n_{3}(0)>0$, we have
\[ \lim_{t\to\infty} n_{\upsilon}(t)= n_{\upsilon}^*(t), \qquad \forall \ups \in \{ 1a, 1i, 3\}. \]
\item Similarly, if $(\wt n_{2a},\wt n_{2d},\wt n_{2i},\wt n_3)$ is an asymptotically stable equilibrium of~\eqref{4dimvirus}, then 
for any solution $((n_{2a}(t),n_{2d}(t),n_{2i}(t),n_3(t))_{t\geq 0}$ to~\eqref{4dimvirus} such that $n_{2a}(0),n_{2i}(0),n_{2d}(0),n_3(0)>0$, we have
\[ \lim_{t\to\infty} n_{\upsilon}(t)= \wt n_{\upsilon}(t), \qquad \forall \ups \in \{ 2a, 2d, 2i, 3\}. \]
\item Let us assume that both $(n_{1a}^*,n_{1i}^*,n_{3}^*)$  and $(\wt n_{2a},\wt n_{2d},\wt n_{2i},\wt n_3)$  are asymptotically stable. Then for any solution $(\mathbf n(t))_{t\geq 0}=((n_{1a}(t),n_{1i}(t),n_{2a}(t),n_{2d}(t),n_{2i}(t),n_3(t))_{t\geq 0}$ to~\eqref{6dimvirus} such that $(n_{1a}(0),n_{1i}(0),n_{2a}(0),n_{2d}(0),n_{2i}(0),n_3(0) \in (0,\infty)^6$, 

\[ \lim_{t\to\infty} \mathbf n(t) = \begin{cases}
 (n_{1a}^*,n_{1i}^*,0,0,0,n_3^*) & \text{ if \eqref{2doesn'tinvade1} and~\eqref{1invades2} both hold
 \emph{(case (i))}}, \\
    (0,0,\widetilde n_{2a}, \widetilde n_{2d}, \widetilde n_{2i},\widetilde n_3) & \text{ if \eqref{2invades1} and~\eqref{1doesn'tinvade2} both hold \emph{(case (ii))},}
    \\ 
    \mathbf x &  \text{ if~\eqref{2invades1} and \eqref{1invades2} both hold \emph{(case (iii))}.} 
\end{cases} \] 
In all these three cases, 
\eqref{6dimvirus} has a unique asymptotically stable coordinatewise non-negative equilibrium, which equals the limit above.
\item Let us assume that $(n_{1a}^*,n_{1i}^*,n_{3}^*)$  is asymptotically stable and $\bar n_{2a} < \wt n_{2a}$. Then for any solution $(\mathbf n(t))_{t\geq 0}$ to~\eqref{6dimvirus} such that $\mathbf n(0) \in (0,\infty)^6$, \[ \lim_{t\to\infty} \mathbf n(t) = \begin{cases}
 (n_{1a}^*,n_{1i}^*,0,0,0,n_3^*) & \text{ if \eqref{2doesn'tinvade1} and~\eqref{1invades2} both hold
 \emph{(case (i))}}, \\
    \mathbf x &  \text{ if~\eqref{2invades1} and \eqref{1invades2} both hold \emph{(case (iii))}.} 
\end{cases} \] 
In both cases, the unique asymptotically stable equilibrium of \eqref{6dimvirus} equals the limit above.
\end{enumerate}
\end{conj}

We note that in assertions (C)--(D), (i) refers to fixation of type 1, (ii) to fixation of type 2, and (iii) to 
stable coexistence of all the six host types. In assertion (D), where type 2 does not coexist with type 3 in absence of type 1, there is no case (ii) (because in absence of the virus epidemic, type 2a cannot defeat type 1a for $\lambda_2<\lambda_1$, cf.\ Remark~\ref{remark-2novirus}).

While the coordinatewise positivity of $\mathbf x$ only requires the coordinatewise positivity of the positive equilibria of the corresponding sub-systems~\eqref{3dimvirus} and~\eqref{4dimvirus} (cf.\ Propositions~\ref{prop-coexeq} and~\ref{prop-coexeqonly1virus}), 
the stability of equilibria of the 6-dimensional system~\eqref{6dimvirus} that we anticipate in Conjecture~\ref{conj-fixation} should only hold under the additional assumption that  $(n_{1a}^*,n_{1i}^*,n_{3}^*)$  is asymptotically stable for~\eqref{3dimvirus} and $(\wt n_{2a},\wt n_{2d},\wt n_{2i},\wt n_3)$ (whenever it exists) for~\eqref{4dimvirus}. The stability of the latter two equilibria may be lost for large $m$ due to Hopf bifurcations, see Appendix~\ref{sec-Hopfwiederholung}.

Let us now also formulate conjectures for the stochastic process $(\mathbf N(t)^K)_{t\geq 0 }$. \bl  
For any equilibrium $\widehat{\mathbf n}=(\widehat n_{1a}, \widehat n_{1i}, \widehat n_{2a}, \widehat n_{2d}, \widehat n_{2i}, \widehat n_3) \in [0,\infty)^6$ of~\eqref{6dimvirus} and for any $\delta>0$ we define the stopping time
\[ T_{\widehat{\mathbf n},\delta} = \inf \{ t \geq 0 \colon  \Vert \mathbf N^K(t) - \widehat{\mathbf n} \Vert_{\infty} \leq \delta \}. \] We write $\mathbf n^* = (n_{1a}^*,n_{1i}^*,0,0,0,n_{3}^*)$ and $\wt{\mathbf n}=(0,0,\wt n_{2a}, \wt n_{2d}, \wt n_{2i}, \wt n_3)$, and we define
\[ T_{\mathbf n^*,\delta}^{\mathrm{fix}} = \inf \Big\{ t \geq 0 \colon  \Vert \mathbf N^K(t) - \mathbf n^* \Vert_{\infty} \leq \delta \text{ and } N_{2a}(t)+N_{2d}(t)+N_{2i}(t)=0 \Big\}  \]
as well as
\[ 
T_{\widetilde{\mathbf n},\delta}^{\mathrm{fix}} = \inf \Big\{ t \geq 0 \colon  \Vert \mathbf N^K(t) - \widetilde{\mathbf n} \Vert_{\infty} \leq \delta \text{ and } N_{1a}(t)+N_{1i}(t)=0 \Big\}. 
\]
Reaching $T_{\widetilde{\mathbf n},\delta}$ means that types 2a, 2d, 2i, and 3 are within a $\delta$-neighbourhood of their equilibrium $(n_{1a}^*,n_{1i}^*,n_3^*)$, while types 1a and 1i are entirely extinct, with a similar interpretation for \bl $T_{\mathbf n^*,\delta}^{\mathrm{fix}}$. Clearly, $T_{\mathbf n^*,\delta} \leq T_{\mathbf n^*,\delta}^{\mathrm{fix}}$ and $T_{\widetilde{\mathbf n},\delta} \leq T_{\widetilde{\mathbf n},\delta}^{\mathrm{fix}}$ a.s.\ as $[0,\infty]$-valued random variables defined on the same probability space. 

In the next two conjectures, the cases (C), (D) and the corresponding sub-cases (i), (iii) and (in case (C)) (ii)  will be numbered analogously to Conjecture~\ref{conj-fixation}. 

\begin{conj}[Invasion of type 2: Conjectured \bl full version of Theorem~\ref{theorem-2invasion}] \label{conj-theorem2invades1}
Let us assume that  $(n_{1a}^*,n_{1i}^*,n_{3}^*)$ is an asymptotically stable equilibrium of~\eqref{3dimvirus}.
\begin{enumerate}[(A)]
\setcounter{enumi}{2}
    \item Assume further that $(\wt n_{2a},\wt n_{2d},\wt n_{2i},\wt n_3)$ is an asymptotically stable equilibrium of~\eqref{4dimvirus}. Then, for all sufficiently small $\delta>0$ we have that
\begin{enumerate}[(i)]
\item[(ii)] if~\eqref{2invades1} and~\eqref{1doesn'tinvade2} both hold, then
\[ \lim_{K\to\infty} \P\Big( T_{\wt{\mathbf n},\delta}^{\mathrm{fix}} < T_0^{2} \wedge T_{\mathbf x,\delta} \Big| \mathbf N^K(0) = \mathbf M^*_K \Big) = 1-s_{2a}, \]
and conditional on the event $\{  T_{\wt{\mathbf n},\delta} < T_0^{2} \wedge  T_{\mathbf x,\delta}  \}$,
\[ \lim_{K\to\infty} \frac{T_{\wt{\mathbf n},\delta}^{\mathrm{fix}}}{\log K} = \frac{1}{\wt \lambda} + \frac{1}{\lambda^*} \text{ in probability;} \]
\item[(iii)] while if~\eqref{2invades1} and~\eqref{1invades2} both hold, then 
\[ \lim_{K\to\infty} \P\Big( T_{\mathbf x,\delta} < T_0^{2} \wedge T_{\wt{\mathbf n},\delta}^{\mathrm{fix}} \Big| \mathbf N^K(0) = \mathbf M^*_K \Big) = 1-s_{2a} \numberthis\label{fulls2a} \]
and conditional on the event $\{ T_{\mathbf x,\delta} < T_0^{2} \wedge T_{\wt{\mathbf n},\delta}^{\mathrm{fix}} \}$,
\[ \lim_{K\to\infty} \frac{T_{\mathbf x,\delta}}{\log K} = \frac{1}{\wt \lambda} \text{ in probability.} \numberthis\label{fulllambdatilde} \]
\end{enumerate}
\item  (iii) \color{black} Assume now that $\wt n_{2a}<\bar n_{2a}$ and~\eqref{2invades1} (equivalently~\eqref{0x3n3}) holds. Then~\eqref{fulls2a} and~\eqref{fulllambdatilde} hold.
\end{enumerate}
\end{conj}

\begin{conj}[Invasion of type 1: Conjectured full version of Theorem~\ref{theorem-1invasion}] \label{conj-theorem1invades2}
Let us assume that  $(n_{1a}^*,n_{1i}^*,n_{3}^*)$ is an asymptotically stable equilibrium of~\eqref{3dimvirus} and $(\wt n_{2a},\wt n_{2d},\wt n_{2i},\wt n_3)$ is an asymptotically stable equilibrium of~\eqref{4dimvirus}. 
\begin{enumerate}[(C)]
\item[(C)]
Then, for all sufficiently small $\delta>0$ we have that \bl 
\begin{enumerate}[(i)]
\item[(i)] if~\eqref{2doesn'tinvade1} and~\eqref{1invades2} both hold, then
\[ \lim_{K\to\infty} \P\Big( T_{{\mathbf n^*},\delta}^{\mathrm{fix}} \color{black} < T_0^{1} \wedge  T_{\mathbf x,\delta} \Big| \mathbf N^K(0) =\widetilde{\mathbf M}_K \Big) = 1-s_{1a}, \]
and conditional on the event $\{  T_{\mathbf n^*,\delta} < T_0^{1} \wedge T_{\mathbf x,\delta} \}$,
\[ \lim_{K\to\infty} \frac{T_{{\mathbf n}^*,\delta}^{\mathrm{fix}}}{\log K} = \frac{1}{\wt \lambda} + \frac{1}{\lambda^*} \text{ in probability;} \]
\item[(iii)] while if~\eqref{2invades1} and~\eqref{1invades2} both hold, then 
\[ \lim_{K\to\infty} \P\Big( T_{\mathbf x,\delta}^{\mathrm{fix}} < T_0^{1} \wedge T_{\mathbf n^*,\delta} \Big| \mathbf N^K(0) = \widetilde{\mathbf M}_K \Big) = 1-s_{1a}, \]
and conditional on the event $\{ T_{\mathbf x,\delta} < T_0^{1} \wedge  T_{\mathbf n^*,\delta}^{\mathrm{fix}} \}$,
\[ \lim_{K\to\infty} \frac{T_{\mathbf x,\delta}}{\log K} = \frac{1}{\lambda^*} \text{ in probability.} \]
\end{enumerate}
\end{enumerate}
\end{conj}


Recall that we know from Theorem~\ref{theorem-2invasion} that type 2 can invade in cases (ii) and (iii) and cannot invade in case (i), while Theorem~\ref{theorem-1invasion} states \bl that type 1 can invade in cases (i) and (iii) but not in case (ii). 
The parameter regimes corresponding to the three cases (i)--(iii), both in case (C) (where type 2 coexists with type 3) and (D) (where it does not) are summarized in Table~\ref{table-cases}.

\medskip



\begin{table}
\small
\begin{tabular}{|lcc|cccc|}  \hline
Case  & $r\kappa\mu_1-v\sigma$ &  Virus \bl &  Coordinatewise  \color{black} & Can type 2 & Can type 1 & Conjectured \\ & & & positive $\mathbf x$ & invade? & invade? & outcome \\ \hline
(C) (i) &	$<0$ & $0<\widetilde n_3<n_3^* < x_3$ & $\nexists$ & 		no & 	yes	& fixation of 1 \\ \hline
(C) (i) & $>0$ & $x_3 <0<\wt n_3<n_3^*$ & $\nexists$ & no & yes & fixation of 1 \\ \hline
(C) (ii) & $<0$  & $0<x_3<\widetilde n_3<n_3^*$ & $\nexists$ & yes & no & fixation of 2 \\ \hline
(C) (iii)  &	$<0$ &$0<\widetilde n_3 < x_3 < n_3^*$ &	$\exists$ &	yes	& yes &	stable coex.\ \\ \hline \hline
(D) (i) &	$<0$ & $0<n_3^* < x_3$ & $\nexists$ & 		no & 	yes	& fixation of 1 \\ \hline
(D) (i) & $>0$ & $x_3 <0<n_3^*$ & $\nexists$ & no & yes & fixation of 1 \\  \hline
(D) (iii)  &	$<0$ &$0 < x_3 < n_3^*$ &	$\exists$ &	yes	& yes &	stable coex.\ \\ 
& & & & & & when 2 invades \\ \hline
\end{tabular}
\caption{Overview of the results of Theorems~\ref{theorem-2invasion} and~\ref{theorem-1invasion}, extended by the conjectured long-term behaviour \bl according to Conjectures~\ref{conj-theorem2invades1} and~\ref{conj-theorem1invades2}, in the case when $\lambda_2<\lambda_1$ and type 1 coexists with the virus \bl in absence of type 2. 
}\label{table-cases}
\end{table}

\begin{remark}[Missing and available proof ingredients]\label{remark-missing_ingredients}
The only missing ingredient of the proof of Conjectures~\ref{conj-theorem2invades1} and~\ref{conj-theorem1invades2} is the convergence of the dynamical system~\eqref{6dimvirus}. If we had such an assertion, the proof of the conjectures could be completed along the lines of the proofs of the main results of~\cite{BT20},  using additional branching process approximations \bl from~\cite{C+16} 
for the final phase of the extinction of the former resident population.
\bl
\end{remark}

\begin{remark}[Unsuccessful invasions and critical cases]\label{remark-afterconj} 

For fixed $\beta>0$, considering our initial condition in Theorem~\ref{theorem-2invasion}, we have $T_{\beta} \leq T_{\mathbf{\widetilde n},\delta}^{\mathrm{fix}} \wedge T_{\mathbf x,\delta}$ for all $\delta$ sufficiently small. This readily implies that~\eqref{sublog} also holds for all sufficiently small $\delta>0$ conditional on the event $\{  T_0^2 < T_{\mathbf{\widetilde n},\delta}^{\mathrm{fix}}  \wedge T_{\mathbf x,\delta} \}$ in cases (ii) and (iii). An analogous assertion applies for the reverse invasion direction. Thus, when the \twoNIone\ non-invasion condition\color{black}~\eqref{2doesn'tinvade1} holds, Theorem~\ref{theorem-2invasion} fully describes the invasion of type 2, 
which is unsuccessful with overwhelming probability. Hence, we left out case (i) from Conjecture~\ref{conj-theorem2invades1} and analogously case (ii) from Conjecture~\ref{conj-theorem1invades2} when~\eqref{1doesn'tinvade2} holds.

Note that (i)--(iii) do not cover the case when the \oneItwo\ and \twoIone-invasion conditions~\eqref{1invades2} or~\eqref{2invades1} hold with equality since then the the equilibrium $(n_{1a}^*,n_{1i}^*,0,0,0,n_3^*)$ resp.\ $(0,0,\wt n_{2a},\wt n_{2d}, \wt n_{2i}, \wt n_3)$ becomes hyperbolic and thus its stability cannot be determined by linearization (equivalently, the corresponding branching process becomes critical).
\bl
\end{remark}

\begin{remark}[Fixation takes longer than reaching six-type coexistence]\label{remark-afterconj2}
In the cases when we expect fixation of type 2 (resp.\ 1) and survival of the virus epidemic, before reaching time $T^{\mathrm{fix}}_{\widetilde{\mathbf n},\delta}$ (resp.\ $T^{\mathrm{fix}}_{\mathbf n^*,\delta}$) there must be an additional phase where type 1 (resp.\ 2) goes extinct. This phase should typically take an additional amount of  $\frac{1}{\wt \lambda} (1+o(1))\log K$ (resp.\ $\frac{1}{\lambda^*} (1+o(1))\log K$) time after time $T_\beta$, $\beta>0$. Indeed, the population size of the waning former resident type can be approximated by the same subcritical branching process as when it is the invader type and during the initial phase of its invasion its population is still small (but not yet extinct). In contrast, in the case when we expect coexistence, there is no such final extinction phase, and $ (T_{\mathbf x,\delta}-T_\beta)/\log K$ should tend to zero in probability conditional on the event $\{ T_{\mathbf x,\delta}<\infty \}$. 
\bl
\end{remark}

\begin{remark}[Correspondence between stability of the small equilibria and six-type coexistence]\label{remark-afterconj3}
If Conjectures~\ref{conj-theorem2invades1} and~\ref{conj-theorem1invades2} hold true, then under the assumption of asymptotic stability of the coordinatewise positive equilibria $(n_{1a}^*,n_{1i}^*,n_{3}^*)$ and $(\wt n_{2a},\wt n_{2d}, \wt n_{2i},\wt n_3)$ of the sub-systems, $\mathbf x$ is always asymptotically stable when it is coordinatewise positive (if $\lambda_2<\lambda_1$), which happens precisely in
case (iii) where the two host types stably coexist with the virus type. Recall that we already know that in cases (i) and (ii) there is no coordinatewise positive equilibrium (cf.\ Proposition~\ref{prop-coexeq}). 
\bl
\end{remark}

\begin{remark}[Type 2 needs type 3 to drive type 1 to extinction, and cannot eliminate type 3]\label{remark-afterconj4}  \hfill The gene\-ral picture is that for $\lambda_2<\lambda_1$, there should be no case where type 2 eliminates type 3, and type 2 should only be able to drive  type 1 to extinction  if type 2 itself coexists with type 3 when being the only host. The reason is that in absence of the virus, and even when the virus population becomes very small on its path to extinction, always the host with a higher reproduction rate (in this case type 1) wins. 
\bl
\end{remark}

\subsection{Simulations and heuristics\bl}\label{sssn-simulations-heuristics}

We now provide case-by-case illustrations \bl of the conjectured behaviour of the dynamical system~\eqref{6dimvirus} and our stochastic process $(\mathbf N(t))_{t\geq 0}$ in the different regimes of fixation and coexistence by a pair of simulations (for the invasion of type 2 resp.\ 1) each. We relate these regimes to the colours used in Figure~\ref{figure-regimes} below, and the choices of parameters in the exemplary simulations are also chosen according to Figure~\ref{figure-regimes}, with the exception of $\lambda_2$ and $q$, which we vary. \bl These parameters are such that $r\kappa\mu_1<v\sigma$, so that stable coexistence is possible for certain values of $\lambda_2<\lambda_1$ and $q>0$ (while founder control can never occur). Moreover, the choice of the parameters 
ensures that type 1 stably coexists with type 3 in absence of type 2. 
For brevity, \bl we refrain from simulations in the biologically less relevant case $\lambda_2>\lambda_1$.

\begin{figure}[!ht]
    \centering
    \includegraphics[width=0.5\linewidth]{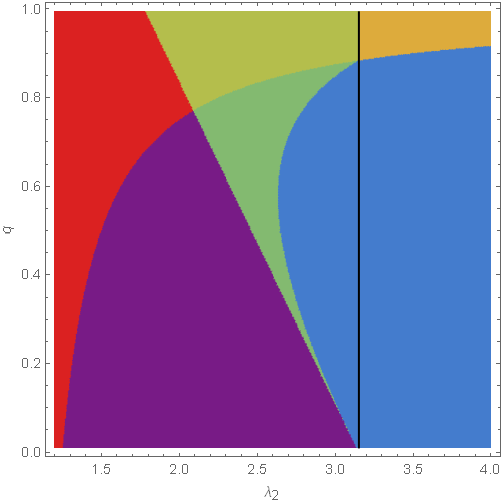}
    \caption{Conjectured invasion outcome depending on $\lambda_2\in(1.2\color{black},4)$ and $q\in(0.01,0.99\color{black})$ for fixed  $\lambda_1=3.15,\mu_1=1, C=1, D=0.5, r=1, v=1, \kappa=0.1, \sigma=2, m=10, \mu_3=0.5$. 
    $\color{red1}\rule{.3cm}{.3cm}$ Red (left): fixation of type 1 (coex.\ with 3), $\color{lightgreen}\rule{.3cm}{.3cm}$ light green (top mid/left): stable 6-dim.\ coexistence (type 2 is not able to coexist with type 3 in absence of type 1)\color{black}, $\color{darkgreen}\rule{.3cm}{.3cm}$ dark green (mid/left):  stable 6-dim.\ coexistence (type 2 is able to coexist with type 3)\color{black}, $\color{orange1}\rule{.3cm}{.3cm}$ orange (top right): fixation of type 2a (without 3), $\color{purple1}\rule{.3cm}{.3cm}$ purple (bottom mid/left): fixation of type 1 (coex.\ with 3), $\color{blue1}\rule{.3cm}{.3cm}$ blue (bottom right): fixation of type 2 (coex.\ with 3). 
   The curve separating red from purple, light green from dark green, and orange from blue corresponds to $\bar n_{2a}=\wt n_{2a}$. Type 2 only coexists with type 3 below this curve. 
    The light green and the orange regime are separated by the line $\lambda_2=\lambda_1=3.15$; below this value of $\lambda_2$, fixation of type 2a without 3 not possible. The dark green area reaches this line at 0, with vanishing width.
    }
    \label{figure-regimes}
\end{figure}

    
\noindent {\bf (C) When type is \bl 2 able to coexist with type 3 in absence of type 1.}
\begin{itemize}
\item 
{\bf The case when type 2 can invade type 1 while coexisting with type 3:} According to Theorem~\ref{theorem-2invasion} reformulated in terms of Proposition~\ref{prop-coexeq}, when $(n_{1a}^*,n_{1i}^*,n_3^*)$ is asymptotically stable, type 2  invades \bl precisely when $x_3<n_3^*$.
Let us consider this scenario under the additional assumption that $(\widetilde n_{2a},\wt n_{2d}, \wt n_{2i},\wt n_3)$ is also asymptotically stable.\smallskip
\begin{enumerate}
    \item[(iii)]  $\wt n_3< x_3<n_3^*$: $\color{darkgreen}\rule{.5cm}{.3cm}$
    \bl \textbf{Dark green areas in Figure~\ref{figure-regimes} -- stable six-type coexistence}:  Under the additional assumption $x_3 > \wt n_3$, type 1 can also invade and hence we expect stable six-type coexistence. In this case, we see interesting  trade-offs \bl equilibrating each other (cf.~Figure~\ref{figure-green}).
    \begin{figure}[!ht]
        \centering
        \includegraphics[width=0.38\linewidth]{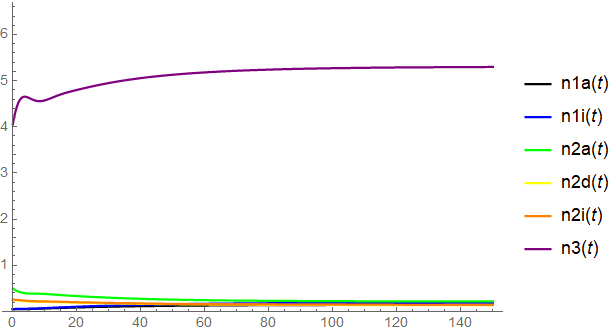} \quad
        \includegraphics[width=0.38\linewidth]{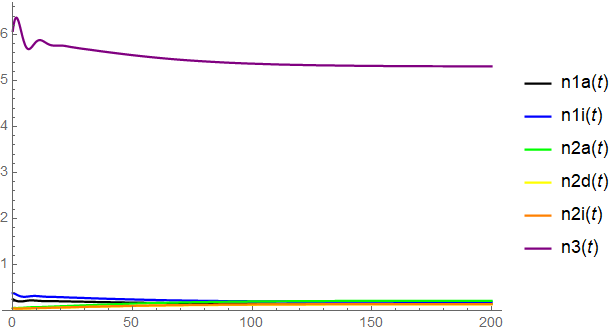} \\
        \includegraphics[width=0.38\linewidth]{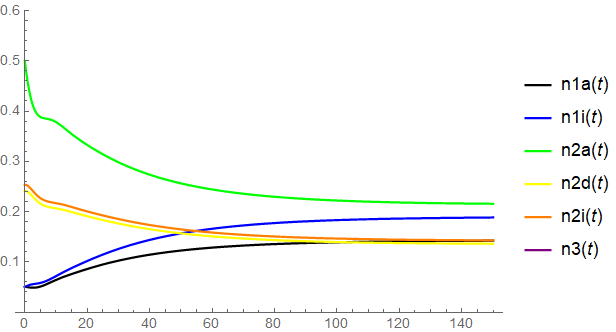} \quad
        \includegraphics[width=0.38\linewidth]{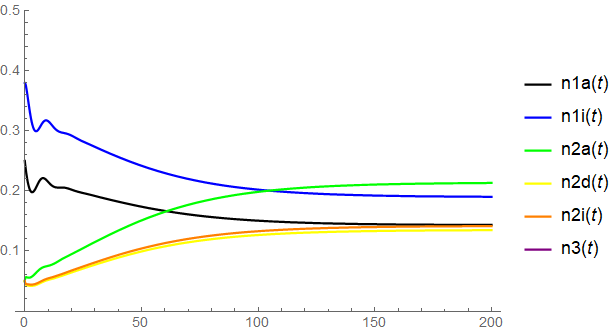}
        \caption{$\color{darkgreen}\rule{.5cm}{.3cm}$ Dark green regime in Figure~\ref{figure-regimes} with \bl $\lambda_2=2.55$ and $q=0.6$: Stable six-type coexistence. Invasion of type 1 against type 2 coexisting with type 3 (starting near $(0,0,\wt n_{2a},\wt n_{2d},\wt n_{2i},\wt n_3)$, left) and of type 2 against type 1 coexisting with type 3 (starting near $(n_{1a}^*,n_{1i}^*,0,0,0,n_3^*)$, right) in the dynamical system~\eqref{6dimvirus}, both letting the solution converge to $\mathbf x$ as $t\to\infty$, in accordance with Conjecture~\ref{conj-fixation}. The bottom images are zoomed versions of the top ones displaying the small coordinates of the solution. Observe in the plots that $\wt n_{3}<x_3<n_3^*$. We checked numerically that here, $\mathbf x$ is indeed at least locally asymptotically stable (the Jacobi matrix has 4 real negative eigenvalues and a conjugate pair of complex eigenvalues with negative real parts).}
        \label{figure-green}
    \end{figure}
    \item[(ii)] $0<x_3 < \wt n_3 < n_3^*$: $\color{blue1}\rule{.5cm}{.3cm}$ \bl\textbf{Blue areas in Figure~\ref{figure-regimes} -- fixation of type 2}: 
    In contrast, when $0<x_3 < \wt n_3$ (and $\wt n_3<n_3^*$ since $\lambda_2<\lambda_1$ and $q>0$\color{black}), type 2 has such a strong benefit from \bl  dormancy that type 1 cannot even invade and~\eqref{6dimvirus} has no coordinatewise positive equilibrium. In this case, we expect fixation of type 2 (see Figure~\ref{figure-blue}). For $r\kappa\mu_1<v\sigma$, this regime notably extends to the area where $\lambda_2<\lambda_1$. \bl 
\end{enumerate}
    \begin{figure}[!ht]
        \centering
        \includegraphics[width=0.38\linewidth]{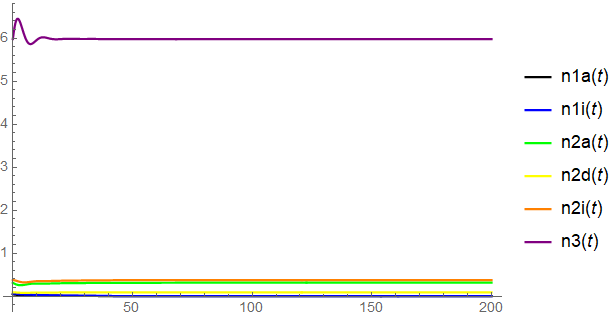} \quad
        \includegraphics[width=0.38\linewidth]{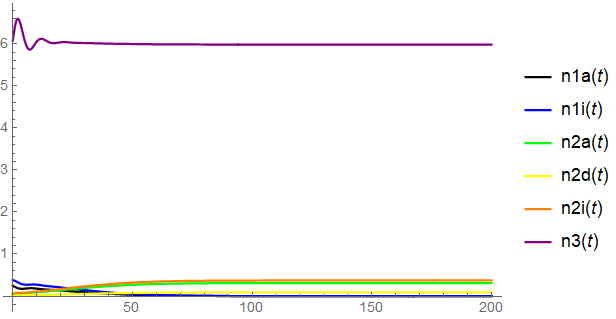} \\
         \includegraphics[width=0.38\linewidth]{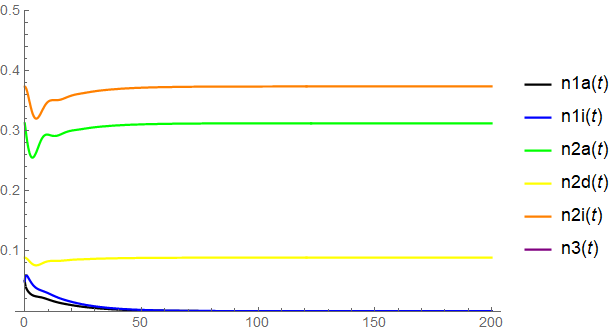} \quad
        \includegraphics[width=0.38\linewidth]{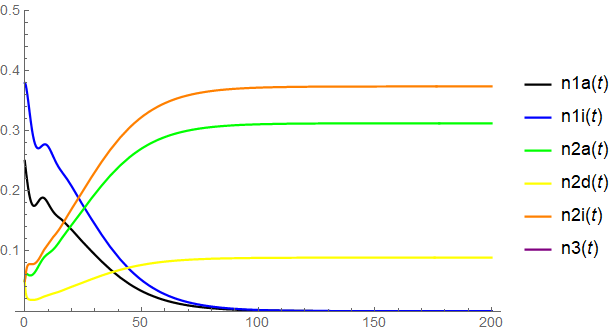}
        \caption{$\color{blue1}\rule{.5cm}{.3cm}$ Blue regime in Figure~\ref{figure-regimes} with \bl $\lambda_2=3<\lambda_1, q=0.2$: Fixation of type~2. The solutions to~\eqref{6dimvirus} now converge to $(0,0,\wt n_{2a},\wt n_{2d},\wt n_{2i},\wt n_3)$. The initial conditions are analogous to Figure~\ref{figure-green}. In this example, we have $x_3<\wt n_{3} < n_3^*$.}\label{figure-blue}
    \end{figure}
\item {$\wt n_{3}<n_3^*<x_3$: \bf The case when type 2 cannot invade type 1 coexisting with type 3}:
\begin{enumerate}
    \item[(i)] $\color{purple1}\rule{.5cm}{.3cm}$ \bl \textbf{Purple areas in Figure~\ref{figure-regimes} -- fixation of type 1}: Here, the evolutionary advantage of type 2 compared to the virus-free case is weak (in other words, $x_3>n_3^*$). Consequently, only type 1 can invade and we expect that it fixates, \bl just as in the \bl absence of the viruses/dormancy (cf.~Figure~\ref{figure-purple}).
    \begin{figure}[!ht]
        \centering
        \includegraphics[width=0.38\linewidth]{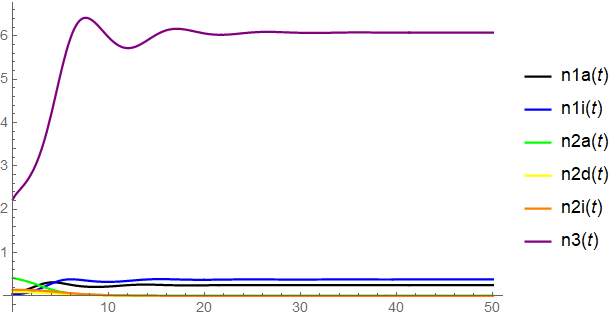} \quad
        \includegraphics[width=0.38\linewidth]{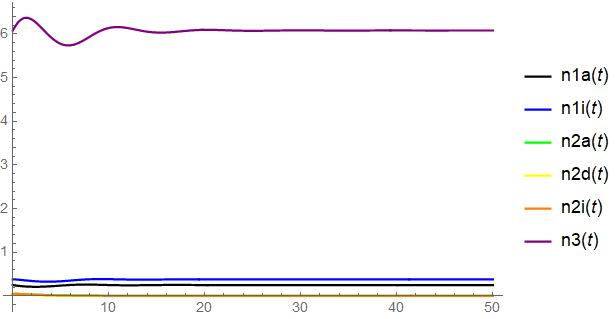} \\
        \includegraphics[width=0.38\linewidth]{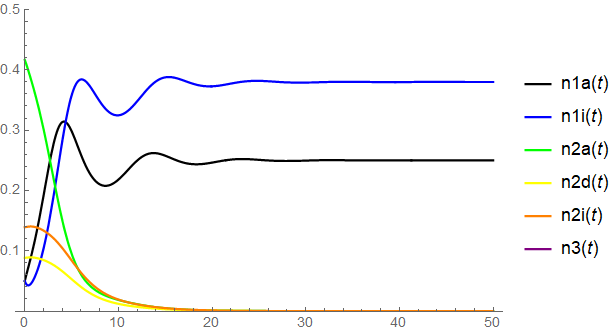} \quad
        \includegraphics[width=0.38\linewidth]{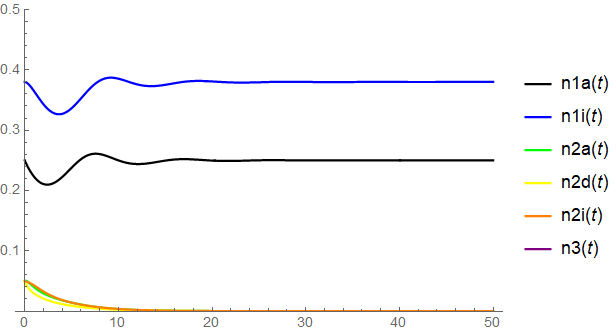}
        \caption{$\color{purple1}\rule{.5cm}{.3cm}$ Purple regime in Figure~\ref{figure-regimes}  with \bl $\lambda_2=2,q=0.4$: Fixation of type 1. The solutions to~\eqref{6dimvirus} now converge to $(n_{1a}^*,n_{1i}^*,0,0,0,n_3^*)$. The initial conditions are analogous to Figure~\ref{figure-green}. In this example, we have $\wt n_3<n_3^*<x_3$.}\label{figure-purple}
    \end{figure}
\end{enumerate}
\end{itemize}

   {\bf (D) Type 2 unable to coexist with type 3 in absence of type 1.}
\begin{itemize}
    \item {\bf The case when type 2 can invade type 1 while coexisting with type 3:}
\begin{enumerate}
    \item[(iii)] $0<x_3<n_3^*$: \color{black} $\color{lightgreen}\rule{.5cm}{.3cm}$ \bl {\bf Light green areas in Figure~\ref{figure-regimes} -- stable six-type coexistence}: Note that according to Conjecture~\ref{conj-fixation}, in the case when $\bar n_{2a} > \wt n_{2a}$ (so that $(\wt n_{2a},\wt n_{2d},\wt n_{2i},\wt n_3)$ does not exist) but the \twoIone-invasion condition\bl~\eqref{2invades1} holds, we still expect that~\eqref{6dimvirus} converges to $\mathbf x$ started from any coordinatewise positive equilibrium, and this conjecture is also supported by simulations (see Figure~\ref{figure-lightgreen}). 
    \begin{figure}[!ht]
        \centering
        \includegraphics[width=0.38\linewidth]{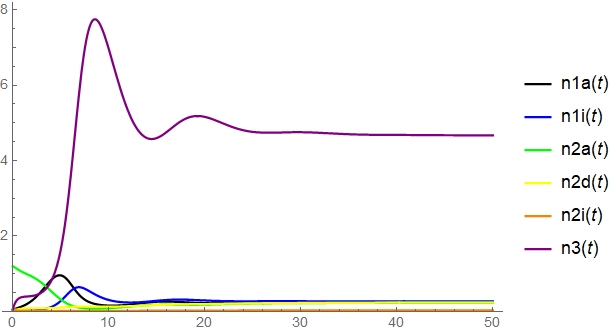} \quad
        \includegraphics[width=0.38\linewidth]{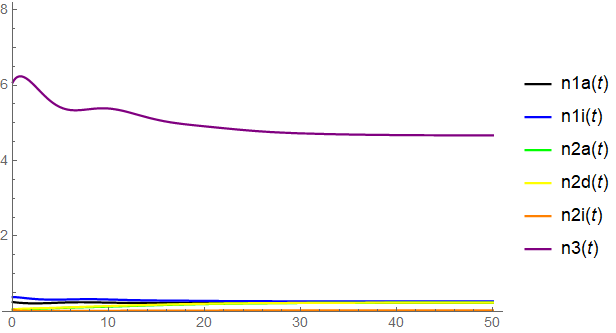} \\
        
        \includegraphics[width=0.38\linewidth]{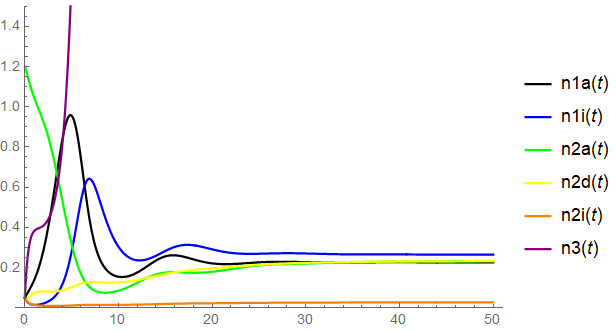} \quad
        \includegraphics[width=0.38\linewidth]{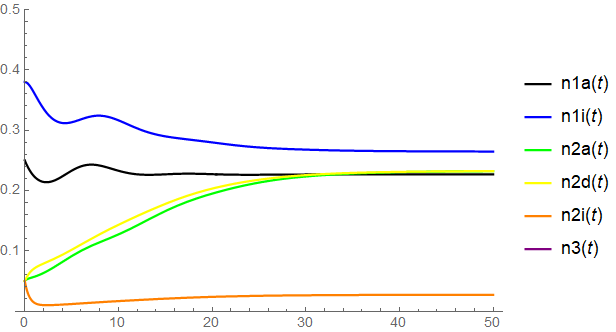}
        \caption{$\color{lightgreen}\rule{.5cm}{.3cm}$ Light green regime in Figure~\ref{figure-regimes}  with \bl $\lambda_2=2.2,q=0.9$: Stable six-type coexistence without four-type coexistence between types 2a, 2d, 2i, and 3. The equilibrium $\mathbf x$ seems to be equal to $\lim_{t\to\infty} \mathbf n(t)$ started from any coordinatewise positive initial condition $\mathbf n(0)$, in particular also close to $\mathbf n(0)$ $(0,0,\bar n_{2a},0,0,0)$ (left) and $(n_{1a}^*,n_{1i}^*,0,0,0,n_3^*)$ (right). We checked numerically that in the example presented here, $\mathbf x$ is also locally asymptotically stable (with exactly one pair of complex conjugate eigenvalues). Observe in the plots that $0<x_3<n_3^*$.}\label{figure-lightgreen}
    \end{figure}
    

    \item[(ii)] $\color{orange1} \rule{.5cm}{.3cm}$ {\bf Orange areas in Figure~\ref{figure-regimes} -- fixation of type 2a}: This regime only occurs in the generally excluded case $\lambda_2>\lambda_1$. Here, we expect that started from any coordinatewise positive initial condition, the system~\eqref{6dimvirus} will converge to $(0,0,\bar n_{2a},0,0,0)$. We expect that this complete eradication of a virus epidemic via an invasion of type 2 against type 1 coexisting with type 3 is not possible for $\lambda_2<\lambda_1$, since if all subpopulations but the type 1a and 2a  had vanished, \bl  type 1a would take over. \bl
\end{enumerate}
\end{itemize}
\begin{itemize}
    \item {\bf The case when type 2 cannot invade type 1 coexisting with type 3:}
\begin{enumerate}
    \item[(i)] $0<n_3^*<x_3$: \color{black} $\color{red1} \rule{.5cm}{.3cm}$  \textbf{Red areas in Figure~\ref{figure-regimes} -- type 2 does not coexist with type 3, fixation of type 1}: Here, type 2 cannot invade type 1, the reason being that its birth rate $\lambda_2$ is too small compared to $\lambda_1$ (even for large values of $q$). Thus, started from a coordinatewise positive initial condition, the system~\eqref{6dimvirus} quickly converges to $(n_{1a}^*,n_{1i}^*,0,0,0,n_3^*)$, see Figure~\ref{figure-red}. 
     \begin{figure}[!ht]
        \centering
        \includegraphics[width=0.38\linewidth]{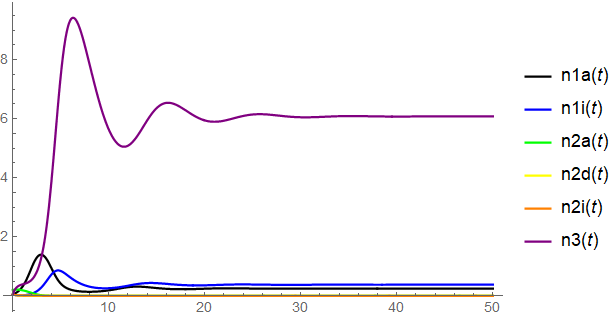} \quad
        \includegraphics[width=0.38\linewidth]{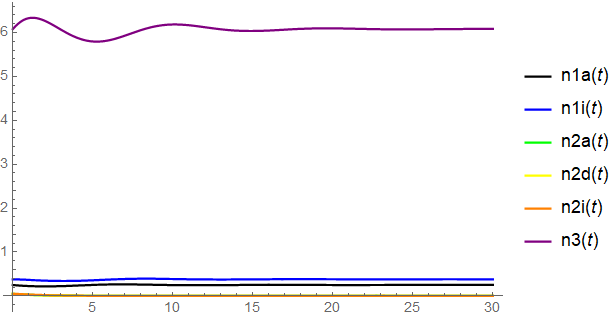} \\
        \includegraphics[width=0.38\linewidth]{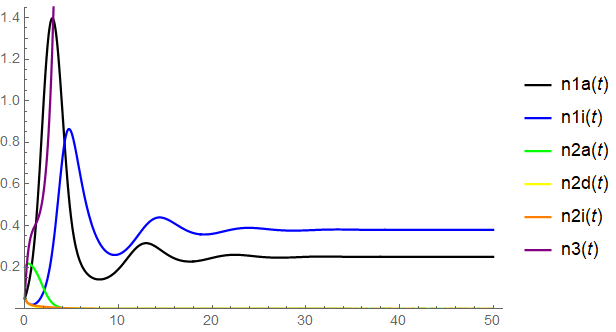} \quad
        \includegraphics[width=0.38\linewidth]{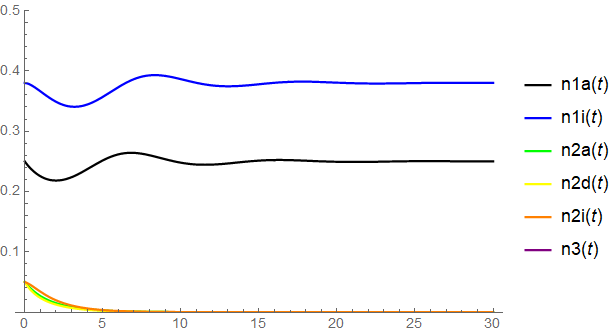}
        \caption{$\color{red1}\rule{.5cm}{.3cm}$ Red regime in Figure~\ref{figure-regimes}  with \bl $\lambda_2=1.2,q=0.4$: Fixation of type 1 (without type 2 coexisting with type 3). The solutions to~\eqref{6dimvirus} started from coordinatewise initial conditions now all converge to $(n_{1a}^*,n_{1i}^*,0,0,0,n_3^*)$. The initial conditions in the plots are close to $(0,0,\bar n_{2a},0,0,0)$ (left) resp.\ $(n_{1a}^*,n_{1i}^*,0,0,0,n_3^*)$ (right) but coordinatewise positive. In this example, we have $0<n_3^*<x_3$. 
        }\label{figure-red}
    \end{figure}
\end{enumerate}
\end{itemize}

We see in Figure~\ref{figure-regimes} that for $r\kappa\mu_1<v\sigma$,
fixation of type 2 is possible for $\lambda_2<\lambda_1$ and stable coexistence is also possible (only) in this case. Moreover, type 1 always goes extinct with asymptotically positive probability for $\lambda_2>\lambda_1$ (while type 3 stays in the system in case it coexists with type 2 in absence of type 1), just as in absence of dormancy. This shows a clear advantage of contact-mediated dormancy, which we already interpreted in Remark~\ref{remark-afterconj}. For $r\kappa\mu_1 > v \sigma$, the situation is rather different, see Appendix~\ref{appendix-regimes} for a short related discussion including Figure~\ref{figure-regimesnocoex}, the corresponding analogue of  Figure~\ref{figure-regimes}. \color{black}

\section{Discussion}\label{sec-discussion}

In this section, we discuss our model and results, and embed them into the literature. Indeed, in Section~\ref{ssn-related-literature} we review related models studied in prior work and in Section~\ref{ssn-modeling-choices} we comment on our modelling choices.   Finally, in Section~\ref{sec-bpdiscussion} we summarize the relation between the stability of equilibria of the dynamical system~\eqref{6dimvirus} and the critical behaviour of the approximating branching processes. We conclude with Section~\ref{sec-futurework} on potential future work.


\subsection{Related literature\bl}\label{ssn-related-literature}
Host dormancy as a defense strategy in host-virus systems has been described \bl in multiple experimental studies in recent years. 
For example, according to  \cite{JF19} and \cite{MNM19}, infected bacteria can enter a dormant state as part of a CRISPR-Cas immune response, thus curbing phage epidemics. Furthermore, Bautista et al \cite{B15} suggested that dormancy of hosts may even be initiated upon mere contact of virus particles with their cell hull, so that the dormant host is not susceptible to the infection anymore. They  investigated this for \bl  populations of {\em Sulfolobus islandicus} (an archeon), which may switch almost entirely into dormancy within hours after being exposed to the {\em Sulfolobus spindle-shape virus SSV9}, even in the case of a relatively small virus-to-host ratio.

Gulbudak and Weitz \cite{GW15} provide an ODE-based biophysical model including contact-mediated host dormancy for the ``early stages'' (covering a few hours) of the above host--virus dynamics. Their deterministic model  is able to reproduce the experimentally \bl observed rapid switches into dormancy for relatively small virus-to-host ratios, given that the parameter corresponding to  the \bl  dormancy initiation probability $q$ is large enough. 
However, their model is focused on a relatively short time-window of  host--virus dynamics  and does not include \bl the model of~\cite{BK98} as a special case. A similar model involving not only lytic but also chronic virus infections was studied in a paper by the same authors~\cite{GW18}. The key differences between the models of~\cite{GW15} and our previous paper~\cite{BT21} were discussed in \cite[Section 2.6.5]{BT21}. 

\subsection{Remarks on modelling choices}\label{ssn-modeling-choices}
The transitions that we include in our model in Section~\ref{sec-modeldefresultsvirus} are  certainly \bl not the only ones that one might deem plausible for a minimal model for the invasion of dormancy during a persistent epidemic. 
Our general guiding principle was to follow and extend previous models if possible, in particular those in~\cite{BK98, GW15, BT21}, so that our results can be compared to these (simpler) setups. 

In particular, we made the following assumptions: 
\begin{itemize}
    \item We allow for host death while in the dormant state, i.e.\ assume  $\kappa>0$. This seems plausible, but it turns out that when $\lambda_2<\lambda_1$, the behaviour of the system for $\kappa=0$ is mathematically not very different from the case $\kappa>0$. This was already observed in the setting of~\cite{BT21}. 
    \item In~\cite{BK98} 
    (unlike in~\cite{BT21} and in the present paper) it was assumed that infected cells cannot recover, i.e.\ $r=0$. We note that the proofs of all assertions on the dynamical systems~\eqref{6dimvirus}, \eqref{3dimvirus}, \eqref{4dimvirus} also apply for $r=0$. This way, under our standing assumption that $\lambda_2<\lambda_1$, assuming $r=0$ does not reduce the richness of possible behaviours of the system substantially. However, for $\lambda_2>\lambda_1$, founder control can only occur if $r\kappa\mu_1>v\sigma$, which requires $r,\kappa>0$. Unfortunately, extending the proof of  Theorems~\ref{theorem-2invasion} and~\ref{theorem-1invasion} to the case $r=0$ would require additional efforts because the mean matrices of the approximating branching processes introduced in Sections~\ref{sec-2invasion} and~\ref{sec-1invasion} are not irreducible.  We do not expect any substantial difficulties in this case, but the proofs would certainly become significantly longer and more technical, and so we refrain from carrying them out in this paper.
    \color{black}
    \item We follow~\cite{GW15} regarding the specifics of the competitive interactions. In particular,  we assume that infected and dormant hosts do not feel competitive pressure but they exert competitive pressure on active hosts; cf.\ also \cite[Section 2.6.5]{BT21}. This could  in principle \bl be changed, but as stated above, we aim to stay in line with previous models.
   
\end{itemize}


\subsection{Discussion of the approximating branching processes and existence/stability of the \bl equilibria}\label{sec-bpdiscussion}
We see from Proposition~\ref{prop-fixation} and Sections~\ref{sec-2invasion} and~\ref{sec-1invasion} that as long as $(n_{1a}^*,n_{1i}^*,n_3^*)$ is asymptotically stable for~\eqref{3dimvirus}, the \twoIone-branching process is supercritical if and only if $(n_{1a}^*,n_{1i}^*,0,0,0,n_3^*)$ is hyperbolically unstable, while it is subcritical whenever $(n_{1a}^*,n_{1i}^*,0,0,0,n_3^*)$ is hyperbolically asymptotically stable. Similarly, as long as $(\wt n_{2a},\wt n_{2d}, \wt n_{2i}, \wt n_3)$ is asymptotically stable for~\eqref{4dimvirus}, the \oneItwo-branching process is supercritical if and only if $(0,0,\wt n_{2a},\wt n_{2d},\wt n_{2i},\wt n_3)$ is hyperbolically unstable and it is subcritical whenever $(0,0,\wt n_{2a},\wt n_{2d},\wt n_{2i},\wt n_3)$ is hyperbolically asymptotically stable.

Given this equivalence, Theorems~\ref{theorem-2invasion} and~\ref{theorem-1invasion} yield that for $i,j \in \{ 1, 2\}$, $i \neq j$, type $i$ can invade type $j$ coexisting with type $3$ if and only if the corresponding 
branching process is supercritical,  or, equivalently, if and only if the corresponding equilibrium where the coordinates belonging to type $i$ are zero and the ones belonging to types $j$ and $3$ are in stable equilibrium is hyperbolically unstable. In particular, for $\lambda_2<\lambda_1$ we do not observe any case where both branching processes are subcritical since this would mean $n_3^* < x_3 < \wt n_3$, which would contradict Proposition~\ref{prop-coexeq}\color{black}. 
The only case where both branching processes are supercritical is when~\eqref{2invades1} and~\eqref{1invades2} both hold, and here we expect stable coexistence (cf.\ Conjectures~\ref{conj-fixation}, \ref{conj-theorem2invades1}, and~\ref{conj-theorem1invades2}). Again, it is remarkable that this is the only case where the coexistence equilibrium $(x_{1a},\ldots,x_3)$ exists (for $\lambda_2<\lambda_1$). In any other case,~\eqref{6dimvirus} has no coordinatewise positive equilibrium, and we expect fixation of one of the two host types while in coexistence with type 3. 

However, rephrasing the observations of Section~\ref{sec-lambda2lambda1}, it follows that for $\lambda_2>\lambda_1$, the two branching processes can be simultaneously subcritical, and in that case, $\mathbf x$ is also coordinatewise positive (and we conjecture that it is unstable).

\subsection{Perpectives for future work}\label{sec-futurework}

On the technical side, we do not know whether there are parameter regimes where the six-dimensional dynamical system~\eqref{6dimvirus} exhibits periodic or chaotic behaviour with all the six types being persistent, and we defer the study of such questions to future work. It seems that in case the sub-systems~\eqref{3dimvirus} and~\eqref{4dimvirus} both have limit cycles, the numerical behaviour of the six-dimensional system behaves unstable and naive simulation methods are not sufficient. Moreover, even in the case when~\eqref{3dimvirus} and~\eqref{4dimvirus} do not behave periodically, the conjectured global stability of equilibria and convergence of solutions to~\eqref{6dimvirus} towards the stable equilibria started from distant initial conditions are open. 

More generally, our model provides a first step to include contact-mediated dormancy into models from stochastic adaptive dynamics. Many other forms of dormancy are possible (e.g.\ virus latency). The model can thus be seen as a further stepping stone for the systematic inclusion of dormancy mechanisms into eco-evolutionary and epidemiological models.

\appendix

\section{Proof of Propositions~\ref{prop-coexeq} and~\ref{prop-coexeqonly1virus}}\label{sec-coexeq} 
In this section we provide the proofs of our two propositions  concerned with \bl the existence of a coordinatewise positive equilibrium of~\eqref{6dimvirus}. 
\begin{proof}[Proof of Proposition~\ref{prop-coexeq}]
Assume that~\eqref{6dimvirus} has a nonzero equilibrium $\mathbf x=(x_{1a},x_{1i},x_{2a},x_{2d},x_{2i},x_3)$. 
Dividing the last equation of~\eqref{6dimvirus} by $x_3 \neq 0$  and substituting the expressions for $n_{1i}$ and $n_{2i}$ from the second and fifth equation, \bl  we obtain
\[ -D x_{1a}-(1-q)D x_{2a} +\frac{m v}{r+v} (D x_{1a}+(1-q) D x_{2a} )- \mu_3 = 0, \]
which yields that
\[ x_{1a} + (1-q)  x_{2a} = \frac{\mu_3(r+v)}{D(mv-(r+v))}, \numberthis\label{12pos} \]
i.e.\ equation~\eqref{xasum}, in case $mv \neq r+v$.
From this it follows that the existence of a coordinatewise positive $\mathbf x$ implies that $mv>r+v$, which we  henceforth \bl assume for the rest of the proof. 
Now, considering the first equation of~\eqref{6dimvirus} divided by $x_{1a}$  and substituting the expressions for $n_{1i}$  from the second  equation\bl, we find that
\[ \lambda_1-\mu_1 - C(x_{1a}+x_{1i}+x_{2a}+x_{2d}+x_{2i} ) = D x_3 \frac{v}{r+v}. \numberthis\label{1pos} \]
Similarly, from the third equation of~\eqref{6dimvirus} divided by $x_{2a}$  and substituting the expressions for $n_{2i}$ and $n_{2d}$ from the fourth and fifth equation, \bl we obtain
\[ \lambda_2 - \mu_1 - C(x_{1a}+x_{1i}+x_{2a}+x_{2d}+x_{2i} ) = D x_3 \big( 1 - \frac{r}{r+v} (1-q) - \frac{\sigma}{\kappa\mu_1+\sigma} q \big). \numberthis\label{2pos} \]
This implies,  by subtracting \eqref{1pos} from \eqref{2pos}, \bl for $r\kappa\mu_1 \neq v \sigma$ and $q>0$, that
\[ x_3 = \frac{\lambda_2-\lambda_1}{qD}
\frac{(\kappa\mu_1+\sigma)(r+v)}{r\kappa\mu_1-v\sigma},  \]
which is~\eqref{x3def}. 
Moreover, from the second and fifth equation of~\eqref{6dimvirus}, and \eqref{12pos}, we get \bl 
\[ x_{1i}+x_{2i} = \frac{\mu_3 x_3}{mv-(r+v)},  \]
which is~\eqref{xicond}. 
Thus, 
a coordinatewise nonzero equilibrium $\mathbf x$ \color{black} exists whenever $\lambda_2-\lambda_1 \neq 0$, $q >0$, $mv - (r+v)\neq 0$ and $r\kappa\mu_1 -v \sigma \neq 0$. 

The positivity of $x_3$ additionally requires that $r\kappa\mu_1-v\sigma$ and $\lambda_2-\lambda_1$ have the same sign.
Since $(x_{1a},x_{1i},x_3)$, $(\wt n_{2a},\wt n_{2d},\wt n_{2i},\wt n_3)$, and $\mathbf x$ are coordinatewise nonzero equilibria of~\eqref{3dimvirus}, \eqref{4dimvirus}, and~\eqref{6dimvirus} respectively, \color{black}
we have
\[ \frac{x_{1i}}{x_{1a}} = \frac{D x_3}{r+v}, \qquad \frac{x_{2i}}{x_{2a}} = \frac{(1-q)D x_3}{r+v}, \qquad \frac{x_{2d}}{x_{2a}} = \frac{q D x_3}{\kappa\mu_1+\sigma}, \numberthis\label{iadproportions} \]
and thus \bl it follows that whenever $x_3>0$, $x_{1a}$ has the same (nonzero) sign as $x_{1i}$, and $x_{2a},x_{2d},x_{2i}$ all have equal (nonzero) signs. 

We see that $x_3$ is uniquely determined by the parameters. Thanks to~\eqref{xasum},~\eqref{xicond}, and~\eqref{iadproportions}, given the value of $x_{3}$, the values of  
$x_{1a},x_{1i},x_{2a},x_{2d},x_{2i}$ are determined
uniquely via a system of linear equations on $x_{1a},x_{1i},x_{2a},x_{2d},x_{2i}$ whose matrix is full rank. It follows that there is a unique coordinatewise nonzero equilibrium $\mathbf x$ (under the aforementioned conditions under which there exists such an equilibrium). \color{black}

Now, let us assume that $(n_{1a}^*,n_{1i}^*,n_3^*)$ is a coordinatewise positive equilibrium of~\eqref{3dimvirus}. Then, from the first equation of~\eqref{3dimvirus} we obtain \color{black}
\[ \lambda_1 -\mu_1 - C(n_{1a}^* + n_{1i}^* ) = D n_3^* \frac{v}{r+v}. \numberthis\label{1star} \]
Hence by~\eqref{xasum} and the characterization in the coexistence condition \oneC \ \bl of $n_{1a}^*$ in~\eqref{viruscoexcondq=0}, if $\mathbf x$ is coordinatewise positive, then we have 
\[ x_{1a} + x_{2a} = x_{1a} + (1-q) x_{2a} + q x_{2a} > n_{1a}^*,  \]
and since $x_{2d}$ is assumed positive and we have
\[ \frac{n_{1i}^*}{n_{3}^*} = \frac{\mu_3}{mv-(r+v)} = \frac{x_{1i}+x_{2i}}{x_3} \numberthis\label{1xproportions} \]
(cf.~also~\eqref{xicond}),~\eqref{1pos} and \eqref{1star} can only hold simultaneously if $x_3 < n_3^*$. This argument also implies that if $x_3>n_3^*$, then $x_{2a},x_{2d}$, and $x_{2i}$ must be negative.

Similarly, if $(\wt n_{2a}, \wt n_{2d}, \wt n_{2i}, \wt n_3)$ is a coordinatewise positive equilibrium of~\eqref{4dimvirus}, then from the first equation of~\eqref{4dimvirus} we obtain \color{black}
\[ 
\lambda_2- \mu_1 - C(\wt n_{2a} + \wt n_{2d} + \wt n_{2i} + \wt n_3) = D \wt n_3 \Big( 1-\frac{r}{r+v}(1-q) - \frac{\sigma}{\kappa\mu_1+\sigma} q \Big). 
\numberthis\label{2tilde} 
\]
Now, by~\eqref{xasum} and the characterization in the coexistence condition \twoC \ \bl  of $\widetilde n_{2a}$ from~\eqref{viruscoexcond2}, if $\mathbf x$ is coordinatewise positive, then 
\[ 
x_{1a} + x_{2a} = x_{1a} + (1-q) x_{2a} + q x_{2a} = (1-q) \wt n_{2a} + q x_{2a} < \wt n_{2a}, 
\]
where in the last step we used that $x_{2a} < \wt n_{2a}$ must be satisfied because otherwise we would have
\[ 
(1-q) x_{2a} + x_{1a} > (1-q) x_{2a} \geq (1-q) \wt n_{2a}, 
\]
a contradiction with \eqref{xasum} and the definition of $n_{1a}^*$ (cf.\ \eqref{viruscoexcondq=0}). Now, since 
\[ \frac{\wt n_{2i}}{\wt n_3} = \frac{\mu_3}{mv-(r+v)} = \frac{x_{1i}+x_{2i}}{x_3} \numberthis\label{2xproportions} \]
(where we also used~\eqref{xicond}) and
\[ \frac{\wt n_{2d}}{\wt n_{2i}} = \frac{x_{2d}}{x_{2i}} = \frac{q(r+v)}{(1-q)(\kappa\mu_1+\sigma)}, \numberthis\label{dormantproportions} \]
\eqref{2tilde} and~\eqref{2pos} can only hold simultaneously if $x_3 > \wt n_{3}$. From this argument, it also follows that if $x_3 < \wt n_{3}$, then $x_{1a}$ and $x_{1i}$ must be negative. We conclude that if $0<n_3^* < x_3 < \wt n_3$, then all coordinates of $(x_{1a},\ldots,x_3)$ but $x_3$ must be negative. However, this contradicts~\eqref{xasum}, and 
thus \bl 
implies assertion (2) of the proposition.

Now, similar arguments yield \bl that the converse is also true: If $x_3,n_3^*$, and $\wt n_{3}$ are all well-defined and positive, and $\wt n_{3}<x_3 < n_{3}^*$, then $\mathbf x$ is coordinatewise positive. Indeed, if $x_{2a}$ were non-positive, then $x_{1a}+(1-q)x_{2a}= n_{1a}^*$ would imply that $x_{1a}+x_{2a} \leq n_{1a}^*$. Now, since~\eqref{1xproportions} still holds and $x_{2d}$ has the same sign as $x_{2a}$ (and $x_{2i}$), arguing analogously to the case when we assumed that $\mathbf x$ is coordinatewise positive, we could derive that $x_3 \geq n_{3}^*$, contradicting our assumption. Also, if $x_{1a}$ were non-positive, then $x_{2a}$ would have to be positive, and hence $x_{1a} + (1-q) x_{2a} = (1-q) \wt n_{2a}$ would imply that $x_{2a} \geq \wt n_{2a}$. But from this, it would follow that $x_{1a} + x_{2a} = (1-q) \wt n_{2a} + q x_{2a} \geq \wt n_{2a}$. From this, using~\eqref{2xproportions}, \eqref{dormantproportions}, and the fact that $\wt n_{2i}>0$ (since $\wt n_{2a}>0$), we could conclude that $x_3 < \wt n_{3}$ and thus we would obtain another contradiction. Hence, the proof of (1) is finished.
\end{proof}

Based on this proof, we now also complete the proof of Proposition~\ref{prop-coexeqonly1virus}.

\begin{proof}[Proof of Proposition~\ref{prop-coexeqonly1virus}]
Assume that the coexistence condition\bl~\eqref{viruscoexcondq=0} is not satisfied but $\mathbf x=(x_{1a},x_{1i},x_{2a},x_{2d},x_{2i},x_3)$ is still coordinatewise positive. This is clearly impossible when $mv \leq r+v$ because then e.g.\ \eqref{xasum} cannot hold. Let us now assume for the rest of the proof that $mv > r+v$.
Then, \[ \bar n_{1a} = \frac{\lambda_1-\mu_1}{C} < n_{1a}^* = x_{1a} + x_{2a}. \]
It follows that the left-hand side of~\eqref{1pos} is negative, so that since~\eqref{1pos} must hold, $x_3$ cannot be positive. We conclude that in this case, \eqref{6dimvirus} has no coordinatewise positive equilibrium, as we claimed in (1).

Assume now that $r\kappa\mu_1 \neq v\sigma$. When~\eqref{virusnoncoexcond2} holds, we need $x_3 < \widetilde n_3^*$ in order to obtain coordinatewise positivity of $\mathbf x$ analogously to the case when~\eqref{viruscoexcond2} holds (cf.\ the proof of Proposition~\ref{prop-coexeq}). Since $(\wt n_{2a},\wt n_{2d},\wt n_{2i},\wt n_3)$ does not exist as a coordinatewise positive equilibrium, the fourth or fifth equation of~\eqref{6dimvirus} yield no additional condition, apart from the trivial condition that $x_3$ must be positive. This completes the proof of (2).
\end{proof}

\section{Proof of Theorem~\ref{theorem-2invasion}}\label{sec-proofdetails}

To carry out the program sketched in Section~\ref{sec-proofsketch}, for {\bf Step i)} \bl
we define, for each $\eps>0$, \bl  the stopping time
\[ R_{\eps} = \inf \big\{ t \geq 0 \colon 
\max \{ |N_{1a}^K(t) - n_{1a}^*|, |N_{1i}^K(t)-n_{1i}^*|, |N_{3}^K(t)-n_3^*| \} > \eps\big\}, \]
i.e.\ the first time when the rescaled population size of some of the resident types 1a, 1i, and 3 leaves an $\eps$-neighbourhood of the corresponding coordinate of  the equilibrium \bl  $(n_{1a}^*,n_{1i}^*,n_3^*)$. 

Recalling the initial condition $\mathbf M^*_K$ from~\eqref{M*def}, \color{black} the stopping time $T_\beta$  from Equation \bl \eqref{Tbetadef} (for $\beta>0$), the time $T_0^2$ of extinction of type 2 (i.e.\ of  sub-types \bl 2a, 2d, and 2i altogether), and the time $T_\eps^2$ when type 2 reaches a total population size of $\lfloor \eps K \rfloor$ (cf.\  Equation \bl \eqref{Tepsdef}), the first step of the sketch of the proof of the theorem (as presented in Section~\ref{sec-proofsketch}) is  verified by \bl the following lemma.

\begin{lemma}\label{lemma-2residentsstay}
Under the assumptions of Theorem~\ref{theorem-2invasion}, there exist $\eps_0>0$ and $b>0$ such that for all $\eps \in (0,\eps_0)$ we have
\[ \limsup_{K\to\infty} \P\big(R_{b\eps} \leq T_\eps^{2} \wedge T_0^2 \mid \mathbf N^K(0)=\mathbf M^*_K\big) = 0. \]
\end{lemma}


The proof of this lemma, which involves  standard \bl Freidlin--Wentzell type large-deviation arguments for stochastic processes exiting \bl from a domain, originates from~\cite{C06,C+19}. Now we provide the proof.

\begin{proof}[Proof of Lemma~\ref{lemma-2residentsstay}]
The structure of the proof of Lemma~\ref{lemma-2residentsstay} is entirely analogous to the one of the proof of \cite[Lemma 3.2]{C+19}. However, the construction of the coupling of the rescaled resident population size with a simpler process and the proof of certain claims within the proof of the lemma have to be adapted to our model.

We verify the lemma via coupling the rescaled population size $(N_{1a}^K(t),N_{1i}^K(t),N_3^K(t))$ coordinatewise with two three-type continuous-time Markov chains, $(Y^{\eps,-}_{1a}(t),Y^{\eps,-}_{1i}(t),Y^{\eps,-}_3(t))$ and $(Y^{\eps,+}_{1a}(t),Y^{\eps,+}_{1i}(t),Y^{\eps,+}_3(t))$, on time scales where the mutant population is still small compared to $K$. More precisely, the coupling relation will be
\[ Y^{\eps,-}_\ups(t) \leq N_{\ups}^K(t)\leq Y^{\eps,+}_\ups(t), \qquad \text{a.s.} \qquad \forall t \leq T_0^2 \wedge T_{\eps}^2, \quad \forall \ups \in \{ 1a, 1i, 3 \}, \numberthis\label{firstresidentcoupling1} \]
with initial condition $(Y^{\eps,i}_{1a}(t),Y^{\eps,i}_{1i}(t),Y^{\eps,i}_3(t))=(N_{1a}^K(0),N_{1i}^K(0),N_{3}^K(0))=(M^*_{1a,K},M^*_{1i,K},M^*_{3,K})$ for $i \in \{+,-\}$. 

The latter processes will also depend on $K$, but we omit the notation $K$ from their nomenclature for simplicity. In the original proof in~\cite{C+19}, the processes involved in the coupling are birth-and-death processes; this will not be the case in our situation anymore because the death of infected individuals by lysis gives rise to $m$ virus particles. However, the jump sizes in each coordinate will still be bounded between $-1$ and $m$. In order to satisfy \eqref{firstresidentcoupling1}, these processes can be chosen with the following birth and death rates: 
\begin{align}
\hspace{-5pt} (Y^{\eps,-}_{1a}(t),Y^{\eps,-}_{1i}(t),Y^{\eps,-}_3(t)) \colon ~ & \Big( \frac{i}{K}, \frac{j}{K}, \frac{k}{K} \Big) \to \Big( \frac{i+1}{K},  \frac{j}{K}, \frac{k}{K} \Big) & \text{at rate } \lambda_1 i, \label{first1} \\
                 & \Big( \frac{i}{K}, \frac{j}{K}, \frac{k}{K} \Big) \to \Big( \frac{i-1}{K},  \frac{j+1}{K}, \frac{k-1}{K} \Big) & \text{at rate } D \frac{ik}{K} , \\
                 & \Big( \frac{i}{K}, \frac{j}{K}, \frac{k}{K} \Big) \to \Big( \frac{i+1}{K},  \frac{j-1}{K}, \frac{k}{K} \Big) & \text{at rate } rj, \\
                 & \Big( \frac{i}{K}, \frac{j}{K}, \frac{k}{K} \Big) \to \Big( \frac{i}{K},  \frac{j-1}{K}, \frac{k+m}{K} \Big) & \text{at rate } vj, \\
                 & \Big( \frac{i}{K}, \frac{j}{K}, \frac{k}{K} \Big) \to \Big( \frac{i}{K},  \frac{j}{K}, \frac{k-1}{K} \Big) & \text{at rate } \mu_3 k, \label{fifth1} \\
                & \Big( \frac{i}{K}, \frac{j}{K}, \frac{k}{K} \Big) \to \Big( \frac{i-1}{K},  \frac{j}{K}, \frac{k}{K} \Big) & \text{at rate }  \big( \mu_1+C \frac{i}{K} + C \eps \big)i, \label{sixth1} \\
                & \Big( \frac{i}{K}, \frac{j}{K}, \frac{k}{K} \Big) \to \Big( \frac{i-1}{K} \vee 0,  \frac{j}{K}, \frac{k-1}{K} \Big) & \text{at rate } (1-q)D\eps i, \label{seventh1} \\
                & \Big( \frac{i}{K}, \frac{j}{K}, \frac{k}{K} \Big) \to \Big( \frac{i-m}{K} \vee 0,  \frac{j}{K}, \frac{k}{K} \Big) & \text{at rate } v\eps. \label{eighth1}
\end{align}
and
\begin{align*}
(Y^{\eps,+}_{1a}(t),Y^{\eps,+}_{1i}(t),Y^{\eps,+}_3(t)) \colon ~ & \Big( \frac{i}{K}, \frac{j}{K}, \frac{k}{K} \Big) \to \Big( \frac{i+1}{K},  \frac{j}{K}, \frac{k}{K} \Big) & \text{at rate } \lambda_1 i,  \\
                 & \Big( \frac{i}{K}, \frac{j}{K}, \frac{k}{K} \Big) \to \Big( \frac{i-1}{K},  \frac{j+1}{K}, \frac{k-1}{K} \Big) & \text{at rate } D \frac{ik}{K} , \\
                 & \Big( \frac{i}{K}, \frac{j}{K}, \frac{k}{K} \Big) \to \Big( \frac{i+1}{K},  \frac{j-1}{K}, \frac{k}{K} \Big) & \text{at rate } rj, \\
                 & \Big( \frac{i}{K}, \frac{j}{K}, \frac{k}{K} \Big) \to \Big( \frac{i}{K},  \frac{j-1}{K}, \frac{k+m}{K} \Big) & \text{at rate } vj, \\
                 & \Big( \frac{i}{K}, \frac{j}{K}, \frac{k}{K} \Big) \to \Big( \frac{i}{K},  \frac{j}{K}, \frac{k-1}{K} \Big) & \text{at rate } \mu_3 k,  \\
                & \Big( \frac{i}{K}, \frac{j}{K}, \frac{k}{K} \Big) \to \Big( \frac{i-1}{K},  \frac{j}{K}, \frac{k}{K} \Big) & \text{at rate }  \big( \mu_1+C \frac{i}{K} \big)i, \numberthis\label{sixth2} \\
                & \Big( \frac{i}{K}, \frac{j}{K}, \frac{k}{K} \Big) \to \Big( \frac{i}{K},  \frac{j+1}{K}, \frac{k}{K} \Big) & \text{at rate } (1-q)D \eps i, \numberthis\label{seventh2} \\
                & \Big( \frac{i+1}{K}, \frac{j}{K}, \frac{k}{K} \Big) \to \Big( \frac{i+m}{K},  \frac{j}{K}, \frac{k+m}{K} \Big) & \text{at rate } v\eps. \numberthis\label{eighth2}
\end{align*}
Let us explain why these processes satisfy~\eqref{firstresidentcoupling1}. It is clear that the transition rates written in \eqref{first1}--\eqref{fifth1} (as well as between~\eqref{eighth1} and~\eqref{sixth2}) do not involve the parameters $\eps$ and are equal to the corresponding rates of $\mathbf N^K(t)$ projected to the type 1a, 1i, and 3 coordinates. The transitions in~\eqref{sixth1} and \eqref{sixth2} only involve type 1a among the resident types. Here, for $Y_{1a}^{\eps,+}(t)$ we only take death by competition within type 1a into account, ignoring the death of type 1a by competition with type 2a individuals, thus making the type 1a population larger, whereas for $Y_{1a}^{\eps,-}(t)$ we add an additional killing term due to competition with type 2a at rate $C \eps i \geq C N_{2a}^K(t) i$, thus making the type 1a population smaller. 

The remaining transitions are somewhat more involved. \eqref{seventh1} and \eqref{seventh2} account for successful virus attacks against type 2a, while \eqref{eighth1} and \eqref{eighth2} for death of type 2i individuals by lysis. However, in all the four cases, transitions happen at $\eps$-dependent rates and they are artificial in the sense that they do not correspond to any transition type of $\mathbf N^K$. In~\eqref{seventh1} and~\eqref{seventh2}, in the case of $(Y^{\eps,-}_{1a}(t),Y^{\eps,-}_{1i}(t),Y^{\eps,-}_{3}(t))$ we remove a virus particle due to a successful virus attack against a type 2a individual, but since this ``surviving'' virus particle could potentially infect another type 1a individual, we artificially also remove a type 1a individual in order to make sure that $Y_{1a}^{\eps,-}(t) \leq N_{1a}^K(t)$ also holds for all $t$ under consideration.  (The case when there is no type 1a individual to remove will not be relevant in the end since we are considering the case where $Y_{1a}^{\eps,-}(t) \approx n_{1a}^*$ with high probability.) The factor of $i$ in the rate of such transitions is $(1-q)D\eps$, which is an upper bound for $(1-q)D N_{2a}^K(t)$ on the time interval under consideration. Similarly, in the case of $(Y^{\eps,+}_{1a}(t),Y^{\eps,+}_{1i}(t),Y^{\eps,+}_{3}(t))$, at the same rate we do not remove the virus particle, but we artificially add a type 1a individual that this non-removed virus could potentially infect in addition to the infections of type 1a individuals in $\mathbf N^K$, in order to ensure that $N_{1a}^K(t) \leq Y_{1a}^{\eps,+}(t)$ holds for all relevant values of $t$. Moreover, regarding~\eqref{eighth1} and \eqref{eighth2}, for $(Y^{\eps,-}_{1a}(t),Y^{\eps,-}_{1i}(t),Y^{\eps,-}_{3}(t))$ we introduce artificial transitions where type 3 individuals are not created by death by lysis events of type 2i individuals but we remove $m$ type 1a individuals, which is the number of type 1a individuals that $m$ new viruses could otherwise potentially infect in the worst case. (The case when there are less than $m$ type 1a individuals to remove will again not be relevant.) The rate $v\eps$ here is an upper bound on $v N_{2i}^K(t)$. For $(Y^{\eps,+}_{1a}(t),Y^{\eps,+}_{1i}(t),Y^{\eps,+}_{3}(t))$, the death by lysis events of type 2i individuals are replaced by transitions where $m$ new viruses are created but also $m$ type 1a individuals are added, so that the creation of the additional viruses cannot reduce the type 1a population in the end.

Let us estimate the time until which the processes $(Y^{\eps,-}_{1a}(t),Y^{\eps,-}_{1i}(t),Y^{\eps,-}_{3}(t))$ and $(Y^{\eps,+}_{1a}(t),Y^{\eps,+}_{1i}(t),Y^{\eps,+}_{3}(t))$ stay close to the value $(n_{1a}^*,n_{1i}^*,n_3^*)$. We define the stopping times
\[ R^i_{\eps}:=\inf \big\{ t \geq 0 \colon (Y_{1a}^i(t),Y_{1i}^i(t),Y_3^i(t)) \notin [n_{1a}^*-\eps, n_{1a}^*+\eps] \times [n_{1i}^*-\eps, n_{1i}^*+\eps] \times [n_{3}^*-\eps, n_{3}^*+\eps]\big\} \]
for $i \in \{ +, - \}$ and $\eps>0$. 

For large $K$, according to \cite[Theorem 2.1, p.~456]{EK}, the dynamics of $(Y^{\eps,-}_{1a}(t),Y^{\eps,-}_{1i}(t),Y^{\eps,-}_{3}(t))$ is close to the one of the unique solution $(n_{1a},n_{1i},n_3)=((n_{1a}(t),n_{1i}(t),n_{3}(t))_{t\geq 0}$ to
\[ \begin{aligned}
\dot n_{1a} & = n_{1a}(\lambda_1-\mu_1- C (n_{1a}+\eps)-Dn_3 - (1-q)D \eps) -mv\eps+ r n_{1i}, \\
\dot n_{1i} & = Dn_{1a}n_3 -(r+v)n_{1i}, \\
\dot n_3 = & -Dn_{1a}n_3 -(1-q)D n_{3} \eps + mv n_{1i}-\mu_3n_3
\end{aligned} \numberthis\label{eps3dimvirus}  \]
with initial condition $(n_{1a}^*,n_{1i}^*,n_{3}^*)$ (since this quantity equals $\lim_{K\to\infty} (M_{1a,K}^*,M_{1i,K}^*,M_{3,K}^*)$). 

The equilibria of this system of ODEs are $(0,0,0)$, $(\bar n_{1a}^{(\eps)},0,0)$, and $(n_{1a}^{*,(\eps)},n_{1i}^{*,(\eps)},n_{3}^{*,(\eps)})$, where $\lim_{\eps\downarrow 0} \bar n_{1a}^{(\eps)} = \bar n_{1a}$ and $\lim_{\eps\downarrow 0} (n_{1a}^{*,(\eps)},n_{1i}^{*,(\eps)},n_{3}^{*,(\eps)})=(n_{1a}^*,n_{1i}^*,n_3^*)$. In particular, for all $\eps>0$ small enough, $\bar n_{1a}^{(\eps)},n_{1a}^{*,(\eps)},n_{1i}^{*,(\eps)},n_{3}^{*,(\eps)}$ are all positive. Moreover, since we have assumed that $(n_{1a}^*,n_{1i}^*,n_3^*)$ is a coordinatewise positive and locally asymptotically stable equilibrium of~\eqref{3dimvirus}, which implies that $(\bar n_{1a},0,0)$ is unstable (cf.\ Section~\ref{sec-3dimvirus}), for $\eps>0$ sufficiently small the equilibrium $(\bar n_{1a}^{(\eps)},0,0)$ is unstable whereas \color{black} $(n_{1a}^{*,(\eps)},n_{1i}^{*,(\eps)},n_{3}^{*,(\eps)})$ is locally asymptotically stable. These also imply that there exists $b>0$ and $\eps_0 >0$ such that for all $0 <\eps \leq \eps_0$,
\[ \forall \ups \in \{1a,1i,3\} \colon \big|n_{\ups}^*-n_{\ups}^{*,(\eps)}\big| \leq b\eps \quad \text{ and } \quad 0 \notin [ n_{\ups}^*-b \eps,  n_{\ups}^*+b \eps]; \quad \text{ and } \quad \bar n_{1a}^{(\eps)} \notin [n_{1a}^*-b \eps, n_{1a}^*+b \eps].\]
Now, using a result about exit of jump processes from a domain by Freidlin and Wentzell \cite[Chapter 5]{FW84}, 
there exists a family (over $K$) of Markov jump processes 
$(\widehat Y^{\eps,-}_{1a}(t),\widehat Y^{\eps,-}_{1i}(t),\widehat Y^{\eps,-}_{3}(t))$
whose transition rates are positive, bounded, Lipschitz continuous, and uniformly bounded away from 0 such that for $i \in \{ +, - \}$ and
\[ \widehat R^i_{\eps}:=\inf \big\{ t \geq 0 \colon (\widehat Y_{1a}^{\eps,i}(t),\widehat Y_{1i}^{\eps,i}(t),\widehat Y_3^{\eps,i}(t)) \notin [n_{1a}^*-\eps, n_{1a}^*+\eps] \times [n_{1i}^*-\eps, n_{1i}^*+\eps] \times [n_{3}^*-\eps, n_{3}^*+\eps]\big\} \]
there exists $V>0$ such that
\[ \P(R^{-}_{b\eps}>\e^{KV}) = \P(\widehat R^{-}_{b\eps}>\e^{KV}) \underset{K \to \infty}{\longrightarrow} 1. \numberthis\label{FW1multitypevirus} \]
Using similar arguments for $(\widehat Y^{\eps,+}_{1a}(t),\widehat Y^{\eps,+}_{1i}(t),\widehat Y^{\eps,+}_{3}(t))$, we derive that for $\eps>0$, $V>0$ small enough, we have that
\[ \P(R^{+}_{b\eps}>\e^{KV})\underset{K \to \infty}{\longrightarrow} 1. \numberthis\label{FW2multitypevirus} \]
Now, on the event $\{ R_{b\eps} \leq T_0^2 \wedge T_{\eps}^2 \}$ we have $R_{b\eps} \geq R_{b\eps}^{-} \wedge R_{b\eps}^{+}$. Using \eqref{FW1multitypevirus} and \eqref{FW2multitypevirus}, we derive that
\[ \limsup_{K \to \infty} \P \big( R_{b\eps} \leq \e^{KV}, R_{b\eps} \leq T_0^2 \wedge T_{\eps}^2 \big) = 0. \]
Furthermore, by Markov's inequality,
\begin{align*}
    \P(R_{b\eps} \leq T_0^2 \wedge T_{\eps}^2) & \leq \P \big( R_{b\eps} \leq \e^{KV}, R_{b\eps} \leq T_0^2 \wedge T_{\eps}^2 \big) + \P(R_{b\eps} \wedge T_0^2 \wedge T_{\eps}^2 \geq \e^{KV} \big) \\ &  \leq \P \big( R_{b\eps} \leq \e^{KV}, R_{b\eps} \leq T_0^2 \wedge T_{\eps}^2 \big) + \e^{-KV} \E(R_{b\eps} \wedge T_0^2 \wedge T_{\eps}^2 ). 
\end{align*}
Since we have
\[ \E \big[R_{b\eps} \wedge T_0^2 \wedge T_{\eps}^2 \big] \leq \E \Big[\int_0^{R_{b\eps} \wedge T_0^2 \wedge T_{\eps}^2} K N_{2}^K(t) \d t \Big], \]
recalling that $N_{1}(t)=N_{1a}(t)+N_{1i}(t),N_{2}(t)=N_{2a}(t)+N_{2d}(t)+N_{2i}(t)$ and writing $N_1^K(t)=\frac1K N_1(t), N_2^K(t)=\frac1K N_2(t)$, it suffices to show that there exists $C>0$ such that 
\[ \E \Big[\int_0^{R_{b\eps} \wedge T_0^2 \wedge T_{\eps}^2} N_{2}(t)\d t \Big] \leq C \eps K. \numberthis\label{expectedstoppingtimemultihostvirus} \]
This can be done analogously to \cite[Section 3.1.2]{C+19}. 
Indeed, let $\Lcal$ denote the infinitesimal generator of $(\mathbf N^K(t))_{t \geq 0}$. We want to show that there exists a function $g \colon (\smfrac{1}{K} \N_0)^6 \to \R$ defined as
\[ g(n_{1a},n_{1i},n_{2a},n_{2d},n_{2i},n_3)=\gamma_{2a} n_{2a} + \gamma_{2d} n_{2d}+\gamma_{2i} n_{2i} \numberthis\label{gammasvirus} \]
such that
\[ \Lcal g(\mathbf N_{t}^K) \geq N_{2,t}^K. \numberthis\label{largegeneratormultitypevirus}\]
If \eqref{largegeneratormultitypevirus} holds, then \eqref{expectedstoppingtimemultihostvirus} follows because using Dynkin's formula,
\[ 
\begin{aligned}
\E \Big[\int_0^{R_{b\eps} \wedge T_0^2 \wedge T_{\eps}^2} K N_{2,t}^K \d t \Big] & \leq  \E \Big[\int_0^{R_{b\eps} \wedge T_0^2 \wedge T_{\eps}^2} K\Lcal g(\mathbf N_{t}^K)\d t \Big] =\E\big[ K g(\mathbf N_{R_{b\eps} \wedge T_0^2 \wedge T_{\eps}^2}^K)-Kg(\mathbf N_{0}^K) \big] \\
& \leq (\gamma_{2a} \vee \gamma_{2d} \vee \gamma_{2i}) \eps K - ( \gamma_{2a} \wedge \gamma_{2d} \wedge \gamma_{2i} ),
\end{aligned}
\]
which implies \eqref{expectedstoppingtimemultihostvirus}, independently of the signs of $\gamma_{2a},\gamma_{2d}$, and $\gamma_{2i}$. Let us apply the infinitesimal generator $\mathcal L$ to the function $g$ defined in \eqref{gammasvirus}. We obtain 
\[ \begin{aligned} \Lcal g(\mathbf N_t^K) &= N_{2a}(t) \big[ (\lambda_2-\mu_1-C (N_{1}^K(t)+N_{2}^K(t))-D N_3^K(t))\gamma_{2a} + q D N_{3}^K(t) \gamma_{2d} + (1-q) D N_{3}^K(t) \gamma_{2i} \big] \\ & \qquad + N_{2d}(t) \big[ \sigma  \gamma_{2a} -(\kappa\mu_1+\sigma)  \gamma_{2d} \big] + N_{2i}(t) \big[ r \gamma_{2a} - (r+v) \gamma_{2i} \big]. \end{aligned} \] 
Hence, according to \eqref{largegeneratormultitypevirus}, it sufficies to show that there exist $\gamma_{2a},\gamma_{2d},\gamma_{2i} \in \R$ such that the following system of inequalities is satisfied:
\begin{align}
    (\lambda_2-\mu_1-C (N_{1}^K(t)+N_{2}^K(t))-D N_{3}^K(t))\gamma_{2a} +q D N_{3}^K(t)\gamma_{2d} +(1-q)D N_{3}^K(t)\gamma_{2i}  & >1, \label{first-gammasvirus} \\
     \sigma \gamma_{2a} -(\kappa\mu_1+\sigma)\gamma_{2d} & >1, \label{second-gammasvirus} \\
     r\gamma_{2a}-(r+v)\gamma_{2i} & >1. \label{third-gammasvirus}
\end{align}
Since $N_{1}^K(t)+N_{2}^K(t)$ and $N_{3}^K(t)$ vary in $t$, the system \eqref{first-gammasvirus}--\eqref{third-gammasvirus} of inequalities is not easy to handle. However, for $t \in [0,R_{b\eps} \wedge T_0^2 \wedge T_{\eps}^2]$, we have $C (N_{1}^K(t)+N_{2}^K(t)) +DN_{3}^K(t) \leq C (n_{1a}^*+n_{1i}^*+b\eps+\eps)+D(n_3^*+b\eps)$ and $D N_{3}^K(t) \geq D (n_3^*-b\eps)$. Hence, 
\[ \begin{aligned} (\lambda_2 &-\mu_1-C (N_{1,t}^K+N_{2}^K(t))-qD N_3^K(t) (N_{1,t}^K+N_{2a}^K(t))) \gamma_{2a} + q D N_3^K(t) \gamma_{2d} + (1-q)D N_3^K(t)\gamma_{2i}  \\ & \geq (\lambda_2-\mu_1-C(n_{1a}^*+n_{1i}^*+b\eps+\eps)-D(n_3^*+b\eps))\gamma_{2a}+q D(n_3^*-b\eps) \gamma_{2d} +(1-q) D(n_3^*-b\eps) \gamma_{2i}, \end{aligned} \]
which implies that \eqref{first-gammasvirus} is satisfied as soon as 
\[  (\lambda_2-\mu_1-C(n_{1a}^*+n_{1i}^*+b\eps+\eps)-D(n_3^*+b\eps))\gamma_{2a}+q D(n_3^*-b\eps) \gamma_{2d} +(1-q) D(n_3^*-b\eps) \gamma_{2i}>1. \numberthis\label{firststronger-gammasvirus} \] 
Let us verify the existence of $\gamma_{2a},\gamma_{2d},\gamma_i$ satisfying \eqref{firststronger-gammasvirus}, \eqref{second-gammasvirus}, and \eqref{third-gammasvirus}. First of all, we can rewrite \eqref{second-gammasvirus} as
\[ \gamma_{2d} < \frac{\sigma \gamma_{2a}-1}{\kappa\mu_1+\sigma} \numberthis\label{gamma2drewrittenmultitypevirus} \]
and
\[ \gamma_{2i} < \frac{r \gamma_{2a}-1}{r+v}. \numberthis\label{gamma2irewritten} \]

Hence, let us first consider the inequality
\[  (\lambda_2-\mu_1-C(n_{1a}^*+n_{1i}^*+b\eps+\eps)-D(n_3^*+b\eps))\gamma_{2a}+q D(n_3^*-b\eps) \frac{\sigma \gamma_{2a}-1}{\kappa\mu_1+\sigma}  +(1-q) D(n_3^*-b\eps) \frac{r \gamma_{2a}-1}{r+v}  >1. \numberthis\label{withgamma2boundvirus} \]
The \twoIone-invasion condition \eqref{2invades1} is satisfied by assumption, and hence there exists $\eps>0$ such that
\[  (\lambda_2-\mu_1-C(n_{1a}^*+n_{1i}^*+b\eps+\eps)-D(n_3^*+b\eps))+q D(n_3^*-b\eps) \frac{\sigma}{\kappa\mu_1+\sigma}  +(1-q) D(n_3^*-b\eps) \frac{r}{r+v}  >0. \numberthis\label{epscontvirus} \]
Indeed, this is true for the following reason. For $\eps=0$ we have
\begin{align*}
  & (\lambda_2-\mu_1-C(n_{1a}^*+n_{1i}^*)-Dn_3^*+q Dn_3^* \frac{\sigma}{\kappa\mu_1+\sigma}  +(1-q) D n_3^* \frac{r}{r+v} \\ &= (\lambda_2-\lambda_1) + \lambda_1-\mu_1 - C(n_{1a}^*+n_{1i}^*)-Dn_3^*+Dn_3^* \frac{r}{r+v} \\ & \qquad - D n_3^* \frac{r}{r+v} + q D n_3^* \frac{\sigma}{\kappa\mu_1+\sigma} +(1-q)Dn_3^* \frac{r}{r+v}  \\
  & = \lambda_2-\lambda_1 - D n_3^* \frac{r}{r+v} + q D n_3^* \frac{\sigma}{\kappa\mu_1+\sigma} +(1-q)Dn_3^* \frac{r}{r+v} \numberthis\label{big0} \\
  & = \lambda_2-\lambda_1 + qDn_3^* \frac{v\sigma-r\kappa\mu_1}{(r+v)(\kappa\mu_1+\sigma)} > 0 \numberthis\label{invposs}
\end{align*}
where~\eqref{big0} is true thanks to~\eqref{1star} and the strict inequality in~\eqref{invposs} thanks to~\eqref{2invades1}. By continuity, there exists $\eps>0$ such that~\eqref{epscontvirus} holds.
Hence,
\[  (\lambda_2-\mu_1-C(n_{1a}^*+n_{1i}^*+b\eps+\eps)-D(n_3^*+b\eps))\gamma_{2a}+q D(n_3^*-b\eps) \frac{\sigma \gamma_{2a}}{\kappa\mu_1+\sigma}  +(1-q) D(n_3^*-b\eps) \frac{r \gamma_{2a}}{r+v}  \]
tends to infinity as $\gamma_{2a} \to \infty$, and therefore for all sufficiently large $\gamma_{2a}$ the inequalilty~\eqref{withgamma2boundvirus} is satisfied. By continuity of the function $x \mapsto D(n_3^*-b\eps)x$, this implies that for any $\gamma_{2a}$ satisfying \eqref{withgamma2boundvirus} there exists $\gamma_{2d}$ satisfying \eqref{gamma2drewrittenmultitypevirus} and $\gamma_{2i}$ satisfying \eqref{gamma2irewritten} such that \eqref{firststronger-gammasvirus} holds. We conclude the lemma.
\color{black}
\end{proof}

In view of Lemma~\ref{lemma-2residentsstay}, we can proceed with {\bf Step ii)}, which will be  covered \bl by the \bl following proposition.

\begin{prop}
\label{prop-2firstphase}
Under the assumptions of Theorem~\ref{theorem-2invasion}, there exists a constant $b>0$ and a function $f^* \colon (0,\infty) \to (0,\infty)$ such that $\lim_{\eps\downarrow 0} f^*(\eps)=0$ and
\[ 
\limsup_{K\to\infty} \Bigg| \P \Bigg( T_{\eps}^2 < T_0^2 \wedge R_{b\eps}, \Big| \frac{T_{\eps}^2}{\log K} - \frac{1}{\lambda^*} \Big| \leq f^*(\eps) \Bigg|\ \mathbf N^K(0) =\mathbf M^*_K  \Bigg) - (1-s_{2a})\Bigg| = o_\eps(1)  \numberthis\label{firstphaseinvade}
\]
and
\[  
\limsup_{K\to\infty} \Big| \P \Big( T_0^2 < T_\eps^2 \wedge R_{b\eps}  \Big|\ \mathbf N^K(0) = \mathbf M^*_K  \Big) - s_{2a} \Big| = o_\eps(1), \numberthis\label{firstphasedie} 
\]
where $o_\eps(1)$ tends to 0 as $\eps \downarrow 0$.
\end{prop}
The proof of Proposition~\ref{prop-2firstphase} is essentially analogous to the one of~\cite[Proposition 4.1]{BT19}, but the precise coupling arguments require some care. We present the proof now.

\begin{proof}[Proof of Proposition~\ref{prop-2firstphase}]

 In what follows, we consider our population process on the event
\[ A_\eps := \{ T_0^2 \wedge T_{\eps}^2 < R_{b\eps} \} \]
for sufficiently small $\eps>0$. On this event, the invasion or extinction of the mutant (i.e.\ type 2) population will happen before the rescaled resident (i.e.\ type 1 plus 3) population substantially deviates from its equilibrium size $(n_{1a}^*,n_{1i}^*,n_3^*)$. We now show how the \twoIone-branching process approximates $(N_{2a}(t),N_{2d}(t),N_{2i}(t))$ for $t \in [0, R_{b\eps} \wedge T_{\eps}^{2} \wedge T_0^{2}]$  via an explicit sandwich coupling between two enveloping processes.  To this end, one constructs \bl $(\mathbf N(t))_{t\geq 0}$  and the \twoIone-branching process $(Z_{2a}(t),Z_{2d}(t),Z_{2i}(t))_{t\geq 0}$ (whose transition rates were defined in Section~\ref{sec-2invasion}) from an identical set of driving Poisson processes (see e.g.\ \cite[Appendix B]{BT21} for details in a similar setup). Further, \bl using the same Poisson processes, one constructs the enveloping processes $(Z_{2a}^{\eps,-}(t),Z_{2d}^{\eps,-}(t),Z_{2i}^{\eps,-}(t))_{t\geq 0}$  and $(Z_{2a}^{\eps,+}(t),Z_{2d}^{\eps,+}(t),Z_{2i}^{\eps,+}(t))_{t\geq 0}$  (which again depend on $K$, but we omit that from the notation for readability) such that for all $\upsilon \in \{ 2a,2d,2i\}$ we have the ``sandwich couplings'' 
\[ 
\begin{aligned}
Z_{\upsilon}^{\eps,-}(t) \leq Z_{\upsilon}(t) \leq Z_{\upsilon}^{\eps,+}(t) \quad \text{ and } \quad 
Z_{\upsilon}^{\eps,-}(t) \leq N_{\upsilon}(t) \leq Z_{\upsilon}^{\eps,+}(t), 
\end{aligned} 
\numberthis\label{couplingmultitypevirus} 
\]
almost surely for all $t \in [0, R_{b\eps} \wedge T_{\eps}^{2} \wedge T_0^{2}]$ on the event $A_\eps$.

Concretely, this can be achieved as follows. \bl For $(k,l,m) \in \N_0^3$, the rates of  the dominated process $(Z_{2a}^{\eps,-}(t),Z_{2d}^{\eps,-}(t),Z_{2i}^{\eps,-}(t))_{t\geq 0}$ are given as \begin{small}
\[ 
(k,l,n) \to \begin{cases}
     (k+1,l,n) & \text{ at rate $k\lambda_2$ (birth of type 2a individuals)}, \\
    (k-1,l,n) & \text{ at rate $k(\mu_1 + C (n_{1a}^*+n_{1i}^*+2b\eps+\eps))$ (death of type 2a individuals)}, \\
    (k-1,l,n+1) & \text{ at rate $(1-q) D k (n_3^*-b\eps)$ (virus contact of type 2 leading to infection)}, \\
   (k-1,l+1,n) & \text{ at rate $q D k (n_3^*-b\eps)$ (virus contact of type 2 leading to dormancy)}, \\
   (k-1,l,n)&  \text{ at rate $ D k \cdot 2 b\eps$ \emph{(artificial auxiliary transition  a), \bl see below)}}, \\
   (k+1,l-1,n)& \text{ at rate $\sigma l$ (resuscitation of type 2d individuals)}, \\
   (k,l-1,n)& \text{ at rate $\kappa\mu_1 l$ (death of type 2d individuals)}, \\
   (k+1,l,n-1) &\text{ at rate $r n$ (recovery of type 2i individuals)}, \\
    (k,l,n-1) &\text{ at rate $vn$ (death of type 2i individuals via lysis)},
\end{cases}
\]
\end{small}
while the rates of  the dominating process \bl $(Z_{2a}^{\eps,+}(t),Z_{2d}^{\eps,+}(t),Z_{2i}^{\eps,+}(t))_{t\geq 0}$ are defined as
\begin{small}
\[ 
(k,l,n) \to \begin{cases}
     (k+1,l,n) & \text{ at rate $k\lambda_2$ (birth of type 2a individuals)}, \\
    (k-1,l,n) & \text{ at rate $k(\mu_1 + C (n_{1a}^*+n_{1i}^*-2b\eps))$ (death of type 2a individuals)}, \\
    (k-1,l,n+1) & \text{ at rate $(1-q) D k (n_3^*-b\eps)$ (virus contact of type 2 leading to infection)}, \\
   (k-1,l+1,n) & \text{ at rate $q D k (n_3^*-b\eps)$ (virus contact of type 2 leading to dormancy)}, \\
   (k,l,n+1)&  \text{ at rate $ (1-q)D k \cdot 2 b\eps$ \emph{(artificial auxiliary transition  b), \bl see below)}}, \\
   (k,l+1,n)&  \text{ at rate $ qD k \cdot 2 b\eps$ \emph{(artificial auxiliary transition  c), \bl see below)}}, \\
   (k+1,l-1,n)& \text{ at rate $\sigma l$ (resuscitation of type 2d individuals)}, \\
   (k,l-1,n)& \text{ at rate $\kappa\mu_1 l$ (death of type 2d individuals)}, \\
   (k+1,l,n-1) &\text{ at rate $r n$ (recovery of type 2i individuals)}, \\
    (k,l,n-1) &\text{ at rate $vn$ (death of type 2i individuals via lysis)}.
\end{cases}
\]
\end{small}
The auxiliary transition  a) corresponds to the successful or unsuccessful virus contacts of type 2, but now only the affected active individual dies and no infected or dormant individual is created (and also no virus particle gets removed, since type 3 is not included in the branching processes under consideration).  Further,  auxiliary transition  b) is similar to the successful virus contacts of type 2, but now only an infected individual is born, while the affected active individual does not get removed, 
and the same holds for the auxiliary transition  c) with the infected individual being replaced by a dormant individual.

The rationale behind the choice of rates that \bl guarantees~\eqref{couplingmultitypevirus} for all sufficiently large $K$ and small enough $\eps,b>0$ is the following. Let $t \in [0, R_{b\eps} \wedge T_\eps^2 \wedge T_0^2]$, then we have $N_{1a}^K(t) \in [n_{1a}^*-b\eps, n_{1a}^*+b\eps]$, $N_{1i}^K(t) \in [n_{1i}^*-b\eps,n_{1i}^*+b\eps]$, $N_{3}^K(t) \in [n_3^*-b\eps, n_3^*+b\eps]$, and $N_{2a}^K(t)+N_{2i}^K(t)+N_{2d}^K(t) \in [0,\eps]$. Now,  the \bl rates of $(\mathbf N(t))_{t\geq 0}$ that are linear and affect only types 2a, 2d, 2i and 3  remain \bl unchanged in $(Z_{2a}^{\eps,-}(t),Z_{2d}^{\eps,\pm}(t),Z_{2i}^{\eps,\pm}(t))_{t\geq 0}$. All the other rates of $(Z_{2a}^{\eps,-}(t),Z_{2d}^{\eps,-}(t),Z_{2i}^{\eps,-}(t))_{t\geq 0}$ resp.\ $(Z_{2a}^{\eps,+}(t),Z_{2d}^{\eps,+}(t),Z_{2i}^{\eps,+}(t))_{t\geq 0}$ are chosen in such a way that the sizes of the type 2a, 2d, and 2i sub-populations  are always smaller / larger than those for the coupled processes \bl under these bounds on the type 1a, 1i, 3, and 2a+2d+2i population sizes. 

Indeed, the rate of death by competition experienced by a type 2a individual,  given by  \bl $C\frac{N_{1}(t)+N_{2}(t)}{K}$, is  dominated \bl from above by $C(n_{1a}^*+n_{1i}^*+2b\eps + \eps)$.  Regarding \bl the virus contacts, here the situation is slightly more complicated because these remove but also create individuals of the types that we  aim \bl to bound from below by the branching process. { Here, we observe that \bl} the smallest per capita rate of loss of type 2a individuals due to successful virus contacts is $(1-q) D (n_3^*-b\eps)$ and the one due to unsuccessful ones is $q D(n_3^*-b\eps)$,  and with these rates we carry out the corresponding classical transitions. To accommodate further events, we incorporate an additional rate of successful virus contacts of $(1-q)D \cdot 2 b \eps$ and an additional rate of unsuccessful virus contacts of $qD \cdot 2 b \eps$ per type 2a individual, but for these events we only allow the death of the involved active individual without creating new a dormant or infected individual, thus again \bl guaranteeing smaller population size of types 2a, 2d, and 2i for the  dominated \bl branching process than in the original individual-based model. This ensures that the first inequality in the second chain of inequalities in~\eqref{couplingmultitypevirus} holds (given a properly chosen Poissonian construction,  in the almost sure sense\bl). Comparing the obtained rates with  those \bl of $(Z_{2a}(t),Z_{2d}(t),Z_{2i}(t))$ defined in Section~\ref{sec-2invasion}, we conclude that the first inequality of the first chain of inequalities in~\eqref{couplingmultitypevirus} is also satisfied.

Similarly, for $(Z_{2a}^{\eps,+}(t),Z_{2d}^{\eps,+}(t),Z_{2i}^{\eps,+}(t))_{t\geq 0}$, the per capita rate of successful resp.\ unsuccessful virus contacts  are reduced to \bl  $(1-q) D (n_3^*-b\eps)$ and $q D(n_3^*-b\eps)$, but at the additional rate of $(1-q)D \cdot 2 b \eps$ resp.\ $qD \cdot 2 b \eps$ per active individual,  an infected resp.\ dormant individual is created. \bl The per capita death-by-competition rate of type 2a individuals is reduced to $C(n_{1a}^*+n_{1i}^*-2b\eps)$ for this process (in particular, the competitive pressure felt by type 2a from any form of type 2  individual \bl is entirely ignored). This yields the second inequality in both chains of inequalities of~\eqref{couplingmultitypevirus}.

The rest of the proof of Proposition~\ref{prop-2firstphase} is now entirely analogous to the corresponding part of the proof of~\cite[Proposition 4.1]{BT19}, which in turn originates from the proof of~\cite[Proposition 3.1]{C+19}. 

For $\diamond \in \{ +, - \}$, let $s_{2a}^{(\eps,\diamond)}$ denote the extinction probability of the process $(Z_{2a,t}^{\eps,\diamond}, Z_{2d,t}^{\eps,\diamond},Z_{2i,t}^{\eps,\diamond})$ started from $(1,0,0)$. The extinction probability of a supercritical branching process is continuous with respect to all kinds of transitions that the mutant population in our model has. Using arguments analogous to the ones in \cite[Appendix A.3]{C+19} and the first chain of inequalities in \eqref{couplingmultitypevirus}, we have
$s_{2a}^{(\eps,+)} \leq s_{2a} \leq s_{2a}^{(\eps,-)}$ for fixed $\eps>0$ and
\[ 0 \leq \liminf_{\eps \downarrow 0} \big| s_{2a}^{(\eps,\diamond)}-s_{2a} \big| \leq \limsup_{\eps \downarrow 0} \big| s_{2a}^{(\eps,\diamond)}-s_{2a} \big| \leq \limsup_{\eps \downarrow 0} \big| s_{2a}^{(\eps,-)}-s_{2a}^{(\eps,+)} \big| = 0, \qquad  \forall \diamond \in \{ +, - \}, \numberthis\label{qineqmvirus} \]
where we recall the extinction probability $s_{2a}$ defined in \eqref{s2as2ds2i}. 

Next, we prove that the probabilities of extinction and invasion of the actual process $(N_{2a}(t),N_{2d}(t),N_{2i}(t))$ also tend to $s_{2a}$ and $1-s_{2a}$, respectively, with high probability as $K \to \infty$. We define the stopping times
\[ T_x^{(\eps,\diamond),2} := \inf \{ t >0 \colon Z^{\eps,\diamond}_{2a}(t) +Z^{\eps,\diamond}_{2d}(t)+Z^{\eps,\diamond}_{2i}(t)= \lfloor K x \rfloor \}, \qquad \diamond \in \{ +, - \}, x \in \R. \]
Thanks to the coupling in the second chain of inequalities in \eqref{couplingmultitypevirus}, which is valid on $A_\eps$, we have
\[ \P\big(T_{\eps}^{(\eps,-),2} \leq T_0^{(\eps,-),2}, A_\eps\big) \leq \P\big(T_{\eps}^{2} \leq T_0^{2},A_\eps \big) \leq  \P\big(T_{\eps}^{(\eps,+),2} \leq T_0^{(\eps,+),2}, A_\eps\big) \numberthis\label{couplednonextinctionmultitypevirus} \] 
Indeed, if a process reaches the size $K\eps$ before dying out, then the same holds for a larger process. However, the event $A_\eps$ is independent of $(Z_{2a,t}^{\eps,\diamond}, Z_{2d,t}^{\eps,\diamond},Z_{2i,t}^{\eps,\diamond})$ for both $\diamond = +$ and $\diamond = -$, and hence 
\[ \liminf_{K \to \infty} \P\big( T_{\eps}^{(\eps,-),2}\leq T_0^{(\eps,-),2}, A_\eps \big)=\liminf_{K \to \infty}   \P(A_\eps)\P\big( T_{\eps}^{(\eps,-),2} \leq T_0^{(\eps,-),2})  \geq (1-s_{2a}^{(\eps,-)})(1-o_\eps(1)) \numberthis\label{eps-LBmultitypevirus} \]
and
\[ \limsup_{K \to \infty}  \P\big( T_{\eps}^{(\eps,+),2}\leq T_0^{(\eps,+),2}, A_\eps \big) =\limsup_{K \to \infty}   \P(A_\eps) \P\big( T_{\eps}^{(\eps,+),2} \leq T_0^{(\eps,+),2})  \leq (1-s_{2a}^{(\eps,+)})(1+o_\eps(1)). \numberthis\label{eps+UBmultitypevirus} \]
Letting $K \to \infty$ in \eqref{couplednonextinctionmultitypevirus} and applying \eqref{eps-LBmultitypevirus} and \eqref{eps+UBmultitypevirus} yields that 
\[ \limsup_{K \to \infty} \big| \P(T_{\eps}^{2} \leq T_0^{2},A_\eps )-(1-s_{2a}) \big|=o_\eps(1), \]
as wanted. Equation \eqref{firstphasedie} can be derived similarly. 

It remains to show that in the case of invasion, 
the time before the total type 2 population size reaches $K \eps$ is of order $\log K/\lambda^*$, where $\lambda^*$ was defined as the largest eigenvalue of the matrix $J^*$ defined in \eqref{J*def}, which is positive under our assumptions. 

Let $\lambda^{*,(\eps,\diamond)}$, $\diamond \in \{ +, - \}$, denote the maximal eigenvalue of the mean matrix of the process $(Z_{2a,t}^{\eps,\diamond}, Z_{2d,t}^{\eps,\diamond},Z_{2i,t}^{\eps,\diamond})$. This eigenvalue is positive for all small enough $\eps>0$ and converges to $\lambda^*$ as $\eps \downarrow 0$. Hence, there exists a function $f \colon (0,\infty) \to (0,\infty)$ with $\lim_{\eps \downarrow 0} f(\eps)=0$ such that for all $\eps>0$ sufficiently small,
\[ \Big| \frac{\lambda^{*,(\eps,\diamond)}}{\lambda^*}-1\Big| \leq \frac{f(\eps)}{2}. \numberthis\label{eigenvaluesclosemultitypevirus}\]
Let us fix $\eps$ small enough such that \eqref{eigenvaluesclosemultitypevirus} holds. Then from the second chain of inequalities in \eqref{couplingmultitypevirus} we deduce that
\[ \P \Big( T^{(\eps,-),2}_{\eps} \leq T^{(\eps,-),2}_{0} \wedge \frac{\log K}{\lambda^*} (1+f(\eps)),A_\eps \Big) \leq \P \Big( T^{2}_{\eps} \leq T^{2}_{0} \wedge \frac{\log K}{\lambda^*} (1+f(\eps)),A_\eps \Big). \]
Using this together with the independence between $A_\eps$ and $(Z_{2a,t}^{\eps,\diamond}, Z_{2d,t}^{\eps,\diamond},Z_{2i,t}^{\eps,\diamond})$ and employing \cite[Section 7.5]{AN72}, we obtain for $\eps>0$ small enough (in particular such that $f(\eps)<1$)
\[
\begin{aligned}
& \liminf_{K \to \infty} \P \Big( T^{(\eps,-),2}_{\eps} \leq T^{(\eps,-),2}_{0} \wedge \frac{\log K}{\lambda^*} (1+f(\eps)),A_\eps \Big) 
\geq (1-s_{2a}^{(\eps,-)})(1-o_\eps(1)). 
\end{aligned} 
\]
This inequality follows from computations that are analogous to \cite[Section 3.1.3, first display below (3.41)]{C+19}.  
Similarly, using the second chain of inequalities in \eqref{couplingmultitypevirus}, we conclude that for all sufficiently small $\eps>0$ 
\[ \liminf_{K \to \infty} \P \Big( T^{(\eps,+),2}_{\eps} \leq T^{(\eps,+),2}_{0}, T^{(\eps,+),2}_{\eps}  \geq \frac{\log K}{\lambda^*} (1-f(\eps)),A_\eps \Big) \geq (1-s_{2a}^{(\eps,+)})(1-o_\eps(1)). \]
These together imply \eqref{firstphaseinvade}, hence the proof of the proposition is finished.
%
%
%
%
\end{proof}

We now derive the following result, which constitutes {\bf Step iii)}.

\begin{prop}\label{prop-2dynsyst}
Under the same assumptions as in Proposition~\ref{prop-2firstphase} and the additional assumption that the \twoIone\ invasion condition\bl~\eqref{2invades1} holds, there exists a constant $\beta>0$ (independent of $\eps$) \bl such that  for any function $f^*$ for which the assertion of Proposition~\ref{prop-2firstphase} holds, \color{black}
\[ 
\begin{aligned}
& \liminf_{K \to \infty} \P \big( T_\eps^2 < T_\beta < T_\eps^2 + f^*(\eps) \log K \big| T^2_{\eps} < T_0^2 \wedge R_{b\eps} \big) \geq 1-o_\eps(1).
\end{aligned}
\]
\end{prop}

\begin{proof} 
First, it is clear that for any  $b>0$ and for all sufficiently small $\eps$, if $T_\eps^2 < T_0^2 \wedge R_{b\eps}$, then $T_\beta > T_\eps^2$. Hence, it suffices to show that 
\[ 
\begin{aligned}
& \liminf_{K \to \infty} \P \big( T_\beta < T_\eps^2 + f^*(\eps) \log K \big| T^2_{\eps} < T_0^2 \wedge R_{b\eps} \big) \geq 1-o_\eps(1),
\end{aligned}
\]
which we will do now. \color{black} Consider the system of linear differential equations
\[ (\dot n_{2a}(t), \dot n_{2d}(t), \dot n_{2i}(t)) = ( n_{2a}(t),  n_{2d}(t), n_{2i}(t)) J^* .\numberthis\label{linearODE} \]
Since $J^*$ is irreducible, by the Perron--Frobenius theorem it has a coordinatewise positive left eigenvector $(\pi_{2a},\pi_{2d},\pi_{2i})$ associated with the positive eigenvalue $\lambda^*$ satisfying $\pi_{2a}+\pi_{2d}+\pi_{2i}=1$, and there are no other coordinatewise non-negative left eigenvectors whose sum of entries equals 1. The general solution to the system~\eqref{linearODE} of equations satisfies
\[ 
\begin{aligned}
    n_{2a}(t) & =  c \color{black} \pi_{2a} \e^{\lambda^* t}(1+o(1))\color{black}, \\
    n_{2d}(t) & =   c \color{black}  \pi_{2d} \e^{\lambda^* t}(1+o(1))\color{black}, \\
    n_{2i}(t) & =   c \color{black}  \pi_{2i} \e^{\lambda^* t}(1+o(1))\color{black},
\end{aligned} \numberthis\label{J^*lin}
\]
as $t\to\infty$, for a suitably chosen $c>0$ depending only on the initial condition but not on $t$\color{black}. 

Indeed, as already discussed in Section~\ref{sec-2invasion}, due to the Perron--Frobenius theorem and the irreducibility of $J^*$, under the condition~\eqref{2invades1}, $\lambda^*$ is the unique eigenvalue of $J^*$ with the largest real part, and it is a single root of the characteristic equation (thus with a one-dimensional eigensubspace). Now, it is classical that if the other two eigenvalues $\lambda',\lambda''$ are distinct \bl real numbers with left eigenvector $(\varrho_{2a},\varrho_{2d},\varrho_{2i})$ associated with $\lambda'$ and left eigenvector $(\nu_{2a},\nu_{2d},\nu_{2i})$ associated with $\lambda''$, then the form of the general solution to~\eqref{linearODE} is
\[ (n_{2a}(t),n_{2d}(t),n_{2i}(t)) = c (\pi_{2a},\pi_{2d},\pi_{2i}) e^{\lambda^* t} + c' (\varrho_{2a},\varrho_{2d},\varrho_{2i}) \e^{\lambda' t} + c'' (\nu_{2a},\nu_{2d},\nu_{2i}) \e^{\lambda'' t} \numberthis\label{linindep} \]
for $c,c',c'' \in \mathbb R$,
which satisfies~\eqref{J^*lin}. If $\lambda'$ and $\lambda''$ are equal with geometric multiplicity 2, then the form of the solution is analogous to~\eqref{linindep} (with $\lambda'=\lambda''$). Next, if they are equal with geometric multiplicity 1, then denoting the unique left eigenvector of $J^*$ associated with $\lambda'$ by $(\varrho_{2a},\varrho_{2d},\varrho_{2i})$ and the corresponding generalized left eigenvector $(\nu_{2a},\nu_{2d},\nu_{2i})$ satisfying $(\nu_{2a},\nu_{2d},\nu_{2i})J^* = \lambda '(\nu_{2a},\nu_{2d},\nu_{2i}) + (\varrho_{2a},\varrho_{2d},\varrho_{2i})$,
then the form of the general solution is
\[ (n_{2a}(t),n_{2d}(t),n_{2i}(t))=c(\pi_{2a},\pi_{2d},\pi_{2i}) e^{\lambda^* t} + c' (\varrho_{2a},\varrho_{2d},\varrho_{2i})e^{\lambda' t} + c'' (t (\varrho_{2a},\varrho_{2d},\varrho_{2i})+(\nu_{2a},\nu_{2d},\nu_{2i})) e^{\lambda' t} \]
for $c,c',c'' \in \mathbb R$. Finally, if $\lambda'$ and $\lambda''$ are complex and conjugate, i.e.\ for some $\alpha',\beta' \in \R$ we have $\lambda' = \alpha'+i\beta'$ and $\lambda''=\alpha'-i\beta'$, with associated left eigenvectors $(u_1,u_2,u_3)=(a_1+i b_1, a_2 + i b_2 , a_3 + ib_3)$ resp.\ $(v_1,v_2,v_3)=(a_1-i b_1, a_2 - i b_2, a_3 - i b_3)$, then the form of the general solution is
\[ 
\begin{aligned}
    (n_{2a}(t),n_{2d}(t),n_{2i}(t)) = c(\pi_{2a},\pi_{2d},\pi_{2i}) e^{\lambda^* t} & + c' e^{\alpha' t} (\cos(\beta' t)(a_1,a_2,a_3) -\sin(\beta' t) (b_1,b_2,b_3)) \\ & \qquad + c'' e^{\alpha' t} (\sin(\beta' t)(a_1,a_2,a_3)+\cos(\beta' t) (b_1,b_2,b_3))
\end{aligned} \]
for $c,c',c'' \in \mathbb R$. In all cases, the
solution still satisfies~\eqref{J^*lin}. 

\color{black}

If we replace $n_{1a}(t)$ by $n_{1a}^*$, $n_{1i}(t)$ by $n_{1i}^*$, and $n_{3}(t)$ by $n_3^*$, each time they occur on the right-hand side of any equation of the six-dimensional system~\eqref{6dimvirus}, then the right-hand sides corresponding to $\dot n_{2a}(t), \dot n_{2d}(t), \dot n_{2i}(t)$ agree with \bl the right-hand sides of~\eqref{J^*lin}. Let us now fix $\eta \in (0,\lambda^*)$. By continuity, we can choose $A>0$ such that as long as $|n_{1a}(t)-n_{1a}^*|<A, |n_{1i}(t)-n_{1i}^*|<A$, and $|n_{3}(t)-n_{3}^*|<A$, $n_{2a}(t),n_{2d}(t),n_{2i}(t)$  then \bl all grow exponentially at rate at least $\eta$. Moreover, any solution to~\eqref{6dimvirus} with a coordinate-wise non-negative initial condition is clearly bounded. Thus, since the right-hand sides are locally Lipschitz, each coordinate of the solution can decay at most exponentially.

Let now $\eps>0$, $b>0$ corresponding to Proposition~\ref{prop-2firstphase}, $\gamma\in (0,\eps)$ arbitrary, and consider the following assumptions on the initial condition $(n_{1a}(0),\ldots,n_3(0))$ of the system~\eqref{6dimvirus}: 
\[ 
\begin{aligned}
|n_{\ups}(0)-n_{\ups}^*| & \leq b\eps, &\forall \ups \in \{1a, 1i, 3 \}, \\
\eps - \gamma \leq n_{2a}(0)+n_{2d}(0)+n_{2i}(0) & \leq \eps + \gamma. &
\end{aligned} \numberthis\label{Edeltaepsb}
\] 
It follows from~\eqref{J^*lin} and the irreducibility of $J^*$ that for  any \bl $\beta_1>0$ and $\gamma>0$ small enough, there exists a deterministic time \bl $s_\eps \in (0,\infty)$ depending on $\eps$ such that for any initial condition $(n_{1a}(0),\ldots,n_3(0))$, satisfying~\eqref{Edeltaepsb}, there exists $t \in (0,s_\eps)$ such that $n_{2a}(t),n_{2d}(t),n_{2i}(t) > \beta_1$. Here we also used that for such initial condition, $c$ from~\eqref{J^*lin} must be positive because it is easy to see that the non-negative orthant is positively invariant under the system~\eqref{linearODE}. Let now $\beta_2$ be the infimum of all the values of $\min \{ n_{1a}(t), n_{1i}(t), n_{3}(t) \}$ over all such $t$ and all such initial conditions $(n_{1a}(0),\ldots,n_{3}(0))$, then $\beta_2>0$. Now, let $\beta := \frac12 \min \{ \beta_1,\beta_2 \}$. 

Finally, when $\eps$ is small enough, $\mathbf N^K(T_\eps^2)$ satisfies~\eqref{Edeltaepsb} on the event $\{ T_{\eps}^2 < T_0^2 \wedge R_{b\eps} \}$, which has probability $1-s_{2a} \pm o_\eps(1)$ thanks to Proposition~\ref{prop-2firstphase}. Then, \bl using the strong Markov property and 
applying~\cite[Theorem 11.2.1, p456]{EK} on the time interval $[T_\eps^2,T_\eps^2+2s_\eps]$, it follows that with conditional probability $1-o_\eps(1)$ (on the same event $\{ T_\eps^2 < T_0^2 \wedge R_{b\eps} \}$, in the limit $K\to\infty$) \color{black} there exists $t \in [0,s_\eps]$ such that $\mathbf N^K(t)$ satisfies $N^K_{\ups}(T_\eps^2+t)>\beta$ for all $\ups \in \mathcal T$. Of course, $s_\eps$ may tend to infinity as $\eps \downarrow 0$, but  $s_\eps/\log K$ tends to zero as $K\to\infty$ for fixed small $\eps$. 
Thus, conditional on the event $\{ T_\eps^2 < T_0^2 \wedge R_{b\eps} \}$ as $K\to\infty$, $T_\beta/\log K$ stays in $[T_\eps^2/\log K,T_\eps^2/\log K+f^*(\eps)]$
with (conditional) probability $1-o_\eps(1)$. \color{black}
\color{black} Thus, the assertion of Proposition~\ref{prop-2dynsyst} follows.
\end{proof}

Now, based on Propositions~\ref{prop-2firstphase} and~\ref{prop-2dynsyst}, we can complete the proof of Theorem~\ref{theorem-2invasion} very similarly to~\cite[Section 4.4]{BT21}. The main difference between our proof here and the one in~\cite{BT21} is that we know less about the qualitative behaviour of the corresponding dynamical system (in our case~\eqref{6dimvirus}), which will be made up for by the additional Proposition~\ref{prop-2dynsyst}.

\begin{proof}[Proof of Theorem~\ref{theorem-2invasion}]

Our proof strongly relies on the coupling \eqref{couplingmultitypevirus}. To be more precise, we define a Bernoulli random variable $B$ as the indicator of non-extinction
\[ 
B:= \mathds 1_{\big\{ \forall t>0, \, Z_{2a}(t)+Z_{2d}(t)+Z_{2i}(t)>0 \big\}} 
\]
of the approximating \twoIone-branching process $((Z_{2a}(t),Z_{2d}(t),Z_{2i}(t)))_{t \geq 0}$ defined in Section~\ref{sec-2invasion}, which is initially coupled between the same two branching processes $(Z^{\eps,-}_{2a}(t),Z^{\eps,-}_{2d}(t),Z^{\eps,-}_{2i}(t))_{t\geq 0}$ and $(Z^{\eps,+}_{2a}(t),Z^{\eps,+}_{2d}(t),Z^{\eps,+}_{2i}(t))_{t\geq 0}$ as $(( N_{2a}(t), N_{2d}(t),N_{2i}(t)))_{t \geq 0}$, according to \eqref{couplingmultitypevirus}. 

Let $f^*$ be a function such that Proposition~\ref{prop-2firstphase} and~\ref{prop-2dynsyst} hold for $f^*/3$ (and hence also for $f^*$). Throughout the rest of the proof, we will assume that $\eps>0$ is so small that $f^*(\eps) <1$. Further, we fix $b\geq 2$ such that Proposition~\ref{prop-2firstphase} holds for $b$.

 We define the ``quick extinction event for type 2'' by 
\[ 
\mathcal E(K,\eps):= \Big\{ \frac{T_0^2}{\log K} \leq f^*(\eps), T_0^2 < T_{\beta}, B=0 \Big\},
\]
where we recall the stopping times $T_0^2$ and $T_\beta$ from Section~\ref{sec-2invades1}. During the rest of the proof, we will assume  that $\mathbf N^K(0)=\mathbf M^*_K$ (which is assumed in Theorem~\ref{theorem-2invasion}), and we will omit this from the notation whenever this does not lead to confusion. \color{black}
We aim to show that  
\[ 
\liminf_{K \to \infty} \, \P \big(\mathcal E(K,\eps)\big) \geq s_{2a}-o_\eps(1). \numberthis\label{extinctionlowermvirus}
\]
\bl 
Further, we  define the ``invasion event for type 2'' by
\[ 
\mathcal I(K,\eps):=\Big\{ \Big| \frac{T_\beta  \wedge T_0^2}{\log K}- \frac{1}{\lambda^*}  \Big| \leq f^*(\eps), T_\beta < T_0^2, B=1 \Big\}. 
\]
for which we aim to show that \bl in case $s_{2a}<1$, 
\[ 
\liminf_{K \to \infty} \, \P\big(\mathcal I(K,\eps)\big) \geq 1-s_{2a}-o_\eps(1).\numberthis\label{survivallowermvirus}
\]

Throughout the proof, $\beta>0$ is understood  to be \bl sufficiently small; later we will explain what conditions precisely it has to satisfy (in accordance with Proposition~\ref{prop-2dynsyst}). 
The assertions~\eqref{extinctionlowermvirus} and~\eqref{survivallowermvirus} together will imply Theorem~\ref{theorem-2invasion}.

Let us start with the case of extinction of the type 2 population in the first phase of the invasion and verify \eqref{extinctionlowermvirus}. Clearly, we have 
\[ 
\P \big(\mathcal E(K,\eps)\big) \bl \geq \P \Big( \frac{T_0^2}{\log K} \leq f^*(\eps), T_0^2 < T_\beta^2, B=0, T_0^2 < T_{\eps}^2 \wedge R_{b\eps} \Big).
\]
Now, considering our initial conditions, one can choose $\beta>0$ sufficiently small (far enough away from $n_{1a}^*, n_{1i}^*$, and $n_3^*$) \bl such that for all sufficiently small $\eps>0$ we have
\[ 
T_{\eps}^2 \wedge R_{b\eps} < T_\beta, 
\]
almost surely. We now assume for the \bl whole section that $\beta$ satisfies this condition. Then,
\[ 
\P \big(\mathcal E(K,\eps)\big) \bl \geq \P \Big( \frac{T_0^2}{\log K} \leq f^*(\eps), B=0, T_0^2 < T_{\eps}^2 \wedge R_{b\eps} \Big). \numberthis\label{andisandmvirus} 
\]
Moreover, similar to the proof of Proposition~\ref{prop-2firstphase} (cf.\ the proof of \cite[Proposition 4.1]{BT19}), we obtain 
\[ 
\limsup_{K \to \infty} \,  \P \big( \{ B=0 \} \, \Delta  \, \{ T_0^2 < T_{\eps}^2 \wedge R_{b\eps}  \}  \big)=o_\eps(1), 
\numberthis\label{undefinedsymmdiffmvirus} 
\]
where $\Delta$  stands for the \bl symmetric difference, and
\[ \limsup_{K \to \infty} \,  \P \big( \{ B=0 \} \,  \Delta \,  \{ T_0^{\eps,+} < \infty \}  \big)=o_\eps(1), \]
where 
\[ 
T_0^{\eps,+} = \inf \{ t \geq 0 \colon Z^{\eps,+}_{2a}(t)+Z^{\eps,+}_{2d}(t)+Z^{\eps,+}_{2i}(t) = 0 \} 
\]
is the extinction time of the dominating branching process $(Z^{\eps,+}_{2a}(t),Z^{\eps,+}_{2d}(t),Z^{\eps,+}_{2i}(t))_{t\geq 0}$.
Together with \eqref{andisandmvirus} and the coupling~\eqref{couplingmultitypevirus}, it follows that
\begin{align*}
\liminf_{K \to \infty} \P \big(\mathcal E(K,\eps)\big) \bl  &\geq \liminf_{K\to\infty} \P \Big( \frac{
 T_0^2}{\log K} \leq f^*(\eps), B=0, T_0^2 \leq T_\eps^2 \wedge R_{b\eps} \Big) \\
 &\geq \liminf_{K\to\infty} \P \Big( \frac{
 T_0^{\eps,+}}{\log K} \leq f^*(\eps), B=0, T_0^2 \leq T_\eps^2 \wedge R_{b\eps} \Big)
 \numberthis\label{secondlinemvirus} \\ &\geq  \liminf_{K\to\infty} \P \Big( \frac{
T_0^{\eps,+}}{\log K} \leq f^*(\eps), T_0^{\eps,+} < \infty \Big) + o_\eps(1).
\end{align*}
Thus, for $\diamond \in \{ +, - \}$, employing the chain of inequalities~\eqref{qineqmvirus},
we obtain \eqref{extinctionlowermvirus}, which implies~\eqref{sublog}. 

Let us continue with the case of ``invasion of type 2'' \bl and verify \eqref{survivallowermvirus}. 
Arguing analogously to \eqref{undefinedsymmdiffmvirus}, we obtain
\[ 
\limsup_{K \to \infty} \, \P \big( \{ B=1 \} \, \Delta \, \{ T_{\eps}^2  < T_0^2 \wedge R_{b\eps}  \}  \big)=o_\eps(1). 
\] 
Thus,
\begin{align*}
\liminf_{K \to \infty} \P \big(\mathcal I(K,\eps)\big) \bl  & = \liminf_{K \to \infty}\,  \P \Big( \Big| \frac{T_\beta}{\log K} - \frac{1}{\lambda^*}  \Big| \leq f^*(\eps), T_\beta<T_0^2,  
T^2_{\eps} < T^2_0 \wedge R_{b\eps} \Big) + o_\eps(1).
\begin{comment}{For $\eps>0,\beta>0$, we introduce the set \color{red} quantifier for $\delta$? This is also a problem in the homogamy paper (and in our papers) \color{black}
\[ \begin{aligned}
\mathfrak B^1_\eps &:= [\pi_{1d}-\delta,\pi_{1d}+\delta] \times [\pi_{1i}-\delta,\pi_{1i}+\delta] \times [\pi_{2}-\delta,\pi_{2}+\delta] \times [\eps/\widehat C,\eps] \times [\barn1a-b\eps,\barn1a+b\eps] \\
\end{aligned}
\]
and the stopping time
\[
\begin{aligned}
T'_\eps:=& \inf \Big\{ t \geq 0 \colon \Big( \frac{N_{1d,t}^K}{N_{1d,t}^K+N_{1i,t}^K+N_{2,t}^K}, \frac{N_{1i,t}^K}{N_{1d,t}^K+N_{1i,t}^K+N_{2,t}^K}, \frac{N_{2,t}^K}{N_{1d,t}^K+N_{1i,t}^K+N_{2,t}^K}, \\ & \qquad N_{1d,t}^K+N_{1i,t}^K+N_{2,t}^K, N_{1a,t}^K \Big)\in \mathfrak B^1_\eps \Big\}.
\end{aligned}
\]
Informally speaking, we want to show that with high probability the process has to pass through $\Bcal^1_\eps$ in order to reach $S_\beta$. Then, thanks to the Markov property, we can estimate $T_\beta$ by estimating $T'_\eps$ and $T_{S_{\beta}}-T'_\eps$. }\end{comment}
\\ 
    \geq  \liminf_{K\to\infty} \ & \P \Big( \Big| \frac{T_\eps^2}{\log K} -\frac{1}{\lambda^*} \Big| \leq \frac{f^*(\eps)}{2}, \Big| \frac{T_\beta -T^2_\eps}{\log K} \Big| \leq \frac{f^*(\eps)}{2},  T_\beta<T_0^2 , T^2_{\eps} < T^2_0 \wedge R_{b\eps} \Big) + o_\eps(1).
    \numberthis\label{beforesetsmviruscoex}
\end{align*}
Hence, recalling that $R_{b\eps} \wedge T_\eps^2 \leq T_\beta$,  \color{black}
for $K>0$, defining
\[ M_\eps^K:=\big\{ \widehat{\mathbf n} \in [0,\infty)^6 \colon |\widehat n_{\tau}-n_\tau^*|\leq b\eps, \forall \tau \in \{ 1a,1i,3\}, \widehat n_{2a}+\widehat n_{2d}+\widehat n_{2i}=\smfrac{\lfloor \eps K \rfloor}{K} \big\}, \]
the strong Markov property applied at time $T_\eps^2$ implies 
\[ \begin{aligned}
 \liminf_{K \to \infty} \, \P \big(\mathcal I(K,\eps)\big) \bl &  \geq \liminf_{K \to \infty} \Big[ \P \Big( \Big| \frac{T^2_\eps}{\log K} -\frac{1}{\lambda^*} \Big| \leq \frac{f^*(\eps)}{2}, T^2_{\eps} < T^2_0 \wedge R_{b\eps}
  \Big) \\
    &  \qquad \qquad \qquad \times \inf_{\begin{smallmatrix} \widehat{\mathbf n} \in M_\eps^K \end{smallmatrix}} \P \Big(  \Big| \frac{T_\beta}{\log K} \, \Big| \, \leq \frac{f^*(\eps)}{2},  T_\beta < T_0^2  \Big| \mathbf N^K(0)= \widehat{\mathbf n}  \Big) \Big].\\
   \end{aligned} \numberthis\label{productformcoexmvirus} \]


It remains to show that the right-hand side of \eqref{productformcoexmvirus} is close to $1-s_{2a}$ as $K \to \infty$ if $\eps$ is small. The fact  that \bl the limes inferior of the first factor on the right-hand side of \eqref{productformcoexmvirus} is at least $1-s_{2a}-o_\eps(1)$ follows analogously to \eqref{beforesetsmviruscoex} (since Propositions~\ref{prop-2firstphase} and~\ref{prop-2dynsyst} hold not only for $f^*$ but also for $f^*/2$). The fact that the limes inferior of the second factor on the right-hand side of~\eqref{productformcoexmvirus} is at least $1-o_\eps(1)$ is a direct consequence of Proposition~\ref{prop-2dynsyst}. 
Hence, we have obtained
\[ 
\liminf_{K \to \infty} \,  \P \big(\mathcal I(K,\eps)\big) \bl \geq 1-s_{2a}-o_\eps(1), 
\]
which implies~\eqref{2invasion} and~\eqref{lambda*conv}. It was already  shown \bl in Section~\ref{sec-2invasion} that if~\eqref{viruscoexcondq=0} holds and 
$$
\lambda_2-\lambda_1 \neq \frac{qDn_3^*(r\kappa\mu_1-v\sigma)}{(r+v)(\kappa\mu_1+\sigma)},
$$ 
then $s_{2a}=1$ is equivalent to~\eqref{2doesn'tinvade1} and $s_{2a}<1$ to~\eqref{2invades1}, and thus the proof of the theorem is finished.
\end{proof}

\section{Invasion regimes for $r\kappa\mu_1>v\sigma$}\label{appendix-regimes}
For $r\kappa\mu_1 > v \sigma$, dormancy has no advantage because type 2 can never invade if $\lambda_2<\lambda_1$ (which was already the case in absence of dormancy), and for $\lambda_2>\lambda_1$ it is also possible that type 1 achieves founder control or even fixates. This can be observed in Figure~\ref{figure-regimesnocoex}, which is the analogue of Figure~\ref{figure-regimes} in the case when $r\kappa\mu_1>v\sigma$. Here, the two regimes of stable coexistence get replaced by two regimes of founder control (cf.\ Section~\ref{sec-lambda2lambda1}), which are the only regimes where $\mathbf x$ is coordinatewise positive (and presumably unstable). In the dark green regime of founder control, each host type coexists with the virus in absence of the other host type, and the asymptotic probability of a successful invasion of either host type is 0. In the light green regime we also have founder control, but there only type 1 coexists with type 3 in absence of the other host. 

\begin{figure}
    \centering
    \includegraphics[width=0.5\linewidth]{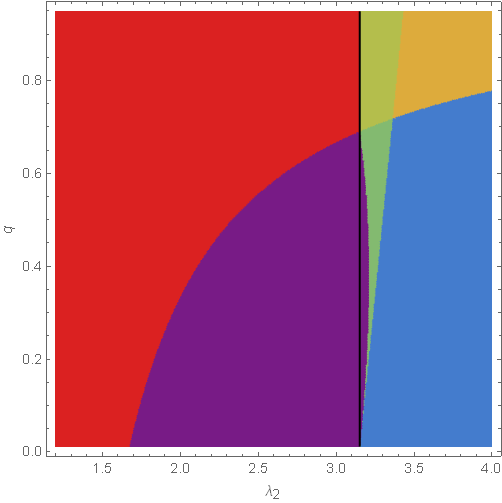}
    \caption{Here, the parameters are the same as in Figure~\ref{figure-regimes}, apart from $r$, which is increased to $3$, and $\kappa$, which is increased to $1$, so that $r\kappa\mu_1 > v \sigma$. Compared to Figure~\ref{figure-regimes}, only the following colours have a new meaning. $\color{lightgreen}\rule{.3cm}{.3cm}$ Light green (top mid/right): founder control (type 2 is not able to coexist with type 3 in absence of type 1)\color{black}, $\color{darkgreen}\rule{.3cm}{.3cm}$ dark green (bottom mid/right):  founder control (type 2 is able to coexist with type 3). The meaning of the following colours is unchanged.
    $\color{red1}\rule{.3cm}{.3cm}$ Red (left): fixation of type 1 (coex.\ with 3), $\color{orange1}\rule{.3cm}{.3cm}$ orange (top right): fixation of type 2a (without 3), $\color{purple1}\rule{.3cm}{.3cm}$ purple (bottom mid): fixation of type 1 (and 3), $\color{blue1}\rule{.3cm}{.3cm}$ blue (bottom right): fixation of type 2 (and 3). 
    \color{black}  The dark green regime reaches the black line $\lambda_2=\lambda_1$ at the $\lambda_2$ axis with a vanishing width.
    }
    \label{figure-regimesnocoex}
\end{figure}

For $\lambda_2=3.2>\lambda_1$ and $q=0.6$ (belonging to the dark green regime in Figure~\ref{figure-regimesnocoex}) and for $\lambda_2=3.2$ and $q=0.8$ (light green regime) 
we checked numerically that $\mathbf x$ is indeed unstable. In both cases, the Jacobi matrix of~\eqref{6dimvirus} at $\mathbf x$ has a pair of complex conjugate eigenvalues with negative real parts, while the remaining four eigenvalues are real, three negative and one positive. In the regimes of founder control, a conjecture analogoue to parts (C)--(D) of Conjecture~\ref{conj-fixation} would be difficult to formulate. Indeed, simulations of solutions to~\eqref{6dimvirus} indicate that here, both the domain of attraction of $(n_{1a}^*,n_{1i}^*,0,0,0,n_3^*)$ and the one of $(0,0,\wt n_{2a},\wt n_{2d},\wt n_{2i},\wt n_3)$ have a nonempty interior. We have no clear idea what the separatrix between these two domains of attraction looks like. \bl

\subsection*{Acknowledgements.} Funding acknowledgements by the second author: This paper was supported by the János Bolyai Research Scholarship of the Hungarian Academy of Sciences. Project no.\ STARTING 149835 has been implemented with the support provided by the Ministry of Culture and Innovation of Hungary from the National Research, Development and Innovation Fund, financed under the STARTING\_24 funding scheme.  The second author also wishes to thank Goethe University Frankfurt for hospitality.  \bl

\section{Hopf bifurcations in the three- and the four-dimensional dynamical system}\label{sec-Hopfwiederholung} 
Fixing all parameters of the three-dimensional system~\eqref{3dimvirus} but $m$ and considering $m$ as a bifurcation parameter, let $m^*$ denote the (unique) \color{black} value of $m$ such that $n_{1a}^*=\bar n_{1a}$. Then, at $m^*$ the system undergoes a transcritical bifurcation, i.e.\ here we have formally $(\bar n_{1a},0,0)=(n_{1a}^*,n_{1i}^*,n_3^*)$,  and \bl at this point $(\bar n_{1a},0,0)$ loses its stability, while $(n_{1a}^*,n_{1i}^*,n_3^*)$ becomes coordinatewise positive and, at least initially (for $m>m^*$ close to $m^*$), asymptotically stable.

As shown in~\cite{BK98}, when $r=0$, the coordinatewise positive equilibrium $(n_{1a}^*,n_{1i}^*,n_3^*)$ of~\eqref{3dimvirus} loses its stability via a supercritical Hopf bifurcation for some $m^{**}>m^*$. For $m>m^{**}$, this equilibrium is unstable and the asymptotic behaviour of the system~\eqref{3dimvirus} started from coordinatewise positive equilibria is periodic. By continuity, the same holds when $r$ is small compared to $v$. This loss of stability and emergence of cyclic behaviour is reminiscent of the \emph{paradox of enrichment} known from predator--prey systems (first observed by Rosenzweig \cite{R71}, see also the related discussion in \cite[Section 3.3]{BT21} and the references therein). On the other hand, we showed in \cite{BT21} that $(n_{1a}^*,n_{1i}^*,n_3^*)$ is also asymptotically stable for all sufficiently large $m$ if $r>v$, which indicates that there should be no Hopf bifurcation in this case and the equilibrium should be stable for all $m>m^*$. 

By continuity, for $q$ sufficiently small, if $r$ is small enough compared to $v$ (for example if $r=0$), the equilibrium $(\wt n_{2a},\wt n_{2d}, \wt n_{2i}, \wt n_3)$ of~\eqref{4dimvirus} may also lose its stability for large $m$ via a supercritical Hopf bifurcation. Simulation results of~\cite[Section 3.2]{BT21} however indicate that for all $q<1$ sufficiently close to 1, $(\wt n_{2a}, \wt n_{2d}, \wt n_{2i}, \wt n_{3})$ stays asymptotically stable for all values of $m$ above the corresponding transcritical bifurcation point.

Numerical results suggest (cf.\ \cite[Sections 3.1 and 3.2]{BT21}) that for $m>m^*$ very close to $m^*$, the convergence of~\eqref{3dimvirus} and~\eqref{4dimvirus} to the corresponding coordinatewise positive equilibrium is eventually coordinatewise monotone, and for somewhat larger values of $m$ the convergence becomes oscillatory and stays like that until the Hopf bifurcation point (which may be infinite). 

We note that if we choose the parameters but $\lambda_2,q,m$ as in Figure~\ref{figure-regimes} and   in \bl the corresponding simulations, then for the three-dimensional system~\eqref{3dimvirus} we are in the regime of oscillatory convergence, and the same holds for the four-dimensional system~\eqref{4dimvirus} for relatively large values of $\lambda_2$ and relatively small values of $q$, but close to the regime where there is no coexistence between types 2 and 3, the convergence becomes monotone again. Fixing these parameters, we expect no Hopf bifurcation in~\eqref{3dimvirus} because $r$ is relatively large compared to $v$. Still, for choices of parameters where the Hopf bifurcation is present in~\eqref{3dimvirus} (and thus also in~\eqref{4dimvirus} if $\lambda_2$ is close to $\lambda_1$ and $q$ to $0$), for moderately large values of $m$ (within the range of oscillatory convergence for~\eqref{3dimvirus}) we obtain a qualitatively rather similar figure.


\end{document}